\documentclass[11pt,oneside]{amsart}

\usepackage{amsmath,ifthen, amsfonts, amssymb,
srcltx, amsopn, color, enumerate} 
\usepackage[cmtip,arrow]{xy}
\usepackage{pb-diagram, pb-xy}
\usepackage{overpic}

\def\CircleArrowright{\:\ensuremath{\reflectbox{\rotatebox[origin=c]{270}{$\circlearrowright$}}}}

 \def\notorth{\:\ensuremath{\reflectbox{\rotatebox[origin=c]{90}{$\nvdash$}}}}

\newcommand{\showcomments}{no}
\newsavebox{\commentbox}

\long\def\Restate#1#2#3#4{
\medskip\par\noindent
{\bf #1 \ref{#2} #3} {\it #4}\par\medskip }

\newcounter{ax}
\newtheorem{thm}{Theorem}[section]

\newtheorem{lem}[thm]{Lemma}

\newtheorem{cor}[thm]{Corollary}

\newtheorem{prop}[thm]{Proposition}

\newtheorem{thmi}{Theorem}

\newtheorem{questioni}[thmi]{Question}

\theoremstyle{definition}
\newtheorem{defn}[thm]{Definition}
\newtheorem{rem}[thm]{Remark}
\newtheorem*{remi}{Remark}

\newtheorem{notation}[thm]{Notation}

\newtheorem{claim}{Claim}

\newtheorem{claim*}{Claim}

\newtheorem{cons}[thm]{Construction}

\DeclareMathOperator{\dimension}{dim}

\DeclareMathOperator{\image}{im}

\DeclareMathOperator{\Aut}{Aut}
\DeclareMathOperator{\Out}{Out}

\DeclareMathOperator{\stabilizer}{Stab}
\DeclareMathOperator{\diam}{\textup{\textsf{diam}}}

\DeclareMathOperator{\hull}{hull}

\newcommand{\neb}{\mathcal N}
\def\MCG{\mathcal{MCG}}

\newcommand{\field}[1]{\mathbb{#1}}
\newcommand{\integers}{\ensuremath{\field{Z}}}

\newcommand{\naturals}{\ensuremath{\field{N}}}
\newcommand{\reals}{\ensuremath{\field{R}}}

\newcommand{\hyperbolic}{\ensuremath{\field{H}}}

\newcommand{\boundary}{{\ensuremath \partial}}

\makeatletter

\newcommand{\Rmnum}[1]{\mathbf{{\expandafter\@slowromancap\romannumeral #1@}}}

\newcommand{\contact}[1]{\ensuremath{\mathcal C#1}}

\makeatother

\newcommand{\tup}[1]{\vec{#1}}

\let\oldmarginpar\marginpar
\renewcommand\marginpar[1]{\-\oldmarginpar[\raggedleft\footnotesize #1]{\raggedright\footnotesize #1}}

\newcommand{\tsh}[1]{\left\{\kern-.7ex\left\{#1\right\}\kern-.7ex\right\}}
\newcommand{\Tsh}[2]{\tsh{#2}_{#1}}
\newcommand{\ignore}[2]{\Tsh{#2}{#1}}

\newcommand{\co}{\colon}

\newcounter{enumitemp}

\newcommand{\dist}{\textup{\textsf{d}}}

\newcommand{\cuco}[1]{{\mathcal #1}}

\newcommand{\fontact}{{\mathcal C}}

\newcommand{\calC}{\mathcal C}
\newcommand{\gate}{\mathfrak g}

\newcommand{\seq}[1]{\mbox{\boldmath$#1$}}

 \usepackage{mathabx}
\newcommand{\propnest}{\sqsubsetneq}

\newcommand{\nest}{\sqsubseteq}
\newcommand{\orth}{\bot}
\newcommand{\transverse}{\pitchfork}

\newcommand{\median}{\mathfrak m}
\newcommand{\relevant}{\mathbf{Rel}}
\newcommand{\induced}{^{*}}
\newcommand{\inducedS}{^{\tiny{\diamondsuit}}}

\begin{document}
\title[HHS II: Combination theorems and the 
distance formula]
{Hierarchically hyperbolic spaces II:\\ Combination theorems and the 
distance formula}

\author[J. Behrstock]{Jason Behrstock}
\address{Barnard College, Columbia University, New York, New York, USA}
\curraddr{Lehman College and The Graduate Center, CUNY, New York, New York, USA}
\email{jason@math.columbia.edu}
\thanks{\flushleft {Behrstock was supported as a Simons Fellow.}}

\author[M.F. Hagen]{Mark F. Hagen}
\address{Dept. of Pure Maths and Math. Stat., University of Cambridge, Cambridge, UK}
\curraddr{School of Mathematics, University of Bristol, Bristol, UK}
\email{markfhagen@gmail.com}
\thanks{\flushleft {Hagen was supported by NSF Grant Number 1045119 and by EPSRC}}

\author[A. Sisto]{Alessandro Sisto}
\address{ETH, Z\"{u}rich, Switzerland}
\email{sisto@math.ethz.ch}
\thanks{\flushleft{Sisto was supported by the Swiss National Science Foundation project 144373}}

\maketitle

\begin{abstract}
We introduce a number of tools for finding and studying \emph{hierarchically
hyperbolic spaces (HHS)}, a rich class of spaces including 
mapping class groups of surfaces, Teichm\"{u}ller space with either
the Teichm\"{u}ller or Weil-Petersson metrics,  
right-angled Artin groups, and the universal cover of any 
compact special cube complex.  We begin by introducing a streamlined set
of axioms defining an HHS. We prove that all HHS satisfy a 
Masur-Minsky-style distance formula, thereby obtaining a new
proof of the distance formula in the mapping class group without
relying on the Masur-Minsky hierarchy machinery.  We then study
examples of HHS; for instance, we prove that
when $M$ is a closed irreducible $3$--manifold then 
$\pi_1M$ is an
HHS if and only if it is neither $Nil$ nor $Sol$.
We establish this by proving a general combination theorem for trees
of HHS (and graphs of HH groups). 
We also introduce a notion of
``hierarchical quasiconvexity'', which in the
study of HHS is analogous to the role played by quasiconvexity in the
study of Gromov-hyperbolic spaces. 
\end{abstract}

\tableofcontents

\renewcommand{\qedsymbol}{$\Box$}

\section*{Introduction}\label{sec:introduction}
One of the most remarkable aspects of the theory of mapping class
groups of surfaces is that the coarse geometry of the
mapping class group, $\MCG(S)$, can be fully reconstructed from its
shadows on a collection of hyperbolic
spaces --- namely the curve graphs of subsurfaces of the underlying surface. Each
subsurface of the surface $S$ is equipped with a hyperbolic curve
graph and a projection, the \emph{subsurface
projection}, to this graph from $\MCG(S)$; 
there are also projections between the various curve
graphs.  The powerful Masur--Minsky distance
formula~\cite{MasurMinsky:II} shows that the distance between points
of $\MCG(S)$ is coarsely the sum over all subsurfaces of the distances
between the projections of these points to the various curve graphs.  Meanwhile, the
consistency/realization theorem~\cite{BKMM:consistency} tells us that
tuples with coordinates in the different curve graphs that obey
``consistency'' conditions characteristic of images of actual points
in $\MCG(S)$ are, coarsely, images of points in $\MCG(S)$.  Finally,
any two points in $\MCG(S)$ are joined by a uniform-quality
quasigeodesic projecting to a uniform unparameterized quasigeodesic in
each curve graph --- a \emph{hierarchy path}~\cite{MasurMinsky:II}.

It is perhaps surprising that analogous behavior should appear in
CAT(0) cube complexes, since the mapping class group cannot act
properly on such
complexes, c.f., \cite{Bridson:semisimple,Haglund:semisimple, KapovichLeeb:actions}.  However,
mapping class groups enjoy several properties reminiscent of
nonpositively/negatively curved spaces, including: automaticity (and, thus, 
quadratic Dehn function)
\cite{Mosher:automatic}, having many quasimorphisms 
\cite{BestvinaFujiwara:boundedcohom}, super-linear divergence 
\cite{Behrstock:asymptotic}, etc. Mapping class groups also exhibit 
coarse versions of some features of CAT(0) cube
complexes, including coarse centroids/medians \cite{BehrstockMinsky:RD} and, 
more generally, a local coarse structure of a cube complex as made 
precise in  
\cite{Bowditch:coarse_median}, applications to embeddings in trees, 
\cite{BehrstockDrutuSapir:MCGsubgroups}, etc. Accordingly, it 
is natural to seek a common thread joining these important classes of
groups and spaces.

In~\cite{Hagen:quasi_arb} it was shown that, for an arbitrary CAT(0)
cube complex $\cuco X$, the intersection-graph of the hyperplane
carriers --- the \emph{contact graph} --- is hyperbolic, and in fact
quasi-isometric to a tree.  This object seems at first glance quite
different from the curve graph (which records, after all,
\emph{non}-intersection), but there are a number of reasons this is 
quite natural, two of which we now mention.  First, the curve graph can be realized as a coarse
intersection graph of product regions in $\MCG$.  Second, the contact
graph is closely related to the intersection graph of the
hyperplanes themselves; when $\cuco X$ is the universal cover of the
Salvetti complex of a right-angled Artin group, the latter graph
records commutation of conjugates of generators, just as the curve
graph records commutation of Dehn twists.

The cube complex $\cuco X$ coarsely projects to its contact graph.
Moreover, using disc diagram techniques, it is not hard to show that
any two $0$--cubes in a CAT(0) cube complex are joined by a
combinatorial geodesic projecting to a geodesic in the contact
graph~\cite{BehrstockHagenSisto:HHS_I}.  This observation --- that
CAT(0) cube complexes have ``hierarchy paths'' with very strong
properties --- motivated a search for an analogue of the theory of
curve graphs and subsurface projections in the world of CAT(0) cube
complexes.  This was largely achieved
in~\cite{BehrstockHagenSisto:HHS_I}, where a theory completely
analogous to the mapping class group theory was constructed for a wide
class of CAT(0) cube complexes, with (a variant of) the contact graph
playing the role of the curve graph.  (Results of this type for right-angled Artin groups, using the \emph{extension 
graph}, were obtained by Kim-Koberda in~\cite{KimKoberda:curve_graph}; see~\cite{BehrstockHagenSisto:HHS_I} for a 
comparison of the two approaches.)  

These results motivated us to 
define a notion of ``spaces with distance formulae'', which we did
in~\cite{BehrstockHagenSisto:HHS_I}, by introducing the class of
\emph{hierarchically hyperbolic spaces (HHS)} to provide a framework
for studying many groups and spaces which arise naturally in geometric
group theory, including mapping class groups and virtually special
groups, and to provide a notion of ``coarse nonpositive curvature''
which is quasi-isometry invariant while still yielding some of those 
properties available via local geometry in the classical setting of nonpositively-curved 
spaces.

As mentioned above, the three most salient features of hierarchically
hyperbolic spaces are: the distance formula, the realization theorem,
and the existence of hierarchy paths.  In the treatment given
in~\cite{BehrstockHagenSisto:HHS_I}, these attributes are part of the
definition of a hierarchically hyperbolic space.  This is somewhat
unsatisfactory since, in the mapping class group and cubical settings,
proving these theorems requires serious work.

In this paper, we show that although the definition of hierarchically 
hyperbolic space previously introduced identifies the right class of 
spaces, there exist a streamlined set of axioms for that class of 
spaces which are much easier to verify in practice than those 
presented in~\cite[Section 13]{BehrstockHagenSisto:HHS_I} and which 
don't require assuming a distance formula, realization
theorem, or the existence of hierarchy paths. Thus, a significant 
portion of this paper is devoted to proving that those results can be 
derived from the simplified axioms we introduce here. 
Along the way, we obtain a new, simplified proof of the actual
Masur-Minsky distance formula for the mapping class group.  We then
examine various geometric properties of hierarchically hyperbolic
spaces and groups, including many reminiscent of the
world of CAT(0) spaces and groups; for example, we show, using an argument due to Bowditch, that
hierarchically hyperbolic groups have quadratic Dehn function.
Finally, taking advantage of the simpler set of axioms, we prove combination theorems enabling the construction of
new hierarchically hyperbolic spaces/groups from old.

The definition of a hierarchically hyperbolic space still has several
parts, the details of which we postpone to Section~\ref{sec:background}.
However, the idea is straightforward: a hierarchically hyperbolic
space is a pair $(\cuco X,\mathfrak S)$, where $\cuco X$ is a metric
space and $\mathfrak S$ indexes a set of $\delta$--hyperbolic spaces 
with several features (to each $U\in\mathfrak S$ the associated space 
is denoted $\fontact U$).  Most notably,
$\mathfrak S$ is endowed with $3$ mutually exclusive relations,
\emph{nesting, orthogonality}, and \emph{transversality}, respectively
generalizing nesting, disjointness, and overlapping of subsurfaces.
For each $U\in\mathfrak S$, we have a coarsely Lipschitz projection
$\pi_U\co \cuco X\to\fontact U$, and there are relative projections
$\fontact U\to\fontact V$ when $U,V\in\mathfrak S$ are non-orthogonal.
These projections are required to obey ``consistency'' conditions
modeled on the inequality identified by Behrstock in \cite{Behrstock:asymptotic}, as well as
a version of the bounded geodesic image theorem and large link lemma
of~\cite{MasurMinsky:II}, among other conditions.  A
finitely generated group $G$ is \emph{hierarchically hyperbolic} if it
can be realized as a group of HHS automorphisms (``hieromorphisms'',
as defined in Section~\ref{sec:background}) so that the
induced action on $\cuco X$ by uniform quasi-isometries is geometric
and the action on $\mathfrak S$ is cofinite.  Hierarchically
hyperbolic groups, endowed with word-metrics, are hierarchically
hyperbolic spaces, but the converse does not appear to be true.

\subsection*{Combination theorems}
One of the main contributions in this paper is to provide many new 
examples of hierarchically 
hyperbolic groups, thus showing that mapping 
class groups and various cubical complexes/groups are just two of many  
interesting families in this class of groups and spaces. 
We provide a number of combination theorems, which we 
will describe below. One consequence of these results 
is the following classification of exactly which $3$--manifold 
groups are hierarchically hyperbolic: 

\Restate{Theorem}{thm:3mflds} {($3$--manifolds are hierarchically 
hyperbolic).}
{Let $M$ be a closed $3$--manifold. 
    Then $\pi_{1}(M)$ is a hierarchically hyperbolic space if and 
    only if $M$ does not have a Sol or Nil component in its prime decomposition.}

This result has a number of applications to 
the many fundamental groups of $3$--manifolds which are HHS. For 
instance, in such cases, it follows from  results in 
\cite{BehrstockHagenSisto:HHS_I} that: except for 
$\integers^{3}$, the top dimension of a quasi-flat in such a group 
is $2$, and any such quasi-flat is locally close to a 
``standard flat'' (this generalizes one of the main results of 
\cite[Theorem~4.10]{KapovichLeeb:haken});  up to finite index, 
$\integers$ and $\integers^{2}$ are the 
only finitely generated nilpotent groups which admit quasi-isometric 
embeddings into $\pi_{1}(M)$; and, except in the degenerate case where 
$\pi_{1}(M)$ is virtually abelian, such groups are all acylindrically 
hyperbolic (as also shown in \cite{MinasyanOsin:acylindrical}).

\begin{remi}[Hierarchically hyperbolic spaces vs.\ hierarchically hyperbolic groups]
There is an important distinction to be made between a
\emph{hierarchically hyperbolic space}, which is a metric space $\cuco
X$ equipped with a collection $\mathfrak S$ of hyperbolic spaces with
certain properties, and a \emph{hierarchically hyperbolic group},
which is a group acting geometrically on a hierarchically hyperbolic
space in such a way that the induced action on $\mathfrak S$ is
cofinite.  The latter property is considerably stronger.  For example,
Theorem~\ref{thm:3mflds} shows that $\pi_1M$, with any word-metric, is
a hierarchically hyperbolic space, but, as we discuss in
Remark~\ref{rem:graph_HHS_cube}, $\pi_1M$ probably fails to be a
hierarchically hyperbolic group in general; for instance we conjecture  
this is the case for those graph manifolds which can not be cocompactly 
cubulated.
\end{remi}

In the course of proving Theorem~\ref{thm:3mflds}, we establish several 
general combination theorems, including one about relative hyperbolicity and one
about graphs of groups.  The first is:

\Restate{Theorem}{thm:rel_hyp}{(Hyperbolicity relative to HHGs).}
{Let the group $G$ be hyperbolic relative to a finite collection
$\mathcal P$ of peripheral subgroups.  If each $P\in\mathcal
P$ is a hierarchically hyperbolic space, then $G$ is a hierarchically
hyperbolic space. Further, if each $P\in\mathcal
P$ is a hierarchically hyperbolic group, then so is $G$.}

Another of our main results is a combination theorem, Theorem~\ref{thm:combination}, establishing when 
a tree of hierarchically hyperbolic spaces is again a hierarchically
hyperbolic space.  In the statement below, \emph{hierarchical quasiconvexity} is a natural generalization of both quasiconvexity in the hyperbolic setting and cubical convexity in the cubical setting, which we shall discuss in some detail shortly.  The remaining conditions are technical and explained in Section~\ref{sec:combination}, but are easily verified in practice.

\Restate{Theorem}{thm:combination}{(Combination theorem for HHS).}
{Let $\mathcal T$ be a tree of hierarchically hyperbolic spaces.  Suppose that:
 \begin{itemize}
  \item edge-spaces are uniformly \emph{hierarchically quasiconvex} in incident vertex spaces;
  \item each edge-map is \emph{full};
  \item $\mathcal T$ has \emph{bounded supports};
  \item if $e$ is an edge of $T$ and $S_e$ is 
  the $\nest$--maximal element of $\mathfrak S_e$, then for all
  $V\in\mathfrak S_{e^\pm}$, the elements $V$ and
  $\phi_{e^\pm}\inducedS (S_e)$ are not orthogonal in $\mathfrak
  S_{e^\pm}$.
 \end{itemize}
Then $\cuco X(\mathcal T)$ is hierarchically hyperbolic.}

As a consequence, we obtain a set of sufficient conditions guaranteeing that a graph of hierarchically hyperbolic groups is a hierarchically hyperbolic group.

\Restate{Corollary}{cor:combination_theorem_for_HHG}{(Combination 
theorem for HHG).}
{Let $\mathcal G=(\Gamma,\{G_v\},\{G_e\},\{\phi_{e}^\pm\})$ be a finite graph of hierarchically hyperbolic groups.  Suppose that $\mathcal G$ equivariantly satisfies the hypotheses of Theorem~\ref{thm:combination}.  Then the total group $G$ of $\mathcal G$ is a hierarchically hyperbolic group.}

Finally, we prove that products of hierarchically hyperbolic spaces
admit natural hierarchically hyperbolic structures.

As mentioned earlier, we will apply the combination theorems to fundamental groups of $3$--manifolds, but their applicability is broader. For example, they can be applied to fundamental groups of higher dimensional manifolds such as the ones considered in \cite{FrigerioLafontSisto:rigidity}.

\subsection*{The distance formula and realization.}
As defined in \cite{BehrstockHagenSisto:HHS_I}, the 
basic definition of a \emph{hierarchically hyperbolic space} is 
modeled on the essential properties underlying the ``hierarchy machinery'' of 
mapping class groups. In this paper, we revisit the basic definition 
and provide a new, refined set of axioms; the main changes are 
the removal of the ``distance formula'' and ``hierarchy path'' axioms and the replacement of the 
``realization'' axiom by a far simpler ``partial realization''.  These new 
axioms are both more fundamental and more readily verified. 

An important result in mapping class groups which provides a starting 
point for much recent research in the field is the celebrated 
``distance formula'' of Masur--Minsky~\cite{MasurMinsky:II} 
which provides a way to 
estimate distances in the mapping class group, up to uniformly 
bounded additive and multiplicative distortion, via distances in the 
curve graphs of subsurfaces. We give a new, elementary, proof of the 
distance formula in the mapping class group. The first step in doing 
so is verifying that mapping class groups satisfy the new axioms of a hierarchically 
hyperbolic space. We provide elementary, simple proofs of the axioms for which elementary proofs do not exist in the literature (most notably, the uniqueness axiom); this is done in 
Section~\ref{sec:MCG_HHS}. This then combines with our proof of the following result which states that any hierarchically 
hyperbolic space satisfies a ``distance formula'' (which in 
the case of the mapping class group provides a new proof of the Masur--Minsky distance formula):

\Restate{Theorem}{thm:distance_formula}{(Distance formula for HHS).}
{Let $(X,\mathfrak S)$ be hierarchically hyperbolic. Then there exists $s_0$ such that for all $s\geq s_0$ there exist constants $K,C$ such that for all $x,y\in\cuco X$, $$\dist_{\cuco X}(x,y)\asymp_{(K,C)}\sum_{W\in\mathfrak S}\ignore{\dist_{ W}(x,y)}{s},$$ where $\dist_W(x,y)$ denotes the distance in the hyperbolic space $\fontact W$ between the projections of $x,y$ and $\ignore{A}{B}=A$ if $A\geq B$ and $0$ otherwise.}

Moreover, we show in Theorem~\ref{thm:monotone_hierarchy_paths} that
any two points in $\cuco X$ are joined by a uniform quasigeodesic
$\gamma$ projecting to a uniform unparameterized quasigeodesic in
$\fontact U$ for each $U\in\mathfrak S$.  The existence of such
\emph{hierarchy paths} was hypothesized as 
part of the definition of a hierarchically
hyperbolic space in~\cite{BehrstockHagenSisto:HHS_I}, but now it is 
proven as a consequence of the other axioms.

The Realization Theorem for the mapping class group was established 
by Behrstock--Kleiner--Minsky--Mosher 
in \cite[Theorem~4.3]{BKMM:consistency}. This theorem 
states that given a surface $S$ and,  
for each subsurface $W\subseteq S$, a point in the 
curve complex of $W$: this sequence of points arises as the 
projection of a point in the mapping class group (up to bounded 
error), whenever the curve complex elements satisfy certain pairwise  
``consistency conditions.''  Thus the Realization Theorem provides another sense in which all of the quasi-isometry invariant geometry of the mapping class group is encoded by the projections onto the curve graphs of subsurfaces.\footnote{In \cite{BKMM:consistency},  
the name Consistency Theorem is used to refer to the necessary and 
sufficient conditions for realization; since we find it useful to 
break up these two aspects, we refer to this half as the 
Realization Theorem, since anything that satisfies the consistency 
conditions is realized.} 
In Section~\ref{sec:realization} we show 
that an arbitrary hierarchically hyperbolic space satisfies a 
realization theorem.  Given our elementary proof of the new axioms for mapping class 
groups in Section~\ref{sec:MCG_HHS}, we thus obtain a new proof 
of \cite[Theorem~4.3]{BKMM:consistency}. 

\subsection*{Hulls and the coarse median property}
Bowditch introduced a notion of \emph{coarse median space} to 
generalize some results about median spaces to a more general 
setting, and, in particular, to the mapping class group 
\cite{Bowditch:coarse_median}. Bowditch observed in 
\cite{Bowditch:large_scale} that any hierarchically hyperbolic space 
is a coarse median space; for completeness we provide a short proof 
of this result in Theorem~\ref{thm:hier_hyp_coarse_median}. Using 
Bowditch's results about coarse median spaces, we obtain a number of 
applications as corollaries. For instance, 
Corollary~\ref{cor:rapid_decay} is obtained from 
\cite[Theorem~9.1]{Bowditch:embeddings} and says that any 
hierarchically hyperbolic space satisfies the Rapid Decay 
Property and Corollary~\ref{cor:HHG_quadratic_Dehn_function} is 
obtained from \cite[Corollary~8.3]{Bowditch:coarse_median} to show 
that all hierarchically hyperbolic groups are finitely presented and have quadratic Dehn 
functions. This provides examples of groups that are not hierarchically hyperbolic, for example:

\Restate{Corollary}{cor:out_fn}{($\Out(F_n)$ is not an HHG).}
{For $n\geq 3$, the group $\Out(F_n)$ is not a hierarchically hyperbolic group.}

Indeed, $\Out(F_n)$ was shown
in~\cite{BridsonVogtmann:dehn_1,HandelMosher:dehn,BridsonVogtmann:dehn_2}
to have exponential Dehn function.  This result is interesting as a 
counter-point to 
the well-known and fairly robust analogy between $\Out(F_n)$ and
the mapping class group of a surface; especially in light of the fact 
that $\Out(F_n)$ is known to have a number of properties reminiscent of the axioms for an 
HHS, c.f., \cite{BestvinaFeighn:projections, BestvinaFeighn:hyperbolicCFF, 
HandelMosher:hyperbolicity, SabalkaSavchuk:projections}.

The coarse median property, via work of Bowditch, also implies that asymptotic cones of hierarchically hyperbolic spaces are contractible.  Moreover, in Corollary~\ref{cor:homologydimension}, we bound the homological dimension of any asymptotic cone of a hierarchically hyperbolic space.  This latter result relies on the use of \emph{hulls} of finite sets of points in the HHS $\cuco X$.  This construction generalizes the \emph{$\Sigma$--hull} of a finite set, constructed in the mapping class group context in~\cite{BKMM:consistency}.  (It also generalizes a special case of the ordinary combinatorial convex hull in a CAT(0) cube complex.)  A key feature of these hulls is that they are coarse retracts of $\cuco X$ (see Proposition~\ref{prop:coarse_retract}), and this plays an important role in the proof of the distance formula.

\subsection*{Hierarchical spaces}\label{subsec:intro_hier_space}
We also introduce the more general notion of a \emph{hierarchical space (HS)}.  
This is the same as a hierarchically hyperbolic space, except that we do not 
require the various associated spaces $\fontact U$, onto which we are 
projecting, to be hyperbolic.  Although we mostly focus on HHS in this paper, a 
few things are worth noting.  First, the realization theorem 
(Theorem~\ref{thm:realization}) actually makes no use of hyperbolicity of the 
$\fontact U$, and therefore holds in the more general context of HS; see 
Section~\ref{sec:realization}.

Second, an important subclass of the class of HS is the class of \emph{relatively hierarchically hyperbolic spaces}, which we introduce in Section~\ref{subsec:generalized_hull}.  These are hierarchical spaces where the spaces $\fontact U$ are uniformly hyperbolic except when $U$ is minimal with respect to the nesting relation.  As their name suggests, this class includes all metrically relatively hyperbolic spaces; see Theorem~\ref{thm:rel_hyp_space}.  With an eye to future applications, in Section~\ref{subsec:generalized_hull} we prove a distance formula analogous to Theorem~\ref{thm:distance_formula} for relatively hierarchically hyperbolic spaces, and also establish the existence of hierarchy paths.  The strategy is to build, for each pair of points $x,y$, in the relatively hierarchically hyperbolic space, a ``hull'' of $x,y$, which we show is hierarchically hyperbolic with uniform constants.  We then apply Theorems~\ref{thm:distance_formula} and~\ref{thm:monotone_hierarchy_paths}.

\subsection*{Standard product regions and hierarchical quasiconvexity}
In Section~\ref{sec:hier_conv_hier_hyp}, we introduce the notion of a
\emph{hierarchically quasiconvex} subspace of a hierarchically
hyperbolic space $(\cuco X,\mathfrak S)$.  In the case where $\cuco X$
is hyperbolic, this notion coincides with the usual notion of
quasiconvexity.  The main technically useful features of
hierarchically quasiconvex subspaces generalize key features of
quasiconvexity: they inherit the property of being hierarchically
hyperbolic (Proposition~\ref{prop:sub_hhs}) and one can coarsely
project onto them (Lemma~\ref{lem:gate}).

Along with the hulls discussed above, the most important examples of
hierarchically quasiconvex subspaces are \emph{standard product
regions}: for each $U\in\mathfrak S$, one can consider the set $P_U$
of points $x\in\cuco X$ whose projection to each $\fontact V$ is
allowed to vary only if $V$ is nested into or orthogonal to $U$;
otherwise, $x$ projects to the same place in $\fontact V$ as $\fontact
U$ does under the relative projection.  The space $P_U$ coarsely
decomposes as a product, with factors corresponding to the nested and
orthogonal parts.  Product regions play an important role in the study
of boundaries and automorphisms of hierarchically hyperbolic spaces in
the forthcoming paper~\cite{DurhamHagenSisto:HHS_III}, as well as in
the study of quasi-boxes and quasiflats in hierarchically hyperbolic
spaces carried out in~\cite{BehrstockHagenSisto:HHS_I}.

\subsection*{Some questions and future directions}
Before embarking on the discussion outlined above, we raise a few 
questions about hierarchically hyperbolic spaces and groups that we 
believe are of significant interest.

The first set of questions concern the scope of the theory, i.e., which
groups and spaces are hierarchically hyperbolic and which operations
preserve the class of HHS:

\begin{questioni}[Cubical groups]\label{question:ccc} 
Let $G$ act properly and cocompactly on a CAT(0) cube complex.  Is $G$
a hierarchically hyperbolic group?  Conversely, suppose that
$(G,\mathfrak S)$ is a hierarchically hyperbolic group; are there 
conditions on the elements of $\mathfrak S$ which imply that $G$ acts properly and
cocompactly on a CAT(0) cube complex?\footnote{The first question was partially answered positively 
in~\cite{HagenSusse} after this paper was first posted.}
\end{questioni}

Substantial evidence for this conjecture was provided in 
\cite{BehrstockHagenSisto:HHS_I} where we established that a CAT(0)
cube complex $\cuco X$ containing a \emph{factor system} is a
hierarchically hyperbolic space, and the associated hyperbolic spaces
are all uniform quasi-trees.  (Roughly speaking, $\cuco X$ contains a
factor-system if the following collection of subcomplexes has finite
multiplicity: the smallest collection of convex subcomplexes that
contains all combinatorial hyperplanes and is closed under collecting
images of closest-point projection maps between its elements.)  The
class of cube complexes that are HHS in this way contains all
universal covers of special cube complexes with finitely many immersed
hyperplanes, but the cube complexes containing factor systems have not
been completely characterized.  In the forthcoming
paper~\cite{DurhamHagenSisto:HHS_III}, we show that the
above question is closely related to a conjecture of the first two
authors on the simplicial boundary of cube
complexes~\cite[Conjecture~2.8]{BehrstockHagen:cubulated1}.

More generally, we ask the following:

\begin{questioni}[Factor systems in median spaces]\label{question:median_factor}
Is there a theory of factor systems in median spaces generalizing that in CAT(0) cube complexes, such that median spaces/groups admitting factor systems are hierarchically hyperbolic?
\end{questioni}

\noindent Presumably, a positive answer to Question~\ref{question:median_factor} would involve the measured wallspace structure on median spaces discussed in~\cite{ChatterjiDrutuHaglund:measured_walls}.  One would have to develop an analogue of the contact graph of a cube complex to serve as the underlying hyperbolic space.  One must be careful since, e.g., the Baumslag-Solitar group $BS(1,2)$ acts properly on a median space but has exponential Dehn function~\cite{Gersten:dehn} and is thus not a hierarchically hyperbolic space, by Corollary~\ref{cor:HHG_quadratic_Dehn_function}.  On the other hand, if the answer to Question~\ref{question:median_factor} is positive, one might try to do the same thing for coarse median spaces.  

There are a number of other groups and spaces where it is natural to 
inquire whether or not they are hierarchically hyperbolic.  For example:

\begin{questioni}[Handlebody group]\label{question:handlebody_group}
Let $H$ be a compact oriented $3$--dimensional genus $g$ handlebody, and let $G_g\leq\MCG(\boundary H)$ be the group of isotopy classes of diffeomorphisms of $H$.  Is $G_g$ a hierarchically hyperbolic group?
\end{questioni}

\begin{questioni}[Graph products]\label{question:graph_products}
 Let $G$ be a (finite) graph product of hierarchically hyperbolic groups. Is $G$ hierarchically hyperbolic?
\end{questioni}

The answer to Question \ref{question:handlebody_group} is presumably
no, while the answer to \ref{question:graph_products} is most likely
yes.  The positive answer to Question \ref{question:graph_products}
would follow from a strengthened version of Theorem
\ref{thm:combination}.

There are other candidate examples of hierarchically hyperbolic spaces.  For example, it is natural to ask whether a right-angled Artin group with the syllable-length metric, introduced in~\cite{KimKoberda:curve_graph}, which is analogous to Teichm\"{u}ller space with the Weil-Petersson metric, is hierarchically hyperbolic.

As far as the difference between hierarchically hyperbolic \emph{spaces} and \emph{groups} is concerned, we conjecture that the following question has a positive answer:

\begin{questioni}
 Is it true that the fundamental group $G$ of a non-geometric graph manifold is a hierarchically hyperbolic group if and only if $G$ is virtually compact special?
\end{questioni}

It is known that $G$ as above is virtually compact special if and only if it is chargeless in the sense of \cite{BuyaloSvetlov:homological}, see \cite{HagenPrzytycki:graph}.

There remain a number of open questions about the geometry of 
hierarchically hyperbolic spaces in general. 
Theorem~\ref{thm:hier_hyp_coarse_median} ensures, via work of
Bowditch, that every asymptotic cone of a hierarchically hyperbolic
space is a median space \cite{Bowditch:coarse_median}; further 
properties in this direction are established in
Section~\ref{sec:convex_hulls}.  Motivated by combining the main
result of~\cite{Sisto:unique_cones} on $3$--manifold groups with
Theorem~\ref{thm:3mflds}, we ask:

\begin{questioni}\label{question:unique_cones} Are
any two asymptotic cones of a given 
hierarchically hyperbolic space bi-Lipschitz equivalent?
\end{questioni}

The notion of hierarchical quasiconvexity of a subgroup of a
hierarchically hyperbolic group $(G,\mathfrak S)$ generalizes
quasiconvexity in word-hyperbolic groups and cubical
convex-cocompactness in groups acting geometrically on CAT(0) cube
complexes with factor-systems.  Another notion of quasiconvexity is
\emph{stability}, defined by Durham-Taylor
in~\cite{DurhamTaylor:stable}.  This is a quite different notion of
quasiconvexity, since stable subgroups are necessarily hyperbolic.
In~\cite{DurhamTaylor:stable}, the authors characterize stable
subgroups of the mapping class group; it is reasonable to ask for a
generalization of their results to hierarchically hyperbolic
groups.

Many hierarchically hyperbolic spaces admit multiple hierarchically
hyperbolic structures.
However, as discussed
in~\cite{BehrstockHagenSisto:HHS_I}, a CAT(0) cube complex with a
factor-system has a ``minimal'' factor-system, i.e., one that is
contained in all other factor systems.  In this direction, it is 
natural to ask whether a
hierarchically hyperbolic space $(\cuco X,\mathfrak S)$ admits a
hierarchically hyperbolic structure that is canonical in some way.

\subsection*{Recent developments}
Since we posted the first version of this paper, there has been further progress on the theory of HHS and its applications.

More examples of HHS/HHG are now available, including a large class of CAT(0) cubical 
groups~\cite{HagenSusse}, ``small-cancellation'' quotients of HHGs~\cite{HHS_3}, and separating curve graphs of 
surfaces~\cite{Vokes:separating}.  It was also recently shown by Spriano that hyperbolic spaces/groups admit alternate HHS 
structures that can be constructed from an arbitrary fixed collection of quasiconvex subspaces/subgroups~\cite{Spriano1}.  
Spriano has proven additional results on modifying hierarchically hyperbolic structures to include pre-specified subgroups, 
under natural conditions~\cite{Spriano2}.  This allows him to prove that a large class of graphs of hierarchically 
hyperbolic groups are hierarchically hyperbolic.  In the latter vein, Berlai-Robbio have, in forthcoming 
work~\cite{BerlaiRobbio}, generalised the combination theorem (Theorem~\ref{thm:combination}) in this paper, and used this 
to show that the class of hierarchically hyperbolic groups is closed under taking graph products.

Further developments of the theory include finiteness of the
asymptotic dimension (including a quadratic upper bound for mapping
class groups)~\cite{HHS_3}; a theory of boundaries generalizing the
Gromov boundary of hyperbolic
groups~\cite{DurhamHagenSisto:HHS_III,Mousley1,Mousley2}; proof of the 
existence 
of largest acylindrical actions  \cite{AbbottBehrstock}; and 
a theorem
controlling quasiflats (new in both mapping class groups and
cubical groups) with many applications including, 
for instance, a new proof of quasi-isometric
rigidity for mapping class groups~\cite{HHS_n}.  Mousley-Russell have recently studied Morse boundaries of hierarchically 
hyperbolic groups~\cite{MousleyRussell}, and Abbott and the first-named author have established a 
linear bound on conjugator lengths in hierarchically hyperbolic groups~\cite{AbbottBehrstock:linear}.

We stress that the present paper is foundational for almost all of the above developments; the results here are used as tools there.

\subsection*{Organization of the paper}
Section~\ref{sec:background} contains the full definition of a hierarchically hyperbolic space (and, more generally, a hierarchical space) and some discussion of background.  Section~\ref{sec:preliminary} contains various basic consequences of the definition, and some tricks that are used repeatedly.  In Section~\ref{sec:realization}, we prove the realization theorem (Theorem~\ref{thm:realization}).  In Section~\ref{sec:hier_path_df} we establish the existence of hierarchy paths (Theorem~\ref{thm:monotone_hierarchy_paths}) and the distance formula (Theorem~\ref{thm:distance_formula}).  Section~\ref{sec:hierarchical_quasiconvexity_and_gates} is devoted to hierarchical quasiconvexity and product regions, and Section~\ref{sec:convex_hulls} to coarse convex hulls and relatively hierarchically hyperbolic spaces.    The coarse median property and its consequences are detailed in Section~\ref{sec:coarse_media}.  The combination theorems for trees of spaces, graphs of groups, and products are proved in Section~\ref{sec:combination}, and groups hyperbolic relative to HHG are studied in Section~\ref{sec:hyperbolicity_rel_HHS}.  This is applied to 3-manifolds in Section~\ref{sec:hier_hyp_3_manifolds}.  Finally, in Section~\ref{sec:MCG_HHS}, we prove that mapping class groups are hierarchically hyperbolic.

\subsection*{Acknowledgments} MFH and AS thank the organizers of the
Ventotene 2015 conference, and JB and MFH thank CIRM and the
organizers of GAGTA 2015, where some of the work on this paper was
completed.  We thank Federico Berlai, Bruno Robbio, Jacob Russell and Davide Spriano for numerous comments/corrections, and 
anonymous referees for further helpful comments.

\section{The main definition and background on hierarchically hyperbolic spaces}\label{sec:background}

\subsection{The axioms}\label{subsec:axioms}
We begin by defining a hierarchically hyperbolic space.  We 
will work in the context of a \emph{$q$--quasigeodesic space},  $\cuco X$, 
i.e., a 
metric space where any two
points can be connected by a $(q,q)$-quasigeodesic.  Obviously, if
$\cuco X$ is a geodesic space, then it is a quasigeodesic space.  Most
of the examples we are interested in are geodesic spaces, but in order to
construct hierarchically hyperbolic structures on naturally-occurring
subspaces of hierarchically hyperbolic spaces, we must work in the
slightly more general setting of quasigeodesic spaces.

\begin{defn}[Hierarchically hyperbolic space]\label{defn:space_with_distance_formula}
The $q$--quasigeodesic space  $(\cuco X,\dist_{\cuco X})$ is a \emph{hierarchically hyperbolic space} if there exists $\delta\geq0$, an index set $\mathfrak S$, and a set $\{\fontact W:W\in\mathfrak S\}$ of $\delta$--hyperbolic spaces $(\fontact U,\dist_U)$,  such that the following conditions are satisfied:\begin{enumerate}
\item\textbf{(Projections.)}\label{item:dfs_curve_complexes} There is
a set $\{\pi_W\co \cuco X\rightarrow2^{\fontact W}\mid W\in\mathfrak S\}$
of \emph{projections} sending points in $\cuco X$ to sets of diameter
bounded by some $\xi\geq0$ in the various $\fontact W\in\mathfrak S$.
Moreover, there exists $K$ so that for all $W\in\mathfrak S$, the coarse map $\pi_W$ is $(K,K)$--coarsely
Lipschitz and $\pi_W(\cuco X)$ is $K$--quasiconvex in $\fontact W$.

 \item \textbf{(Nesting.)} \label{item:dfs_nesting} $\mathfrak S$ is
 equipped with a partial order $\nest$, and either $\mathfrak
 S=\emptyset$ or $\mathfrak S$ contains a unique $\nest$--maximal
 element; when $V\nest W$, we say $V$ is \emph{nested} in $W$.  (We
 emphasize that $W\nest W$ for all $W\in\mathfrak S$.)  For each
 $W\in\mathfrak S$, we denote by $\mathfrak S_W$ the set of
 $V\in\mathfrak S$ such that $V\nest W$.  Moreover, for all $V,W\in\mathfrak S$
 with $V\propnest W$ there is a specified subset
 $\rho^V_W\subset\fontact W$ with $\diam_{\fontact W}(\rho^V_W)\leq\xi$.
 There is also a \emph{projection} $\rho^W_V\colon \fontact
 W\rightarrow 2^{\fontact V}$.  (The similarity in 
notation is justified by viewing $\rho^V_W$ as a coarsely constant map $\fontact
 V\rightarrow 2^{\fontact W}$.)
 
 \item \textbf{(Orthogonality.)} 
 \label{item:dfs_orthogonal} $\mathfrak S$ has a symmetric and
 anti-reflexive relation called \emph{orthogonality}: we write $V\orth
 W$ when $V,W$ are orthogonal.  Also, whenever $V\nest W$ and $W\orth
 U$, we require that $V\orth U$.  We require that for each
 $T\in\mathfrak S$ and each $U\in\mathfrak S_T$ for which
 $\{V\in\mathfrak S_T\mid V\orth U\}\neq\emptyset$, there exists $W\in
 \mathfrak S_T-\{T\}$, so that whenever $V\orth U$ and $V\nest T$, we
 have $V\nest W$.  Finally, if $V\orth W$, then $V,W$ are not
 $\nest$--comparable.
 
 \item \textbf{(Transversality and consistency.)}
 \label{item:dfs_transversal} If $V,W\in\mathfrak S$ are not
 orthogonal and neither is nested in the other, then we say $V,W$ are
 \emph{transverse}, denoted $V\transverse W$.  There exists
 $\kappa_0\geq 0$ such that if $V\transverse W$, then there are
  sets $\rho^V_W\subseteq\fontact W$ and
 $\rho^W_V\subseteq\fontact V$ each of diameter at most $\xi$ and 
 satisfying: $$\min\left\{\dist_{
 W}(\pi_W(x),\rho^V_W),\dist_{
 V}(\pi_V(x),\rho^W_V)\right\}\leq\kappa_0$$ for all $x\in\cuco X$.
 
 For $V,W\in\mathfrak S$ satisfying $V\nest W$ and for all
 $x\in\cuco X$, we have: $$\min\left\{\dist_{
 W}(\pi_W(x),\rho^V_W),\diam_{\fontact
 V}(\pi_V(x)\cup\rho^W_V(\pi_W(x)))\right\}\leq\kappa_0.$$ 
 
 The preceding two inequalities are the \emph{consistency inequalities} for points in $\cuco X$.
 
 Finally, if $U\nest V$, then $\dist_W(\rho^U_W,\rho^V_W)\leq\kappa_0$ whenever $W\in\mathfrak S$ satisfies either $V\propnest W$ or $V\transverse W$ and $W\not\orth U$.
 
 \item \textbf{(Finite complexity.)} \label{item:dfs_complexity} There exists $n\geq0$, the \emph{complexity} of $\cuco X$ (with respect to $\mathfrak S$), so that any set of pairwise--$\nest$--comparable elements has cardinality at most $n$.
  
 \item \textbf{(Large links.)} \label{item:dfs_large_link_lemma} There
exist $\lambda\geq1$ and $E\geq\max\{\xi,\kappa_0\}$ such that the following holds.
Let $W\in\mathfrak S$ and let $x,x'\in\cuco X$.  Let
$N=\lambda\dist_{_W}(\pi_W(x),\pi_W(x'))+\lambda$.  Then there exists $\{T_i\}_{i=1,\dots,\lfloor
N\rfloor}\subseteq\mathfrak S_W-\{W\}$ such that for all $T\in\mathfrak
S_W-\{W\}$, either $T\in\mathfrak S_{T_i}$ for some $i$, or $\dist_{
T}(\pi_T(x),\pi_T(x'))<E$.  Also, $\dist_{
W}(\pi_W(x),\rho^{T_i}_W)\leq N$ for each $i$.

 \item \textbf{(Bounded geodesic image.)}
 \label{item:dfs:bounded_geodesic_image} There exists $E>0$ such that 
 for all $W\in\mathfrak S$,
 all $V\in\mathfrak S_W-\{W\}$, and all geodesics $\gamma$ of
 $\fontact W$, either $\diam_{\fontact V}(\rho^W_V(\gamma))\leq E$ or
 $\gamma\cap\neb_E(\rho^V_W)\neq\emptyset$.
 
 \item \textbf{(Partial Realization.)} \label{item:dfs_partial_realization} There exists a constant $\alpha$ with the following property. Let $\{V_j\}$ be a family of pairwise orthogonal elements of $\mathfrak S$, and let $p_j\in \pi_{V_j}(\cuco X)\subseteq \fontact V_j$. Then there exists $x\in \cuco X$ so that:
 \begin{itemize}
 \item $\dist_{V_j}(x,p_j)\leq \alpha$ for all $j$,
 \item for each $j$ and 
 each $V\in\mathfrak S$ with $V_j\nest V$, we have 
 $\dist_{V}(x,\rho^{V_j}_V)\leq\alpha$, and
 \item if $W\transverse V_j$ for some $j$, then $\dist_W(x,\rho^{V_j}_W)\leq\alpha$.
 \end{itemize}

\item\textbf{(Uniqueness.)} For each $\kappa\geq 0$, there exists
$\theta_u=\theta_u(\kappa)$ such that if $x,y\in\cuco X$ and
$\dist_{\cuco X}(x,y)\geq\theta_u$, then there exists $V\in\mathfrak S$ such
that $\dist_V(x,y)\geq \kappa$.\label{item:dfs_uniqueness}
\end{enumerate}
We say that the $q$--quasigeodesic metric spaces $\{\cuco X_i\}$ are \emph{uniformly
hierarchically hyperbolic} if each $\cuco X_i$ satisfies the axioms
above and all constants, including the complexities, can be chosen
uniformly.  We often refer to $\mathfrak S$, together with the nesting
and orthogonality relations, and the projections as a \emph{hierarchically hyperbolic structure} for the space $\cuco
X$.  Observe that $\cuco X$ is hierarchically hyperbolic with respect
to $\mathfrak S=\emptyset$, i.e., hierarchically hyperbolic of
complexity $0$, if and only if $\cuco X$ is bounded.  Similarly,
$\cuco X$ is hierarchically hyperbolic of complexity $1$ with respect
to $\mathfrak S=\{\cuco X\}$, if and only if $\cuco X$ is hyperbolic.
\end{defn}

\begin{notation}\label{notation:suppress_pi}
Where it will not cause confusion, given $U\in\mathfrak S$, we will often suppress the projection
map $\pi_U$ when writing distances in $\fontact U$, i.e., given $x,y\in\cuco X$ and
$p\in\fontact U$  we write
$\dist_U(x,y)$ for $\dist_U(\pi_U(x),\pi_U(y))$ and $\dist_U(x,p)$ for
$\dist_U(\pi_U(x),p)$. Note that when we measure distance between a 
pair of sets (typically both of bounded diameter) we are taking the minimum distance 
between the two sets. 
Given $A\subset \cuco X$ and $U\in\mathfrak S$ 
we let $\pi_{U}(A)$ denote $\cup_{a\in A}\pi_{U}(a)$.
\end{notation}

\begin{rem}[Surjectivity of projections]\label{rem:surjectivity}
In all of the motivating examples, and in most applications, the maps $\pi_U$ are uniformly coarsely surjective.

 One can always replace each $\fontact U$ with a thickening of $\pi_{U}(\cuco 
X)$, and hence make each $\pi_U$ coarsely surjective. This is first discussed in \cite{DurhamHagenSisto:HHS_III}, where this 
procedure gets used; the 
resulting spaces are termed \emph{normalized} hierarchically hyperbolic spaces.

More precisely, since each $\pi_U(\cuco X)$ is $K$--quasiconvex, the
subset $\fontact U_{norm}$ of $\fontact U$ consisting of all geodesics
that start and end in $\pi_U(\cuco X)$ is uniformly quasiconvex, is a
(uniformly) hyperbolic geodesic metric space, and uniformly coarsely
coincides with $\pi_U(\cuco X)$.  (This ``quasiconvex hull'' procedure
is discussed in more detail in Section~\ref{sec:convex_hulls}.)  Hence
we can endow $\cuco X$ with a slightly different, normalized,
hierarchically hyperbolic structure.  Indeed, the index set is still
$\mathfrak S$, each $\fontact U$ is replaced by $\fontact U_{norm}$,
and the maps $\pi_U$ remain unchanged (but are now coarsely
surjective).  Given $U,V\in\mathfrak S$ such that $\rho^U_V$ is
defined, we replace $\rho^U_V$ (viewed as a coarse map $\fontact
U\to\fontact V$) with the composition $p_V\circ\rho^U_V$, where
$p_V\co\fontact V\to\fontact U_{norm}$ is the coarse closest-point
projection.
\end{rem}

\begin{rem}[Surjectivity/quasiconvexity of projections in the extant applications]\label{rem:surjectivity_in_other_papers}
In the motivating examples (mapping class groups, Teichm\"uller space, virtually special groups, hyperbolic spaces, etc.), 
the projections $\pi_U$ are uniformly coarsely surjective, but it is convenient to relax that requirement.  As is evident 
from Theorem~\ref{thm:realization} and the key Lemma~\ref{lem:median_cons}, the appropriate relaxation of coarse 
surjectivity is the requirement, from Definition~\ref{defn:space_with_distance_formula}.\eqref{item:dfs_curve_complexes}, 
that each $\pi_U(\cuco X)$ be uniformly quasiconvex in $\fontact U$.  

In a few other places in the literature, this is not spelled out, but in each case where an issue arises, it does not affect 
the arguments in question.  In the interest of clarity, we now summarise this as follows:
\begin{itemize}
     \item In~\cite[p. 4, p. 19]{DurhamHagenSisto:HHS_III}, the authors establish a standing assumption that they are 
working with normalized HHSs --- each $\pi_U$ is uniformly coarsely surjective.  In view of Remark~\ref{rem:surjectivity} 
(or~\cite[Proposition 1.16]{DurhamHagenSisto:HHS_III}), the results about normalized HHSs can be promoted to corresponding 
statements about general HHSs.
\item In~\cite{HHS_3}, Remark~\ref{rem:surjectivity} allows one to assume that the HHSs in question are normalised.  
However, there are three places where a new HHS is constructed from an old one, and one must observe that in each of these 
cases, the new projections have quasiconvex image.  In~\cite[Proposition~2.4]{HHS_3}, this holds just because the 
projections used in the new HHS structure coincide with those used in the old HHS structure, so quasiconvexity persists.  In 
Proposition 6.14 and Theorem 6.2 of~\cite{HHS_3}, the projections in the new HHS structures are of two types: they either 
coincide with projections from the old HHS structures, and thus have quasiconvex images, or they are surjective by 
construction.
\end{itemize}
\end{rem}

\begin{rem}[Large link function]\label{rem:large_link}
It appears as though there is no actual need to require in Definition~\ref{defn:space_with_distance_formula}.\eqref{item:dfs_large_link_lemma} that $N$ depend linearly on $\dist_W(x,x')$.  Instead, we could have hypothesized that for any $C\geq0$, there exists $N(C)$ so that the statement of the axiom holds with $N=N(C)$ whenever $\dist_W(x,x')\leq C$.  However, one could deduce from this and the rest of the axioms that $N(C)$ grows linearly in $C$, so we have elected to simply build linearity into the definition.
\end{rem}

\begin{rem}[Summary of constants]\label{rem:constants}
Each hierarchically hyperbolic space $(\cuco X,\mathfrak S)$ is
associated with a collection of constants often, as above, denoted 
$\delta,\xi,n,\kappa_0,E,\theta_u,K$, where:
\begin{enumerate}
 \item $\fontact U$ is $\delta$--hyperbolic for each $U\in\mathfrak S$,
\item each $\pi_U$ has image of diameter at most $\xi$ and each
$\pi_U$ is $(K,K)$--coarsely Lipschitz, $\pi_U(\cuco X)$ is
$K$--quasiconvex in $\fontact U$, and each $\rho^U_V$ has (image of)
diameter at most $\xi$,
 \item for each $x\in\cuco X$, the tuple $(\pi_U(x))_{U\in\mathfrak S}$ is $\kappa_0$--consistent,
 \item $E$ is the larger of the constants from the bounded geodesic 
 image axiom and the large link axiom.
 \end{enumerate}
Whenever working in a fixed hierarchically hyperbolic space, we use the above notation freely.  We can, and shall, assume that $E\geq q,E\geq\delta,E\geq\xi,E\geq\kappa_0,E\geq K$, and $E\geq\alpha$.
\end{rem}

\begin{rem}
    We note that in 
    Definition~\ref{defn:space_with_distance_formula}.(\ref{item:dfs_curve_complexes}),
    the assumption that the projections are Lipschitz can be replaced
    by the weaker assumption that there is a proper function of the 
    projected distance which is a lower bound for the distance in the 
    space $\cuco X$. From this weaker assumption, the fact that the 
    projections are actually coarsely Lipschitz then follows from the fact that we assume $\cuco X$ to be quasi-geodesic.
Since the Lipschitz hypothesis is cleaner to state and, in practice, fairly easy to verify, we just remark on this for those that might find 
     this fact useful in proving that more exotic spaces are hierarchically hyperbolic. 
\end{rem}

\subsection{Comparison to the definition in~\cite{BehrstockHagenSisto:HHS_I}}\label{subsec:old_and_new_defs}

Definition~\ref{defn:space_with_distance_formula} is very similar to the definition of a hierarchically hyperbolic space given in~\cite{BehrstockHagenSisto:HHS_I}, with the following differences:
\begin{enumerate}
 \item The existence of \emph{hierarchy paths} and the \emph{distance formula} were stated as axioms in~\cite{BehrstockHagenSisto:HHS_I}; below, we deduce them from the other axioms.  Similarly, the below \emph{realization theorem} was formerly an axiom, but has been replaced by the (weaker) partial realization axiom.
 \item We now require $\cuco X$ to be a quasigeodesic space. In~\cite{BehrstockHagenSisto:HHS_I}, this follows from the existence of hierarchy paths, which was an axiom there.
 \item We now require the projections $\pi_U:\cuco X\to\fontact U$ to be coarsely Lipschitz; although this requirement was not imposed explicitly in~\cite{BehrstockHagenSisto:HHS_I}, it follows from the distance formula, which was an axiom there.
 \item In~\cite{BehrstockHagenSisto:HHS_I}, there were five consistency inequalities; in Definition~\ref{defn:space_with_distance_formula}.\eqref{item:dfs_transversal}, there are two.  The last three inequalities in the definition from~\cite{BehrstockHagenSisto:HHS_I} follow from Proposition~\ref{prop:rho_consistency} below.  (Essentially, the partial realization axiom has replaced part of the old consistency axiom.)
 \item In Definition~\ref{defn:space_with_distance_formula}.\eqref{item:dfs_transversal}, we require that, if $U\nest V$, 
then $\dist_W(\rho^U_W,\rho^V_W)\leq\kappa_0$ whenever $W\in\mathfrak S$ satisfies either $V\propnest W$ or $V\transverse W$ 
and $W\not\orth U$.  In the context of~\cite{BehrstockHagenSisto:HHS_I}, this follows by considering the standard product 
regions constructed using realization (see~\cite[Section 13.1]{BehrstockHagenSisto:HHS_I} and 
Section~\ref{subsec:product_regions} of the present paper).
\end{enumerate}

\begin{prop}[$\rho$--consistency]\label{prop:rho_consistency}
There exists $\kappa_1$ so that the following holds.  Suppose that $U,V,W\in\mathfrak S$ satisfy both of the following conditions: $U\propnest V$ or $U\transverse V$; and $U\propnest W$ or $U\transverse W$.  Then, if $V\transverse W$, then 
 $$\min\left\{\dist_{
 W}(\rho^U_W,\rho^V_W),\dist_{
 V}(\rho^U_V,\rho^W_V)\right\}\leq\kappa_1$$
 and if $V\propnest W$, then 
 $$\min\left\{\dist_{
 W}(\rho^U_W,\rho^V_W),\diam_{\fontact
 V}(\rho^U_V\cup\rho^W_V(\rho^U_W))\right\}\leq\kappa_1.$$
\end{prop}

\begin{proof}
Suppose that $U\propnest V$ or $U\transverse V$ and the same holds for $U,W$.  Suppose that $V\transverse W$ or $V\nest W$.  Choose $p\in\pi_U(\cuco X)$.  There is a uniform $\alpha$ so that partial realization (Definition~\ref{defn:space_with_distance_formula}.\eqref{item:dfs_partial_realization}) provides $x\in\cuco X$ so that $\dist_U(x,p)\leq\alpha$ and $\dist_T(x,\rho^U_T)\leq\alpha$ whenever $\rho^U_T$ is defined and coarsely constant.  In particular, $\dist_V(x,\rho^U_V)\leq\alpha$ and $\dist_W(x,\rho^U_W)\leq\alpha$.  The claim now follows from Definition~\ref{defn:space_with_distance_formula}.\eqref{item:dfs_transversal}, with $\kappa_1=\kappa_0+\alpha$.
\end{proof}

In view of the discussion above, we have:

\begin{prop}\label{prop:same_definition}
 The pair $(\cuco X,\mathfrak S)$ satisfies Definition~\ref{defn:space_with_distance_formula} if and only if it is hierarchically hyperbolic in the sense of~\cite{BehrstockHagenSisto:HHS_I}.
\end{prop}

\noindent In particular, as observed in~\cite{BehrstockHagenSisto:HHS_I}:

\begin{prop}\label{prop:qie}
If $(\cuco X,\mathfrak S)$ is a hierarchically hyperbolic space, and $\cuco X'$ is a quasigeodesic space quasi-isometric to $\cuco X$, then there is a hierarchically hyperbolic space $(\cuco X',\mathfrak S)$.
\end{prop}

\subsection{A variant on the axioms}\label{subsec:axiom_variant} Here 
we introduce two slightly simpler versions of the HHS axioms and show 
that in the case, as in most situations which arise naturally, that 
the projections are coarsely surjective, then it suffices to verify 
the simpler axioms.

The following is a subset of the nesting axiom; here we 
remove the definition of the projection map $\rho^W_V\colon
\fontact W\rightarrow 2^{\fontact V}$ in the case $V\propnest 
W$.

\newtheorem*{axiomNest}{Definition~\ref{defn:space_with_distance_formula}.(\ref{item:dfs_nesting})$'$}

\begin{axiomNest}[Nesting variant]\label{axiom:nesting_variant}
   $\mathfrak S$ is
   equipped with a partial order $\nest$, and either $\mathfrak
   S=\emptyset$ or $\mathfrak S$ contains a unique $\nest$--maximal
   element; when $V\nest W$, we say $V$ is \emph{nested} in $W$.  We
   require that $W\nest W$ for all $W\in\mathfrak S$.  For each
   $W\in\mathfrak S$, we denote by $\mathfrak S_W$ the set of
   $V\in\mathfrak S$ such that $V\nest W$.  Moreover, for all $V,W\in\mathfrak S$
   with $V\propnest W$ there is a specified subset
   $\rho^V_W\subset\fontact W$ with $\diam_{\fontact W}(\rho^V_W)\leq\xi$.
\end{axiomNest}

The following is a subset of the transversality and consistency axiom.

\newtheorem*{axiomTrans}{Definition~\ref{defn:space_with_distance_formula}.(\ref{item:dfs_transversal})$'$}

\begin{axiomTrans}[Transversality]\label{axiom:transversal_variant}
    If $V,W\in\mathfrak S$ are not
    orthogonal and neither is nested in the other, then we say $V,W$ are
    \emph{transverse}, denoted $V\transverse W$.  There exists
    $\kappa_0\geq 0$ such that if $V\transverse W$, then there are
     sets $\rho^V_W\subseteq\fontact W$ and
    $\rho^W_V\subseteq\fontact V$ each of diameter at most $\xi$ and 
    satisfying: $$\min\left\{\dist_{
    W}(\pi_W(x),\rho^V_W),\dist_{
    V}(\pi_V(x),\rho^W_V)\right\}\leq\kappa_0$$ for all $x\in\cuco X$.
      
    Finally, if $U\nest V$, then $\dist_W(\rho^U_W,\rho^V_W)\leq\kappa_0$ whenever $W\in\mathfrak S$ satisfies either $V\propnest W$ or $V\transverse W$ and $W\not\orth U$.  
\end{axiomTrans}

The following is a variant of the bounded geodesic image axiom:

\newtheorem*{axiomBGIv}{Definition~\ref{defn:space_with_distance_formula}.(\ref{item:dfs:bounded_geodesic_image})$'$}

\begin{axiomBGIv}[Bounded geodesic image variant]\label{axiom:BGI_variant}
    Suppose that $x,y\in X$ and $V\propnest W$ have the property that there exists a geodesic from $\pi_W(x)$ to $\pi_W(y)$ which stays $(E+2\delta)$-far from $\rho^V_W$. Then 
    $\dist_V(x,y)\leq E$.
\end{axiomBGIv}

\begin{prop}\label{prop:new_hhs_defn} Given a quasigeodesic space $\cuco X$ and an index set 
    $\mathfrak S$, then $(\cuco X,\mathfrak S)$ is an HHS if it 
    satisfies the axioms of 
    Definition~\ref{defn:space_with_distance_formula} with the 
    following changes:
    \begin{itemize}
	\item Replace 
	Definition~\ref{defn:space_with_distance_formula}.(\ref{item:dfs_nesting})
    by
    Definition~\ref{defn:space_with_distance_formula}.(\ref{item:dfs_nesting})$'$.

        \item Replace Definition~\ref{defn:space_with_distance_formula}.(\ref{item:dfs_transversal}) by 
    Definition~\ref{defn:space_with_distance_formula}.(\ref{item:dfs_transversal})$'$.
    
        \item Replace 
    Definition~\ref{defn:space_with_distance_formula}.(\ref{item:dfs:bounded_geodesic_image})
    by Definition~\ref{defn:space_with_distance_formula}.(\ref{item:dfs:bounded_geodesic_image})$'$. 

        \item Assume that for each $\fontact U$ the map 
	$\pi_{U}$ is uniformly coarsely surjective.
    \end{itemize}
\end{prop}

\begin{proof}
To verify Definition~\ref{defn:space_with_distance_formula}.\eqref{item:dfs_nesting}, for each $V,W\in\mathfrak S$ 
with $V\propnest W$, we define a map $\rho^W_V\colon\fontact W\to2^{\fontact V}$ as follows.    If $p\in\fontact 
W-\neb_E(\rho^V_W)$, then let $\rho^W_V(p)=\pi_V(x)$ for some $x\in\cuco X$ with $\pi_W(x)$ (uniformly) coarsely 
coinciding with $p$.  Since $p$ does not lie $E$--close to $\rho^V_W$, this definition is coarsely independent of $x$ 
by Definition~\ref{defn:space_with_distance_formula}.\eqref{item:dfs:bounded_geodesic_image}$'$.  On 
$\neb_E(\rho^V_W)$, we define $\rho^W_V$ arbitrarily.  By definition, the resulting map satisfies 
Definition~\ref{defn:space_with_distance_formula}.\eqref{item:dfs_transversal}.  Moreover, coarse surjectivity of 
$\pi_W$ and Definition~\ref{defn:space_with_distance_formula}.\eqref{item:dfs:bounded_geodesic_image}$'$ ensure that 
Definition~\ref{defn:space_with_distance_formula}.\eqref{item:dfs:bounded_geodesic_image} holds.  The rest of the 
axioms hold by hypothesis.
\end{proof}

\begin{rem}
The definition of an HHS provided by Proposition~\ref{prop:new_hhs_defn} is convenient because it does not require one to define certain maps between hyperbolic spaces: Definition~\ref{defn:space_with_distance_formula}.\eqref{item:dfs_nesting}$'$ is strictly weaker than Definition~\ref{defn:space_with_distance_formula}.\eqref{item:dfs_nesting}.  On the other hand, it is often convenient to work with HHS in which some of the projections $\pi_U$ are not coarsely surjective; for example, this simplifies the proof that hierarchically quasiconvex subspaces inherit HHS structures in Proposition~\ref{prop:sub_hhs}.  Hence we have included both definitions.
\end{rem}

In practice, we almost always apply consistency and bounded geodesic image in concert, which involves applying bounded geodesic image to geodesics of $\fontact W$ joining points in $\pi_W(\cuco X)$.  Accordingly, Definition~\ref{defn:space_with_distance_formula}.\eqref{item:dfs:bounded_geodesic_image}$'$ is motivated by the following easy observation:

\begin{prop}\label{prop:new_defn_local}
Let $(\cuco X,\mathfrak S)$ be an HHS.  Then the conclusion of  Definition~\ref{defn:space_with_distance_formula}.\eqref{item:dfs:bounded_geodesic_image}$'$ holds for all $x,y\in\cuco X$ and $V,W\in\mathfrak S$ with $V\propnest W$.
\end{prop}

\subsection{Hierarchical spaces}\label{subsec:hierarchical_space}
Although most of our focus in this paper is on hierarchically hyperbolic spaces, there are important contexts in which hyperbolicity of the spaces $\fontact U,U\in\mathfrak S$ is not used; notably, this is the case for the realization theorem (Theorem~\ref{thm:realization}).  Because of the utility of a more general definition in later applications, we now define the following more general notion of a \emph{hierarchical space}; the reader interested only in the applications to the mapping class group, $3$--manifolds, cube complexes, etc.\ may safely ignore this subsection.

\begin{defn}[Hierarchical space]\label{defn:hierarchical_space}
A \emph{hierarchical space} is a pair $(\cuco X,\mathfrak S)$ as in Definition~\ref{defn:space_with_distance_formula}, with $\cuco X$ a quasigeodesic space and $\mathfrak S$ an index set, where to each $U\in\mathfrak S$ we associate a geodesic metric space $\fontact U$, which we \emph{do not require to be hyperbolic}.  As before, there are coarsely Lipschitz projections $\pi_U\colon\cuco X\to\fontact U$ and relative projections $\rho^U_V\colon\fontact U\to\fontact V$ whenever $U,V$ are non-orthogonal.  We require all statements in the Definition~\ref{defn:space_with_distance_formula} to hold, except for hyperbolicity of the $\fontact U$.
\end{defn}

\begin{rem}\label{rem:rel_hyp_HS}
Let $\cuco X$ be a quasigeodesic space that is hyperbolic relative to a collection $\mathcal P$ of subspaces.  Then $\cuco X$ has a hierarchical space structure: the associated spaces onto which we project are the various $\mathcal P$, together with the space $\widehat{\mathcal X}$ obtained by coning off the elements of $\mathcal P$ in $\cuco X$.  When the elements of $\mathcal P$ are themselves hierarchically hyperbolic, we obtain a hierarchically hyperbolic structure on $\cuco X$ (see Section~\ref{sec:hyperbolicity_rel_HHS}).  Otherwise, the hierarchical structure need not be hierarchically hyperbolic since $\widehat{\cuco X}$ is the only one of the elements of $\mathfrak S$ known to be hyperbolic.
\end{rem}

\begin{rem}
Other than hierarchically hyperbolic spaces, we are mainly interested in hierarchical spaces $(\cuco X,\mathfrak S)$ where for all $U\in\mathfrak S$, except possibly when $U$ is $\nest$--minimal, we have that $\fontact U$ is hyperbolic. This is the case, for example, in relatively hyperbolic spaces.
\end{rem}

\subsection{Consistency and partial realization points}\label{subsec:consistency}
The following definitions, which abstract the consistency inequalities
from Definition~\ref{defn:space_with_distance_formula}.\eqref{item:dfs_transversal}
and the partial realization axiom, Definition~\ref{defn:space_with_distance_formula}.\eqref{item:dfs_partial_realization},
play important roles throughout our discussion. We will consider this topic in depth in Section~\ref{sec:realization}.

\begin{defn}[Consistent]\label{defn:consistent}
Fix $\kappa\geq0$ and let $\tup b\in\prod_{U\in\mathfrak S}2^{\fontact U}$ be a tuple such that for each $U\in\mathfrak S$, 
the coordinate $b_U$ is a subset of $\fontact U$ with $\diam_{\fontact U}(b_U)\leq\kappa$.  The tuple $\tup b$ is 
\emph{$\kappa$--admissible} if $\dist_U(b_U,\pi_U(\cuco X))\leq \kappa$ for all $U\in\mathfrak S$.  The 
$\kappa$--admissible tuple $\tup b$ is \emph{$\kappa$--consistent} if, whenever $V\transverse W$,
$$\min\left\{\dist_{ W}(b_W,\rho^V_W),\dist_{ V}(b_V,\rho^W_V)\right\}\leq\kappa$$
and whenever $V\nest W$, 
$$\min\left\{\dist_{ W}(b_W,\rho^V_W),\diam_{\fontact V}(b_V\cup\rho^W_V(b_W))\right\}\leq\kappa.$$
In typical situations, where the maps $\pi_U$ are uniformly coarsely surjective, up 
to a uniform enlargement of $E$, all tuples are admissible, so verifying consistency amounts to verifying the second 
condition.
\end{defn}

\begin{defn}[Partial realization point]\label{defn:partial_realization_point}
Given $\theta\geq0$ and a $\kappa$--consistent tuple $\tup b$, we say that $x\in\cuco X$ is a $\theta$--\emph{partial 
realization point} for $\{V_j\}\subseteq\mathfrak S$ if
\begin{enumerate}
 \item $\dist_{V_j}(x,b_{V_j})\leq \theta$ for all $j$,
 \item for all $j$, we have $\dist_{V}(x,\rho^{V_j}_V)\leq\theta$ for any $V\in\mathfrak S$ with $V_j\nest V$, and
 \item for all $W$ such that $W\transverse V_j$ for some $j$, we have $\dist_W(x,\rho^{V_j}_W)\leq\theta$.
\end{enumerate}
\end{defn}

Observe that if $\tup b$ is consistent and $\{V_j\}$ is a set of pairwise-orthogonal elements, then 
partial realisation (Definition~\ref{defn:space_with_distance_formula}.\eqref{item:dfs_partial_realization}) provides a 
partial realisation point, because of admissibility.

\subsection{Level}
The following definition is very useful for proving statements about
hierarchically hyperbolic spaces inductively. Although it is natural,
and sometimes useful, to induct on complexity, it is often better to
induct on level:

\begin{defn}[Level]\label{defn:level}
 Let $(X,\mathfrak S)$ be hierarchically hyperbolic. The \emph{level} $\ell_U$ of $U\in\mathfrak S$ is defined inductively as follows. If $U$ is $\nest$-minimal then we say that its level is $1$.  The element $U$ has level $k+1$ if $k$ is the maximal integer such that there exists $V\nest U$ with $\ell_V=k$ and $V\neq U$.  Given $U\in\mathfrak S$, for each $\ell\geq0$, let $\mathfrak S_U^\ell$ be the set of $V\nest U$ with $\ell_U-\ell_V\leq\ell$ and let $\mathfrak T_U^\ell=\mathfrak S_U^\ell-\mathfrak S_U^{\ell-1}$.
\end{defn}

\subsection{Maps between hierarchically hyperbolic spaces}\label{subsec:maps}
\begin{defn}[Hieromorphism]\label{defn:hieromorphism}
Let $(\cuco X,\mathfrak S)$ and $(\cuco X',\mathfrak S')$ be
hierarchically hyperbolic structures on the spaces $\cuco X,\cuco X'$
respectively.  A \emph{hieromorphism}, 
consists of a map $f\co\cuco
X\rightarrow\cuco X'$, an injective map $f\inducedS \co\mathfrak S\rightarrow\mathfrak
S'$ preserving nesting, transversality, and orthogonality, and, for 
each $U\in\mathfrak S$ a map 
$f\induced(U)\colon \fontact U\rightarrow \fontact(f\inducedS(U))$ 
which is a quasi-isometric embedding where the constants are uniform
over all elements of $\mathfrak S$ and for which 
the following two diagrams coarsely commute (with
uniform constants) for all nonorthogonal $U,V\in\mathfrak S$:
\begin{center}
$
\begin{diagram}
\node{\cuco 
X}\arrow[3]{e,t}{f}\arrow{s,l}{\pi_U}\node{}\node{}\node{\cuco 
X'}\arrow{s,r}{\pi_{f\inducedS(U)}}\\
\node{\fontact (U)}\arrow[3]{e,t}{f\induced(U)}\node{}\node{}\node{\fontact 
(f\inducedS(U))}
\end{diagram}
$
\end{center}
and
\begin{center}
$
\begin{diagram}
\node{\fontact 
U}\arrow[4]{s,l}{\rho^U_V}\arrow[7]{e,t}{f\induced(U)}\node{}\node{}\node{}\node{}\node{}\node{}\node{}\node{}\node{}\node{\fontact(f\inducedS(U))}\arrow[4]{s,r}{\rho^{f\inducedS(U)}_{f\inducedS(V)}}\\
\node{}\node{}\node{}\node{}\node{}\node{}\node{}\node{}\\
\node{}\node{}\node{}\node{}\node{}\node{}\node{}\node{}\\
\node{}\node{}\node{}\node{}\node{}\node{}\node{}\node{}\\
\node{\fontact 
V}\arrow[7]{e,t}{f\induced(V)}\node{}\node{}\node{}\node{}\node{}\node{}\node{}\node{}\node{}\node{\fontact (f\inducedS (V))}
\end{diagram}
$
\end{center}
where $\rho^U_V\colon\fontact U\to\fontact V$ is the projection from
Definition~\ref{defn:space_with_distance_formula}, which, by 
construction, is coarsely
constant if $U\transverse V$ or $U\nest V$. As the functions $f, 
f\induced(U),$ and $f\inducedS$ all have distinct domains, it is 
often clear from the 
context which is the relevant map; in that case we periodically abuse 
notation slightly by dropping the superscripts and just calling all of the maps $f$.
\end{defn}

\begin{defn}[Automorphism, hierarchically hyperbolic group]\label{defn:hhs_automorphism}
An \emph{automorphism} of the hierarchically hyperbolic space 
$(\cuco X,\mathfrak S)$ is a hieromorphism 
$f\co (\cuco X,\mathfrak S)\rightarrow(\cuco X,\mathfrak S)$ such   
that $f\inducedS$ is bijective and each $f\induced(U)$ is an 
isometry; hence $f\co \cuco X\to \cuco X$ is a uniform quasi-isometry by the distance formula (Theorem 
\ref{thm:distance_formula}).  

Note that the composition of two automorphisms is again an automorphism.  We say that the automorphisms $f,f'$ are 
\emph{equivalent} if $f\inducedS=(f')\inducedS$ and $f\induced(U)=(f')\induced(U)$ for each $U\in\mathfrak S$.  In 
particular, equivalent automorphisms give equivalent quasi-isometries.  Given an automorphism $f$, any quasi-inverse $\bar 
f$ of $f$ is an automorphism with $\bar f\inducedS=(f\inducedS)^{-1}$ and each $\bar f\induced(U)=f\induced(U)^{-1}$.  Hence 
the set of equivalence classes of automorphisms forms a group, the \emph{full automorphism group} of $(\cuco X,\mathfrak 
S)$, denoted $\Aut(\mathfrak S)$.

The finitely generated group $G$ is \emph{hierarchically hyperbolic}
if there exists a hierarchically hyperbolic space $(\cuco X,\mathfrak
S)$ and an action $G\rightarrow\Aut(\mathfrak S)$ so that the uniform quasi-action of $G$ on $\cuco X$ is metrically proper and cobounded
and $\mathfrak S$ contains finitely many $G$--orbits.  Note that if
$G$ is hierarchically hyperbolic by virtue of its action on the
hierarchically hyperbolic space $(\cuco X,\mathfrak S)$, then
$(G,\mathfrak S)$ is a hierarchically hyperbolic structure with
respect to any word-metric on $G$; for any $U\in\mathfrak S$ the 
projection is the
composition of the projection $\cuco X\rightarrow \fontact
U$ with a $G$--equivariant quasi-isometry
$G\rightarrow\cuco X$.  In this case, $(G,\mathfrak S)$ (with the
implicit hyperbolic spaces and projections) is a \emph{hierarchically
hyperbolic group structure}.
\end{defn}

\begin{defn}[Equivariant hieromorphism]\label{defn:homomorphism_of_HHGs}
Let $(\cuco X,\mathfrak S)$ and $(\cuco X',\mathfrak G')$ be
hierarchically hyperbolic spaces and consider actions
$G\rightarrow\Aut(\mathfrak S)$ and $G'\to\Aut(\mathfrak S')$.  For each $g\in G$, let
$(f_g,f_g\inducedS,\{f_g\induced(U)\})$  denote its image in
$\Aut(\mathfrak S)$ (resp., for $g'\in G'$ we obtain 
$(f_{g'},f_{g'}\inducedS,\{f_{g'}\induced(U)\})$ in  $\Aut(\mathfrak S')$).
Let $\phi\co G\rightarrow G'$ be a homomorphism.  The hieromorphism
$(f,f\inducedS,\{f\induced(U)\})\colon(\cuco X,\mathfrak S)\rightarrow(\cuco
X',\mathfrak S')$ is \emph{$\phi$--equivariant} if for all $g\in G$
and $U\in\mathfrak S$, we have 
$f\inducedS(f_{g}\inducedS(U))=f_{\phi(g)}\inducedS(f\inducedS(U))$  
and the following
diagram (uniformly) coarsely commutes:
\begin{center}
$
\begin{diagram}
\node{}\node{}\node{\fontact 
U}\arrow[4]{s,l}{f_g\induced(U)}\arrow[15]{e,t}{f\induced(U)}\node{}\node{}\node{}\node{}\node{}\node{}\node{}\node{}\node{}\node{}\node{}\node{}\node{}\node{}\node{}\node{}\node{}\node{\fontact (f\inducedS(U))}\arrow[3]{s,r}{f_{\phi(g)}\induced(U))}\\
\node{}\node{}\node{}\node{}\node{}\node{}\node{}\node{}\\
\node{}\node{}\node{}\node{}\node{}\node{}\node{}\node{}\\
\node{}\node{}\node{}\node{}\node{}\node{}\node{}\node{}\\
\node{}\node{}\node{\fontact(f_g\inducedS(U))}\arrow[14]{e,t}{f\induced(f_{g}\inducedS(U))}\node{}\node{}\node{}\node{}\node{}\node{}\node{}\node{}\node{}\node{}\node{}\node{}\node{}\node{}\node{}\node{}\node{}\node{}\node{\fontact (f\inducedS(f_{g}\inducedS(U)))}
\end{diagram}
$
\end{center}
In this case, $f\colon\cuco X\rightarrow\cuco X'$ is (uniformly) 
coarsely $\phi$-equivariant in the usual sense. Also, we note for the 
reader that $f_{g}\inducedS\co \mathfrak S \CircleArrowright$, while 
$f_{\phi(g)}\inducedS\co \mathfrak S'\CircleArrowright$, and  
$f\inducedS\co \mathfrak S \to \mathfrak S'$. 
\end{defn}

\section{Tools for studying hierarchically hyperbolic spaces}\label{sec:preliminary}
We now collect some basic consequences of the axioms that are used repeatedly throughout the paper. However, this section need not all be read in advance. Indeed, the 
 reader should feel free to skip this section on a first reading and 
 return to it later when necessary. Throughout this section, we work in a hierarchically hyperbolic space $(\cuco X,\mathfrak S)$.

\subsection{Handy basic consequences of the axioms}\label{lem:exportable}

\begin{lem}[``Finite dimension'']\label{lem:pairwise_orthogonal}
Let $(\cuco X,\mathfrak S)$ be a hierarchically hyperbolic space of complexity $n$ and let $U_1,\ldots,U_k\in\mathfrak S$ be pairwise-orthogonal.  Then $k\leq n$.
\end{lem}

\begin{proof}
By Definition~\ref{defn:space_with_distance_formula}.\eqref{item:dfs_orthogonal}, there exists $W_1\in\mathfrak S$, 
not $\nest$--maximal, so that $U_2,\ldots,U_k\nest W_1$.  Applying 
Definition~\ref{defn:space_with_distance_formula}.\eqref{item:dfs_orthogonal} inductively yields a sequence 
$W_{k-1}\nest W_{k-2}\nest\ldots\nest W_1\nest S$ of distinct elements, where $S$ is $\nest$--maximal, so that 
$U_{i-1},\ldots,U_k\nest W_i$ for $1\leq i\leq k-1$.  Hence $k\leq n$ by 
Definition~\ref{defn:space_with_distance_formula}.\eqref{item:dfs_complexity}.
\end{proof}

\begin{lem}\label{lem:ramsey_corollary}
 There exists $\chi$ so that $|\mathfrak S'|\leq\chi$ whenever $\mathfrak S'\subseteq\mathfrak S$ does not contain a pair of transverse elements.
\end{lem}

\begin{proof}
 Let $\mathfrak S'\subseteq\mathfrak S$ be a collection of pairwise non-transverse elements, and let $n$ be large enough that any collection of pairwise orthogonal (resp. pairwise $\nest$-comparable) elements of $\mathfrak S$ has cardinality at most $n$; the complexity provides such an $n$, by Definition~\ref{defn:space_with_distance_formula}.\eqref{item:dfs_complexity} and Lemma \ref{lem:pairwise_orthogonal}.  By Ramsey's theorem, there exists $N$ so that if $|\mathfrak S'|> N$ then $\mathfrak S'$ contains a collection of elements, of cardinality at least $n+1$, whose elements are either pairwise orthogonal or pairwise $\nest$-comparable. Hence, $|\mathfrak S'|\leq N$.
\end{proof}

\begin{lem}[Consistency for pairs of points]\label{lem:order_1}
Let $x,y\in\cuco X$ and $V,W\in\mathfrak S$ satisfy $V\transverse W$ and $\dist_{ V}(x,y),\dist_{ W}(x,y)>10E$.  Then, up to exchanging $V$ and $W$, we have $\dist_{ V}(x,\rho^W_V)\leq E$ and $\dist_{ W}(y,\rho^V_W)\leq E$.  
\end{lem}

\begin{proof}
If $\dist_V(x,\rho^W_V)>E$, 
then 
Definition~\ref{defn:space_with_distance_formula}.\eqref{item:dfs_transversal} 
implies $\dist_W(x,\rho^V_W)\leq E$.  Then, either $\dist_W(y,\rho^V_W)\leq 
9E$, in which case $\dist_W(x,y)\leq 10E$, a contradiction, or 
$\dist_W(y,\rho^V_W)>E$, in which case consistency implies that 
$\dist_V(y,\rho^W_V)\leq E$.
\end{proof}

\begin{cor}\label{cor:not_far_from_both}
For $x,y,V,W$ as in Lemma~\ref{lem:order_1}, and any $z\in\cuco X$, there exists $U\in\{V,W\}$ such that $\dist_{ U}(z,\{x,y\})\leq10E$.
\end{cor}

\begin{proof}
By Lemma~\ref{lem:order_1}, we may assume that $\dist_V(x,\rho^W_V),\dist_W(y,\rho^V_W)\leq E$.  Suppose that $\dist_W(z,\{x,y\})>10E$.  Then $\dist_W(z,\rho^V_W)>9E$, so that, by consistency, $\dist_V(z,\rho^W_V)\leq E$, whence $\dist_V(z,x)\leq 2E$.
\end{proof}

The following is needed for Theorem~\ref{thm:realization} and in the forthcoming~\cite{DurhamHagenSisto:HHS_III}.

\begin{lem}[Passing large projections up the $\nest$--lattice]\label{lem:nest_progress}
 For every $C\geq0$ there exists $N$ with the following property.  Let
 $V\in\mathfrak S$, let $x,y\in\cuco X$, and let
 $\{S_i\}_{i=1}^{N}\subseteq \mathfrak S_V$ be distinct and satisfy
 $\dist_{S_i}(x,y)\geq E$.  Then there exists $S\in\mathfrak
 S_V$ and $i$ so that $S_i\propnest S$ and $\dist_{S}(x,y)\geq C$.
\end{lem}

\begin{proof}
The proof is by induction on the level $k$ of a $\nest$-minimal $S\in\mathfrak S_V$ into which each $S_i$ is nested. The base case $k=1$ is empty.
 
Suppose that the statement holds for a given $N=N(k)$ when the level
of $S$ is at most $k$.  Suppose further that $|\{S_i\}|\geq N(k+1)$
(where $N(k+1)$ is a constant much larger than $N(k)$ that will be
determined shortly) and there exists a $\nest$-minimal $S\in\mathfrak
S_V$ of level $k+1$ into which each $S_i$ is nested.  There are two
cases.

If $\dist_{\fontact S}(x,y)\geq C$, then we are done.  If not, then
the large link axiom
(Definition~\ref{defn:space_with_distance_formula}.\eqref{item:dfs_large_link_lemma})
yields $K=K(C)$ and $T_1,\dots,T_K$, each properly
nested into $S$ (and hence of level less than $k+1$), so that any
given $S_i$ is nested into some $T_j$.  In particular, if $N(k+1)\geq
KN(k)$, there exists a $j$ so that at least $N(k)$ elements of
$\{S_i\}$ are nested into $T_j$.  By the induction hypothesis and the 
finite complexity axiom
(Definition~\ref{defn:space_with_distance_formula}.\eqref{item:dfs_complexity}), 
we are  done. 
\end{proof}

The next lemma is used in the proof of 
Proposition~\ref{prop:monotone_paths}, on which the existence of 
hierarchy paths (Theorem~\ref{thm:monotone_hierarchy_paths}) relies.  It is used again in Section \ref{sec:coarse_media} to construct coarse media.

\begin{lem}[Centers are consistent]\label{lem:median_cons}
 There exists $\kappa$ with the following property.  Let
 $x,y,z\in\cuco X$.  Let $\tup b=(b_W)_{W\in\mathfrak S}$ be so that
 $b_W$ is a point in $\fontact W$ with the property that there exists
 a geodesic triangle in $\fontact W$ with vertices in
 $\pi_W(x),\pi_W(y),\pi_W(z)$ each of whose sides contains a point
 within distance $\delta$ of $b_W$.  Then $\tup b$ is
 $\kappa$--consistent.
\end{lem}

\begin{proof}
Recall that for $w\in\{x,y,z\}$ the tuple $(\pi_V(w))_{V\in\mathfrak
S}$ is $E$--consistent.  Let $U,V\in\mathfrak S$ be transverse.  Then,
by $E$--consistency, up to exchanging $U$ and $V$ and substituting $z$
for one of $x,y$, we have $\dist_V(x,\rho^U_V),\dist_V(y,\rho^U_V)\leq
E$, so $\dist_V(x,y)\leq 3E$ (recall that the diameter of $\rho^U_V$
is at most $E$).  Since $b_V$ lies at distance $\delta$ from the
geodesic joining $\pi_V(x),\pi_V(y)$, we have
$\dist_V(b_V,\rho^U_V)\leq 3E+\delta$, whence the lemma holds with
$\kappa=3E+\delta$.

Suppose now $U\propnest V$.  If $b_V$ is within distance $10E$ of
$\rho^U_V$, then we are done.  Otherwise, up to permuting $x,y,z$, any
geodesic $[\pi_V(x),\pi_V(y)]$ is $5E$--far from $\rho^U_V$.  By
consistency of $(\pi_W(x))$, $(\pi_W(y))$ and bounded geodesic image 
(Definition~\ref{defn:space_with_distance_formula}.\eqref{item:dfs:bounded_geodesic_image}),
we have $\dist_U(x,y)\leq 10E$, $\diam_U(\rho^V_U(\pi_V(y))\cup
\pi_U(y))\leq E$, and $\diam_U(\rho^V_U(b_V\cup \pi_V(y)))\leq 10E$.
The first inequality and the definition of $b_U$ imply
$\dist_U(b_U,y)\leq 20E$, and taking into account the other
inequalities we get $\diam_U(\rho^V_U(b_V)\cup b_U)\leq 100E$.

Moreover, since $\pi_W(\cuco X)$ is $K$--quasiconvex, and $b_W$ lies
$\delta$--close to a geodesic starting and ending in $\pi_W(\cuco X)$,
we see that $b_W$ lies $(K+\delta)$--close to a point in $\pi_W(\cuco
X)$.  Hence, provided our initial choice of $E$ was sufficiently large
in terms of the constants from
Definition~\ref{defn:space_with_distance_formula}, $\tup b$ is
admissible.
\end{proof}

\subsection{Partially ordering sets of maximal relevant elements of $\mathfrak S$}\label{subsec:partial_order}
In this subsection, we describe a construction used several times in 
this paper, including: in the proof of
realization (Theorem~\ref{thm:realization}), in the construction of
hierarchy paths (Theorem~\ref{thm:monotone_hierarchy_paths}), 
and in the proof of the distance formula
(Theorem~\ref{thm:distance_formula}). 
We expect that this construction will have numerous other applications, as 
is the case with the corresponding partial ordering in the case of the 
mapping class group, see for example 
\cite{BKMM:consistency, 
BehrstockMinsky:RD,ClayLeiningerMangahas:RAAGs}.

Fix $x\in\cuco X$ and a tuple $\tup b\in\prod_{U\in\mathfrak 
S}2^{\image(\pi_U)}$, where the $U$--coordinate $b_U$ is a set of
diameter at most some fixed $\xi\geq0$.  For example, $\tup b$ could
be the tuple $(\pi_U(y))$ for some $y\in\cuco X$.  

In the remainder of this section, we choose $\kappa\geq 0$ and require that $\tup b$ is $\kappa$--consistent.  (Recall that if $\tup b$ is the tuple of projections of a point in $\cuco X$, then $\tup b$ is $E$--consistent.)

\begin{defn}[Relevant]\label{defn:relevant}
First, fix $\theta\geq 100\max\{\kappa,E\}$.  Then $U\in\mathfrak S$
is \emph{relevant} (with respect to $x,\tup b,\theta$) if
$\dist_U(x,b_U)>\theta$.  Denote by $\relevant(x,\tup b,\theta)$ the
set of relevant elements.
\end{defn}

Let $\relevant_{max}(x,\tup b,\theta)$ be a subset of $\relevant(x,\tup b,\theta)$ whose elements are pairwise $\nest$--incomparable (for example, they could all be $\nest$--maximal in $\relevant(x,\tup b,\theta)$, or they could all have the same level).  Define a relation $\preceq$ on $\relevant_{max}(x,\tup b,\theta)$ as follows.  Given $U,V\in\relevant_{max}(x,\tup b,\theta)$, we have $U\preceq V$ if $U=V$ or if $U\transverse V$ and $\dist_U(\rho^V_U,b_U)\leq\kappa$. Figure~\ref{fig:order} illustrates $U\prec V$.

\begin{figure}[h]
 \includegraphics[width=0.7\textwidth]{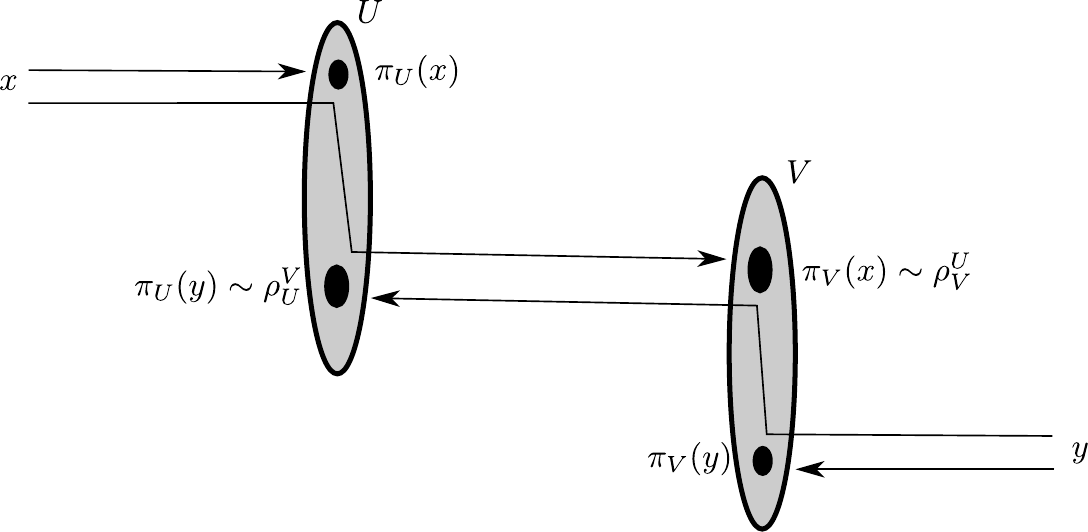}
 \caption{Heuristic picture of $U\prec V$ (for $\tup b$ the coordinates of $y\in\cuco X$, for concreteness). The idea is that ``on the way'' from $x$ to $y$ one ``first encounters'' $U$ and is forced to change the projection from $\pi_U(x)$ to $\pi_U(y)\sim \rho^V_U$. In doing so the projection to $V$ is not affected.}\label{fig:order}
\end{figure}

\begin{prop}\label{prop:partial_order}
The relation $\preceq$ is a partial order.  Moreover, either $U,V$ are $\preceq$--comparable or $U\orth V$.
\end{prop}

\begin{proof}
Clearly $\preceq$ is reflexive.  Antisymmetry follows from Lemma~\ref{lem:antisymmetric}.  Suppose that $U,V$ are $\preceq$--incomparable.  If $U\orth V$, we are done, and we cannot have $U\nest V$ or $V\nest U$, so suppose $U\transverse V$.  Then, by $\preceq$--incomparability of $U,V$, we have $\dist_U(\rho^V_U,b_U)>\kappa$ and $\dist_V(\rho^U_V,b_V)>\kappa$, contradicting $\kappa$--consistency of $\tup b$.  This proves the assertion that transverse elements of $\relevant_{max}(x,\tup b,\theta)$ are $\preceq$--comparable.  Finally, transitivity follows from Lemma~\ref{lem:transitivity}.   
\end{proof}

\begin{lem}\label{lem:antisymmetric}
The relation $\preceq$ is antisymmetric.
\end{lem}

\begin{proof}
If $U\preceq V$ and $U\neq V$, then $\dist_U(b_U,\rho^V_U)\leq\kappa$,
so $\dist_U(x,\rho^V_U)>\theta-\kappa\geq 99\kappa>E$.  Then, 
$\dist_V(x,\rho^U_V)\leq E$, by consistency. Thus 
$\dist_V(b_V,\rho^U_V)>\kappa$, and so, by definition $V\not\preceq U$.
\end{proof}

\begin{lem}\label{lem:transitivity}
The relation $\preceq$ is transitive.
\end{lem}

\begin{proof}
Suppose that $U\preceq V\preceq W$.  If $U=V$ or $V=W$, then $U\preceq W$, and by Lemma~\ref{lem:antisymmetric}, we cannot have $U=W$ unless $U=V=W$.  Hence suppose $U\transverse V$ and $\dist_U(\rho^V_U,b_U)\leq\kappa$, while $V\transverse W$ and $\dist_V(\rho^W_V,b_V)\leq\kappa$.  By the definition of $\relevant_{max}(x,\tup b,\theta)$, we have $\dist_T(x,b_T)>100\kappa$ for $T\in\{U,V,W\}$.

We first claim that $\dist_V(\rho^U_V,\rho^W_V)>10E$.  Indeed, $\dist_U(b_U,\rho^V_U)\leq\kappa$, so $\dist_U(\rho^V_U,x)\geq 90\kappa$, whence $\dist_V(\rho^U_V,x)\leq E\leq\kappa$ by $E$--consistency of the tuple $(\pi_T(x))_{T\in\mathfrak S}$.  On the other hand, $\dist_V(\rho^W_V,b_V)\leq\kappa$, so $\dist_V(\rho^U_V,\rho^W_V)>10E$ as claimed.  Hence, by Lemma~\ref{lem:transverse_far}, we have $U\transverse W$.

Since $\diam(\image(\pi_W))>100\kappa$ --- indeed, $\dist_W(x,b_W)>100\kappa$ and $b_W\in\image(\pi_W(\cuco X))$ --- partial realization (Definition~\ref{defn:space_with_distance_formula}.\eqref{item:dfs_partial_realization}) provides $a\in\cuco X$ satisfying $\dist_W(a,\{\rho^U_W,\rho^V_W\})\geq 10\kappa$.    

We thus have $\dist_U(a,\rho^W_U)\leq E$ by $E$--consistency of 
$(\pi_T(a))_{T\in\mathfrak S}$, and the same is true with $V$ replacing $U$.  Hence $\dist_V(\rho^U_V,a)>E$, so consistency implies $\dist_U(a,\rho^V_U)\leq E$.  Thus $\dist_U(\rho_U^V,\rho_U^W)\leq 2E$.  Thus $\dist_U(b_U,\rho^W_U)\leq 2E+\kappa<10\kappa$, whence $\dist_U(x,\rho^W_U)>50\kappa>E$, so $\dist_W(x,\rho^U_W)\leq E$ by consistency and the fact that $U\transverse W$.  It follows that $\dist_W(b_W,\rho^U_W)\geq 100\kappa-E>\kappa$, so, again by consistency, $\dist_U(b_U,\rho^U_W)\leq\kappa$, i.e., $U\preceq W$.
\end{proof}

\begin{lem}\label{lem:transverse_far}
Let $U,V,W\in\mathfrak S$ satisfy $\diam(\image(\pi_U)),\diam(\image(\pi_V)),\diam(\image(\pi_W))>10E$, and $U\transverse V,W\transverse V$, and $\dist_V(\rho^U_V,\rho^W_V)>10E$.  Suppose moreover that $U$ and $W$ are $\nest$--incomparable.  Then $U\transverse W$.
\end{lem}

\begin{proof}
If $U\orth W$, then by the partial realization axiom  (Definition~\ref{defn:space_with_distance_formula}.\eqref{item:dfs_partial_realization}) and the lower bound on diameters, there exists an $E$--partial realization point $x$ for $\{U,W\}$ so that $$\dist_{  U}(\rho^V_U,x),\dist_{  W}(\rho^V_W,x)>E.$$  This contradicts consistency since $\dist_V(\rho^U_V,\rho^W_V)>10E$; indeed, by consistency $\dist_{  V}(\rho^U_V,x)\leq E,\dist_{  V}(\rho^W_V,x)\leq E$, i.e., $\dist_V(\rho^U_V,\rho^W_V)\leq2E$.  Hence $U\transverse W$.
\end{proof}

\subsection{Coloring relevant elements}\label{subsec:coloring}
In this subsection, the key result is Lemma~\ref{lem:coloring}, 
which we will apply in proving the existence of hierarchy paths
in Section~\ref{subsec:good_and_proper}.  

Fix $x,y\in\cuco X$.  As above, let $\relevant(x,y,100E)$ consist of those $V\in\mathfrak S$ for which $\dist_V(x,y)>100E$. Recall that, given $U\in\mathfrak S$, we denote by $\mathfrak T_U^\ell$ the set of $V\in\mathfrak S_U$ such that $\ell_U-\ell_V=\ell$.  In particular, if $V,V'\in\mathfrak T^\ell_U$ and $V\nest V'$, then $V=V'$.  Let $\relevant_U^\ell(x,y,100E)=\relevant(x,y,100E)\cap\mathfrak T^\ell_U$, the set of $V\nest U$ so that $\dist_V(x,y)>100E$ and $\ell_U-\ell_V=\ell$.

By Proposition~\ref{prop:partial_order}, the relation $\preceq$ on $\relevant^\ell_U(x,y,100E)$ defined as follows is a partial order: $V\preceq V'$ if either $V=V'$ or $\dist_V(y,\rho^{V'}_V)\leq E$.

\begin{defn}[Relevant graph]\label{defn:relevant_graph}
Denote by $\mathcal G$ the graph with vertex-set
$\relevant^\ell_U(x,y,100E)$, with two vertices adjacent if and only
if the corresponding elements of $\relevant^\ell_U(x,y,100E)$ are
orthogonal.  Let $\mathcal G^c$ denote the complementary graph of $\mathcal G$,
i.e., the graph with the same vertices and edges corresponding to
$\preceq$--comparability.
\end{defn}

The next lemma is an immediate consequence of Proposition~\ref{prop:partial_order}:

\begin{lem}\label{lem:incomparability_graph}
Elements of $V,V'\in\relevant^\ell_U(x,y,100E)$ are adjacent in $\mathcal G$ if and only if they are $\preceq$--incomparable. 
\end{lem}

\begin{lem}[Coloring relevant elements]\label{lem:coloring}
Let $\chi$ be the maximal cardinality of a set of pairwise orthogonal elements of $\mathfrak T_U^{\ell}$. Then there exists a $\chi$--coloring of the set of relevant elements of $\mathfrak T_U^{\ell}$ such that non--transverse elements have different colors.
\end{lem}

\begin{proof}
Since each clique in $\mathcal G$ --- i.e., each $\preceq$--antichain in
$\relevant^\ell_U(x,y,100E)$ --- has cardinality at most $\chi$, Dilworth's
theorem~\cite[Theorem~1.1]{Dilworth:poset} implies that $\mathcal G$
can be colored with $\chi$ colors in such a way that
$\preceq$--incomparable elements have different colors; hence
non-transverse elements have different colors.
\end{proof}

\begin{rem}\label{rem:chi_finite}
The constant $\chi$ provided by Lemma~\ref{lem:coloring} is bounded by the complexity of $(\cuco X,\mathfrak S)$, by
Lemma~\ref{lem:ramsey_corollary}.
\end{rem}

\section{Realization of consistent tuples}\label{sec:realization}
The goal of this section is to prove Theorem~\ref{thm:realization}.
In this section we will work with a fixed hierarchical space $(\cuco
X,\mathfrak S)$.  We will use the concepts of consistency and partial realization points; see Definition~\ref{defn:consistent} and Definition~\ref{defn:partial_realization_point}.

\begin{thm}[Realization of consistent tuples]\label{thm:realization}
 For each $\kappa\geq1$ there exist $\theta_e,\theta_u\geq0$ such that
 the following holds.  Let $\tup
 b\in\prod_{W\in\mathfrak S}2^{\fontact W}$ be $\kappa$--consistent;
 for each $W$, let $b_W$ denote the $\fontact W$--coordinate of $\tup
 b$.

 Then there exists $x\in \cuco X$ so that $\dist_{
 W}(b_W,\pi_W(x))\leq\theta_e$ for all $\fontact W\in\mathfrak S$.
 Moreover, $x$ is \emph{coarsely unique} in the sense that the set of
 all $x$ which satisfy $\dist_{ W}(b_W,\pi_W(x))\leq\theta_e$ in each
 $\fontact W\in\mathfrak S$, has diameter at most $\theta_u$.
\end{thm}

\begin{rem}In typical cases, where the $\pi_U$ are uniformly coarsely surjective, the admissibility part of the consistency 
hypothesis is satisfied automatically.\end{rem}

\begin{proof}[Proof of Theorem~\ref{thm:realization}]
The main task is to prove the following claim about a $\kappa$--consistent admissible tuple $\tup b$:

\begin{claim}\label{claim:main_realization_claim}
Let $\{V_j\}$ be a family of pairwise orthogonal elements of $\mathfrak S$, all of level at most $\ell$.  Then there exists $\theta_e=\theta_e(\ell,\kappa)>100E\kappa \alpha$ and pairwise-orthogonal $\{U_i\}$ so that:
 \begin{enumerate}
  \item each $U_i$ is nested into some $V_j$,\label{item:U_i_nested}
  \item for each $V_j$ there exists some $U_i$ nested into it, and\label{item:each_V_j_covered}
  \item any $E$--partial realization point $x$ for $\{U_i\}$ satisfies
  $\dist_{W}(b_W,x)\leq\theta_e$ for each $W\in \mathfrak S$ for which
  there exists $j$ with $W\nest
  V_j$.\label{item:nested_in_V_means_good}
\end{enumerate}
\end{claim}

Applying Claim~\ref{claim:main_realization_claim} when $\ell=\ell_S$,
where $S\in\mathfrak S$ is the unique $\nest$--maximal element, along
with the Partial Realization
axiom~(Definition~\ref{defn:space_with_distance_formula}.\eqref{item:dfs_partial_realization}),
completes the existence proof, giving us a constant $\theta_e$.  If $x,y$ both have the desired
property, then $\dist_V(x,y)\leq 2\theta_e+\kappa$ for all
$V\in\mathfrak S$, whence the uniqueness axiom
(Definition~\ref{defn:space_with_distance_formula}.\eqref{item:dfs_uniqueness})
ensures that $\dist(x,y)\leq\theta_u$, for an appropriate $\theta_u$.  Hence to prove the theorem it 
remains to prove
Claim~\ref{claim:main_realization_claim}, which we do now.

The claim when $\ell=1$ follows from admissibility and the partial realization axiom
(Definition
\ref{defn:space_with_distance_formula}.\ref{item:dfs_partial_realization}),
so we assume that the claim holds for $\ell-1\geq 1$, with
$\theta_e(\ell-1,\kappa)=\theta_e'$, and prove it for level $\ell$.

\textbf{Reduction to the case $|\{V_j\}|=1$:} It suffices to prove the
claim in the case where $\{V_j\}$ has a single element, $V$.  To see 
this, note that once we prove the claim for each $V_j$ separately, yielding a 
collection of 
pairwise-orthogonal sets $\{U_i^j\nest V_j\}$ with the desired
properties, then we take the union of these sets to obtain the claim 
for the collection $\{V_j\}$.

\textbf{The case $\{V_j\}=\{V\}$:} Fix $V\in\mathfrak S$ so that $\ell_V=\ell$.  If for each $x\in\cuco X$ that satisfies $\dist_V(x,b_V)\leq E$ we have $\dist_{W}(b_W,x)\leq100E\kappa\alpha$ for $W\in \mathfrak S_V$,  then the claim follows with $\{U_i\}=\{V\}$. Hence, we can suppose that this is not the case.

We are ready for the main argument, which is contained in Lemma~\ref{lem:orthogonal_induction} below. We will construct $\{U_i\}$ incrementally, using Lemma~\ref{lem:orthogonal_induction}, which essentially says that either we are done at a certain stage or we can add new elements to $\{U_i\}$.

We will say that the collection $\mathfrak U$ of elements of $\mathfrak S_V$ is \emph{totally orthogonal} if any pair of distinct elements of $\mathfrak U$ are orthogonal.
Given a totally orthogonal family $\mathfrak U$ we say that $W\in\mathfrak S_V$ is $\mathfrak U$--\emph{generic} if there exists $U\in\mathfrak U$ so that $W$ is not orthogonal to $U$. Notice that no $W$ is $\emptyset$--generic.

A totally orthogonal collection $\mathfrak U\subseteq \mathfrak S_V$ is $C$--\emph{good} if any $E$--partial realization point $x$ for $\mathfrak U$ has the property that for each $W\in\mathfrak S_V$ we have $\dist_W(x,b_W)\leq C$. (Notice that our goal is to find such $\mathfrak U$.)
A totally orthogonal collection $\mathfrak U\subseteq \mathfrak S_V$ is $C$--\emph{generically good} if any $E$--partial realization point $x$ for $\mathfrak U$ has the property that for each $\mathfrak U$--generic $W\in\mathfrak S_V$ we have $\dist_W(x,b_W)\leq C$ (e.g., for $\mathfrak U=\emptyset$).

We can now quickly finish the proof of the claim using 
Lemma~\ref{lem:orthogonal_induction} about extending generically good 
sets, which we state and prove below.
Start with $\mathfrak U=\emptyset$.  If $\mathfrak U$ is $C$--good for
$C=100E\kappa\alpha$, then we are done.  Otherwise we can apply Lemma
\ref{lem:orthogonal_induction} and get $\mathfrak U_1=\mathfrak U'$ as
in the lemma.  Inductively, if $\mathfrak U_n$ is not $10^nC$--good,
we can apply the lemma and extend $\mathfrak U_n$ to a new totally
orthogonal set $\mathfrak U_{n+1}$.  Since there is a bound on the
cardinality of totally orthogonal sets by Lemma
\ref{lem:pairwise_orthogonal}, in finitely many steps we necessarily
get a good totally orthogonal set, and this concludes the proof of the
claim, and hence of the theorem.
\end{proof}

\begin{lem}\label{lem:orthogonal_induction}
 For every $C\geq 100E\kappa\alpha$ the following holds. Let $\mathfrak U\subseteq \mathfrak S_V- \{V\}$ be totally orthogonal and $C$--\emph{generically good} but not $C$--good. Then there exists a totally orthogonal, $10C$--\emph{generically good} collection $\mathfrak U'\subseteq \mathfrak S_V$ with $\mathfrak U \subsetneq \mathfrak U'$.
\end{lem}

\begin{proof}
Let $x_0$ be an $E$--partial realization point for $\mathfrak U$ so that there exists some $W\nest V$ for 
which 
$\dist_{W}(b_W,x_0)>C$. 

The idea is to try to ``move towards'' $\tup b$ starting from $x_0$, by looking at all relevant elements of $\mathfrak S_V$ that lie between them and finding out which ones are the ``closest'' to $\tup b$.

Let $\mathfrak V_{max}$ be the set of all $W\nest V$ for which:
\begin{enumerate}
 \item \label{item:bad_W_1} $\dist_{W}(b_W,x_0)>C$ and
 \item \label{item:bad_W_2} $W$ is not properly nested into any element of $\mathfrak S_V$ satisfying the above inequality.
\end{enumerate}
We now establish two facts about $\mathfrak V_{max}$.

\textbf{Applying Proposition~\ref{prop:partial_order} to partially order $\mathfrak V_{max}$:}  For $U,U'\in \mathfrak V_{\max}$, write $U\preceq U'$ if either $U=U'$ or $U\transverse U'$ and $\dist_{U}(\rho^{U'}_U, b_U)\leq 10E\kappa$; this is a partial order by Proposition~\ref{prop:partial_order}, which also implies that if $U,U'\in \mathfrak V_{\max}$ are transverse then they are $\preceq$-comparable.  Hence any two $\preceq$--maximal elements of $\mathfrak V_{\max}$ were orthogonal, and we denote by $\mathfrak V'_{max}$ the set of $\preceq$--maximal (hence pairwise-orthogonal) elements of $\mathfrak V_{max}$.

\textbf{Finiteness of $\mathfrak V_{max}$:}  We now show that $|\mathfrak V_{max}|<\infty$.  By Lemma~\ref{lem:ramsey_corollary} and Ramsey's theorem, if $\mathfrak V_{max}$ was infinite then it would contain an infinite subset of pairwise transverse elements, so, in order to conclude that $|\mathfrak V_{max}|<\infty$, it suffices to bound the cardinality of a pairwise-transverse subset of $\mathfrak V_{max}$.

Suppose that $W_1\prec\dots\prec W_{s}\in\mathfrak V_{max}$ are pairwise transverse. By partial realization (Definition 
\ref{defn:space_with_distance_formula}.\eqref{item:dfs_partial_realization}) and admissibility, there exists $z\in\cuco X$ 
such that $\dist_{W_s}(z,b_{W_s})\leq \alpha$ and $\dist_{W_i}(\rho^{W_s}_{W_i},z)\leq \alpha$ for each $i\neq s$, and such 
that $\dist_V(z,\rho^{W_s}_V)\leq \alpha$. By consistency of $\tup b$ and bounded geodesic image, $\rho^{W_s}_V$ has to be 
within distance $10E\kappa$ of a geodesic in $\fontact V$ from $x_0$ to $b_V$. In particular $\dist_V(x_0,z)\leq 
\theta_e'+100E\kappa\alpha + 10E\kappa$. Also, for each $i\neq s$,
\begin{eqnarray*}
\dist_{W_i}(x_0,z) & \geq &\dist_{W_i}(x_0,b_{W_i})-\dist_{W_i}(b_{W_i},\rho^{W_s}_{W_i})-\dist_{W_i}(\rho^{W_s}_{W_i},z)\\ 
& \geq & 100E\kappa\alpha-10E\kappa-\alpha\geq 50E\kappa\alpha\geq50E.
\end{eqnarray*}
Indeed, $\dist_{W_i}(b_{W_i},\rho^{W_s}_{W_i})\leq 10E\kappa$ since $W_i\prec W_s$, while $\dist_{W_i}(\rho^{W_s}_{W_i},z)\leq\alpha$ by our choice of $z$.  Lemma \ref{lem:nest_progress} now provides the required bound on $s$. 

\textbf{Choosing $\mathfrak U'$:}  Since $\ell_{U}<\ell_V$ for all $U\in \mathfrak V'_{max}$, by induction there exists a totally orthogonal set $\{U_i\}$ so that any $E$--partial realization point $x$ for $\{U_i\}$ satisfies $\dist_{T}(b_T,x)\leq\theta_e'$ for each $T\in \mathfrak S$ nested into some $U\in\mathfrak V'_{max}$.  Let $\mathfrak U'=\{U_i\}\cup \mathfrak U$. 

Choose such a partial realization point $x$ and let $W\nest V$ be $\mathfrak U'$--generic. Our goal is to bound
$\dist_W(x,b_W)$, and we will consider 4 cases.

If there exists $U\in\mathfrak U$ that is not orthogonal to $W$, then we are done by hypothesis, since any $E$--partial realization point for $\mathfrak U'$ is also an $E$--partial realization point for $\mathfrak U$.

Hence, from now on, assume that $W$ is orthogonal to each $U\in\mathfrak U$, i.e. $W$ is not $\mathfrak U$--generic. 

If $W\nest U$ for some $U\in \mathfrak V'_{max}$, then we are done by induction.

Suppose that $W\transverse U$ for some $U\in \mathfrak V'_{max}$.    For each $U_i\nest U$ --- and our induction
hypothesis implies that there is at least one such $U_i$ --- we have
$\dist_W(x,\rho^{U_i}_W)\leq E$ since $x$ is a partial
realization point for $\{U_i\}$ and either $U_i\nest W$ or
$U_i\transverse W$ (since $W$ is $\mathfrak U'$--generic but not $\mathfrak U$--generic).
The triangle inequality therefore yields:
$$\dist_W(x,b_W)\leq E+\dist_W(\rho^{U_i}_W,\rho^{U}_W)+\dist_W(b_W,\rho^{U}_W).$$
By Definition~\ref{defn:space_with_distance_formula}.\eqref{item:dfs_transversal}, $\dist_W(\rho^{U_i}_W,\rho^{U}_W)\leq E$, and we will show that $\dist_W(b_W,\rho^{U}_W)\leq 2C$, so that $\dist_W(x,b_W)\leq 2E+2C$.

Suppose, for a contradiction, that $\dist_{W}(b_W,\rho^{U}_W)>2C.$
If $\dist_{U}(\rho^W_{U},x_0)\leq E$, then $$\dist_{U}(\rho^W_{U},b_{U})\geq C-E>\kappa,$$ by consistency, whence $\dist_W(\rho^{U}_W,b_W)\leq\kappa,$ a contradiction.

On the other hand, if $\dist_{U}(\rho^W_{U},x_0)>E$, then $\dist_W(x_0,\rho^{U}_W)\leq E$ by consistency. Hence $\dist_W(x_0,b_W)\geq2C-E.$  Hence there exists a $\nest$--maximal $W'\neq V$ with the property that $W\nest W'\nest V$ and $\dist_{W'}(x_0,b_W)>C$ (possibly $W'=W$).  Such a $W'$ is in $\mathfrak V_{max}$ by definition.     

Since $W\transverse U$, and $W'$ and $U$ are
$\nest$--incomparable, $W'\transverse U$.  Thus $U$ and $W'$ are
$\preceq$--comparable, by Proposition~\ref{prop:partial_order}.  Since
$W'\neq U$ and $U$ is $\preceq$--maximal, we have $W'\preceq U$,
i.e., $\dist_{W'}(b_{W'},\rho^{U}_{W'})\leq10E\kappa$.  Since
$\preceq$ is antisymmetric, by Lemma~\ref{lem:antisymmetric}, we have
$\dist_{U}(b_{U},\rho^{W'}_{U})>10E\kappa$.  Since $\dist_{U}(\rho^{W}_{U},\rho^{W'}_{U})\leq E$ (from Definition~\ref{defn:space_with_distance_formula}.\eqref{item:dfs_transversal}), we have
$\dist_{U}(b_{U},\rho^W_{U})>10E\kappa-E>\kappa$, since $E\geq1$, so, by consistency,
$\dist_W(b_W,\rho^{U}_W)\leq\kappa$, a contradiction.

Finally, suppose $U\propnest W$ for some $U\in\mathfrak V'_{max}$. In this case, by $\nest$--maximality of $U$, we have $\dist_W(x_0,b_W)\leq C$. Also, $\dist_W(x,\rho^{U_i}_W)\leq E$ for any $U_i\nest U$ since $x$ is a partial realization point, so that $\dist_W(x,\rho^{U}_W)\leq 2E$, since $\dist_W(\rho^{U}_W,\rho^{U_i}_W)\leq E$ by Definition~\ref{defn:space_with_distance_formula}.\eqref{item:dfs_transversal}.  If $\dist_W(x,b_W)>2C$, then we claim $\dist_{U}(x_0,b_{U})\leq 10E\kappa$, a contradiction.  Indeed, any geodesic in $\fontact W$ from $\pi_W(x_0)$ to $b_W$ does not enter the $E$--neighborhood of $\rho^{U}_W$.  By bounded geodesic image, $\diam_{U}(\rho^W_{U}(\pi_W(x_0))\cup\rho^W_{U}(b_W))\leq E$ and by consistency, $\diam_{U}(\rho^W_{U}(\pi_W(x_0))\cup\pi_{U}(x_0))\leq E$ and $\diam_{U}(\rho^W_{U}(b_W)\cup b_{U})\leq\kappa$, and we obtain the desired bound on $\dist_{U}(x_0,b_{U})$.  This completes the proof of the lemma.
\end{proof}

\section{Hierarchy paths and the distance formula}\label{sec:hier_path_df}
Throughout this section, fix a hierarchically hyperbolic space $(\cuco X,\mathfrak S)$.  

\subsection{Definition of hierarchy paths and statement of main theorems}\label{subsec:dist_formula_state}
Our goal is to deduce the existence of hierarchy paths (Theorem~\ref{thm:monotone_hierarchy_paths}) from the other axioms and to prove the distance formula (Theorem~\ref{thm:distance_formula}).

\begin{defn}[Quasigeodesic, unparameterized quasigeodesic]\label{defn:quasi_geodesic}
A \emph{$(D,D)$--quasigeodesic} in the metric space $M$ is a
$(D,D)$--quasi-isometric embedding $f\co[0,\ell]\to M$; we allow $f$ to
be a coarse map, i.e., to send points in $[0,\ell]$ to uniformly
bounded sets in $M$.  A (coarse) map $f\co[0,\ell]\to M$ is a
\emph{$(D,D)$--unparameterized quasigeodesic} if there exists a
strictly increasing function $g\co[0,L]\to[0,\ell]$ such that $g(0)=f(0),g(L)=f(\ell)$, and $f\circ
g\co[0,L]\to M$ is a $(D,D)$--quasigeodesic and for each
$j\in[0,L]\cap\naturals$, we have $\diam_M\left(f(g(j))\cup
f(g(j+1))\right)\leq D$.
\end{defn}

\begin{defn}[Hierarchy path]\label{defn:hierarchy_path}
For $D\geq 1$, a (not necessarily continuous) path $\gamma\co[0,\ell]\to\cuco X$ is a \emph{$D$--hierarchy path} if
 \begin{enumerate}
  \item $\gamma$ is a $(D,D)$-quasigeodesic,
  \item for each $W\in\mathfrak S$, the path $\pi_W\circ\gamma$ is an unparameterized $(D,D)$--quasigeodesic.
\end{enumerate}
\end{defn}

\begin{notation}\label{notation:asymp_ignore}
Given $A,B\in\reals$, we denote by $\ignore{A}{B}$ the quantity which is $A$ if $A\geq B$ and $0$ otherwise.  Given $C,D$, we write $A\asymp_{C,D}B$ to mean $C^{-1}A-D\leq B\leq CA+D$.
\end{notation}

\begin{thm}[Existence of Hierarchy Paths]\label{thm:monotone_hierarchy_paths}
Let $(\cuco X,\mathfrak S)$ be hierarchically hyperbolic. Then there exists $D_0$ so that any $x,y\in\cuco X$ are joined by a $D_0$-hierarchy path. 
\end{thm}

\begin{thm}[Distance Formula]\label{thm:distance_formula}
 Let $(X,\mathfrak S)$ be hierarchically hyperbolic. Then there exists $s_0$ such that for all $s\geq s_0$ there exist
 constants $K,C$ such that for all $x,y\in\cuco X$,
 $$\dist_{\cuco X}(x,y)\asymp_{(K,C)}\sum_{W\in\mathfrak S}\ignore{\dist_{ W}(x,y)}{s}.$$
\end{thm}

The proofs of the above two theorems are intertwined, and we give the proof immediately below.  This relies on several lemmas, namely Lemma~\ref{lem:proper_path}, proved in Section~\ref{subsec:good_and_proper}, and Lemmas~\ref{lem:df_l} and~\ref{lem:df_u}, proved in Section~\ref{subsec:df_ul}.

\begin{proof}[Proof of Theorems~\ref{thm:distance_formula} and~\ref{thm:monotone_hierarchy_paths}]
The lower bound demanded by Theorem~\ref{thm:distance_formula} is given by Lemma~\ref{lem:df_l} below. By Lemma~\ref{lem:proper_path} and Lemma~\ref{lem:df_u}, there is a monotone path (see Definition~\ref{defn:monotone}) whose length realizes the upper bound on $\dist_{\cuco X}(x,y)$, and the same holds for any subpath of this path, which is therefore a hierarchy path, proving Theorem~\ref{thm:monotone_hierarchy_paths} and completing the proof of Theorem~\ref{thm:distance_formula}.
\end{proof}

\subsection{Good and proper paths: definitions}\label{defn:good_proper_definition}
We now define various types of (non-continuous) paths in $\cuco X$ that will appear on the way to hierarchy paths.

\begin{defn}[Discrete path]\label{defn:discrete_path}
A \emph{$K$--discrete path} is a map $\gamma\co I\to\cuco X$, where
$I$ is an interval in $\mathbb Z$ and $\dist_{\cuco
X}(\gamma(i),\gamma(i+1))\leq K$ whenever $i,i+1\in I$.  The
\emph{length} $|\alpha|$ of a discrete path $\alpha$ is $\max I-\min
I$.
\end{defn}

\begin{defn}[Efficient path]\label{defn:efficient}
A discrete path $\alpha$ with endpoints $x,y$ is \emph{$K$--efficient} if $|\alpha|\leq K\dist_{\cuco X}(x,y)$.
\end{defn}

\begin{defn}[Monotone path]\label{defn:monotone}
Given $U\in\mathfrak S$, a $K$--discrete path $\alpha$ and a constant
$L$, we say that $\alpha$ is \emph{$L$--monotone in $U$} if whenever
$i\leq j$ we have $\dist_{ U}(\alpha(0),\alpha(i))\leq \dist_{
U}(\alpha(0),\alpha(j))+L$. A path which is $L$--monotone in $U$ for 
all $U\in \mathfrak S$ is said to be \emph{$L$--monotone}.
\end{defn}

\begin{defn}[Good path]\label{defn:discrete_efficient_monotone}  
A $K$--discrete path that is $L$--monotone in $U$ is said to be $(K,L)$--\emph{good for $U$}.  Given $\mathfrak S'\subseteq\mathfrak S$, a path $\alpha$ that is $(K,L)$--good for each $V\in\mathfrak S'$ is \emph{$(K,L)$--good} for $\mathfrak S'$.
\end{defn}

\begin{defn}[Proper path]\label{defn:discrete_and_proper}
A discrete path $\alpha\co\{0,\ldots,n\}\rightarrow\cuco X$ is 
\emph{$(r,K)$--proper} if for $0\leq i<n-1$, we have $\dist_{\cuco X}(\alpha(i),\alpha(i+1))\in[r,r+K]$ and $\dist_{\cuco X}(\alpha(n-1),\alpha(n))\leq r+K$.  
Observe that $(r,K)$-properness is preserved by passing to subpaths.
\end{defn}

\subsection{Good and proper paths: existence}\label{subsec:good_and_proper}
Our goal in this subsection is to join points in $\cuco X$ with proper paths, i.e., to prove Lemma~\ref{lem:proper_path}.  This relies on the much more complicated Proposition~\ref{prop:monotone_paths}, which produces good paths (which are then easily made proper).

\begin{lem}\label{lem:proper_path}
There exists $K$ so that for any $r\geq 0$, any $x,y\in\cuco X$ are joined by a $K$-monotone, $(r,K)$--proper discrete path.
\end{lem}

\begin{proof}
Let $\alpha_0\co\{0,\dots,n_0\}\rightarrow\cuco X$ be a $K$--monotone,
$K$--discrete path joining $x,y$, which exists by
Proposition~\ref{prop:monotone_paths}.  We modify $\alpha_0$ to obtain
the desired path in the following way.  Let $j_0=0$ and, proceeding
inductively, let $j_i$ be the minimal $j\leq n$ such that either
$\dist_{\cuco X}(\alpha_0(j_{i-1}),\alpha_0(j))\in[r,r+K]$ or $j=n$.
Let $m$ be minimal so that $j_m=n$ and define
$\alpha\co\{0,\ldots,m\}\to\cuco X$ by $\alpha(j)=\alpha_0(i_j)$.  The
path $\alpha$ is $(r,K)$-proper by construction; it is easily checked
that $K$--monotonicity is not affected by the above modification; the
new path is again discrete, although for a larger discreteness 
constant.
\end{proof}

It remains to establish the following proposition, whose proof is postponed until the end of this section, after several preliminary statements have been obtained.

\begin{prop}\label{prop:monotone_paths}
There exists $K$ so that any $x,y\in\cuco X$ are joined by path that is $(K,K)$--good for each $U\in\mathfrak S$.
\end{prop}

\begin{defn}[Hull of a pair of points]\label{defn:hull_pair}
 For each $x,y\in \cuco X,\theta\geq 0$, let $H_{\theta}(x,y)$ be the set of all $p\in\cuco X$ so that, for each 
$W\in\mathfrak S$, the set $\pi_W(p)$ lies at distance at most $\theta$ from a geodesic in $\fontact W$ joining $\pi_W(x)$ to 
$\pi_W(y)$. Note that $x,y\in H_{\theta}(x,y)$. 
\end{defn}

\begin{rem}The notion of a hull is generalized in Section~\ref{sec:convex_hulls} to hulls of arbitrary finite sets, but we require only the version for pairs of points in this section.\end{rem}

\begin{lem}[Retraction onto hulls]\label{lem:2_point_retract}
There exist $\theta,K\geq 0$ such that, for each $x,y\in\cuco X$, there exists a $(K,K)$--coarsely Lipschitz map $r:\cuco 
X\to H_{\theta}(x,y)$ that restricts to the identity on $H_\theta(x,y)$.
\end{lem}

\begin{proof}
Let $\kappa$ be the constant from Lemma~\ref{lem:median_cons}, let
$\theta_e$ be chosen as in the realization theorem
(Theorem~\ref{thm:realization}), and let $p\in \cuco X -
H_{\theta_e}(x,y)$.  Define a tuple $\tup
b=(b^p_W)\in\prod_{W\in\mathfrak S}2^{\fontact W}$ so that $b^p_W$ is
on a geodesic in $\fontact W$ from $\pi_W(x)$ to $\pi_W(y)$ and is
within distance $\delta$ of the other two sides of a triangle with
vertices in $\pi_W(x),\pi_W(y),\pi_W(p)$.  By
Lemma~\ref{lem:median_cons}, this is a consistent tuple.  Hence, by
the realization theorem (Theorem~\ref{thm:realization}), there exists
$r(p)\in H_{\theta_e}(x,y)$ so that $\dist_{ W}(\pi_W(r(p)),b^p_W)\leq
\theta_e$.  For $p\in H_{\theta_e}(x,y)$, let
$r(p)=p$.

To see that $r$ is coarsely Lipschitz, it suffices to bound $\dist_{\cuco X}(r(p),r(q))$ when $p,q\in\cuco X$ satisfy $\dist_{\cuco X}(p,q)\leq 1$. For such $p,q$ we have $\dist_{ W}(b^p_W,b^q_W)\leq 100E$, so that Theorem~\ref{thm:realization} implies $\dist_{\cuco X}(r(p),r(q))\leq \theta_u(100E)$, as required.
\end{proof}

\begin{cor}\label{cor:paths_in_hulls}
 There exist $\theta,K\geq0$ such that, for each $x,y\in\cuco X$, there exists a $K$--discrete and $K$--efficient path that lies in $H_{\theta}(x,y)$ and joins $x$ to $y$.
\end{cor}

\begin{proof}
We can assume that $\dist_{\cuco X}(x,y)\geq 1$.
Since $\cuco X$ is a quasigeodesic space, there exists $C=C(\cuco X)\geq 1$
and a $(C,C)$--quasi-isometric embedding $\gamma\co[0,L]\to\cuco X$ with
$\gamma(0)=x,\gamma(L)=y$.  Let $\rho$ be the path obtained by
restricting $r\circ\gamma\co[0,L]\to H_{\theta}(x,y)$ to
$[0,L]\cap\naturals$, where $r$ is the retraction obtained in 
Lemma~\ref{lem:2_point_retract}.  Then $\dist_{\cuco X}(\rho(i),\rho(i+1))\leq
10KC$ since $r$ is $(K,K)$--coarsely Lipschitz and $\gamma$ is
$(C,C)$--coarsely Lipschitz, i.e., $\rho$ is $10KC$--discrete.
Finally, $\rho$ is efficient because $L\leq C\dist_{\cuco X}(x,y)+C\leq 2C \dist_{\cuco X}(x,y)$.
\end{proof}

The efficiency part of the corollary is used in Lemma \ref{lem:df_l}.

\subsubsection{Producing good paths}\label{subsubsec:monotone_paths_proof}
We will need the following lemma, which is a special case of Proposition \ref{prop:HullProperties}.\ref{item:subhull}. We give a proof in the interest of a self-contained exposition.

\begin{lem}\label{lem:nested_hulls}
 For any $\theta_0$ there exists a constant $\theta$ such
 that for every $x,y\in \cuco X$ and every $x',y'\in H_{\theta_0}(x,y)$,
 we have $H_{\theta_0}(x',y')\subseteq H_{\theta}(x,y)$.
\end{lem}

\begin{proof}
For any $z\in H_{\theta_0}(x',y')$ and $W\in\mathfrak S$ the projection
$\pi_{W}(z)$ lies $2(\delta+\theta_0)$--close to a geodesic in $\fontact
W$ from $\pi_W(x)$ to $\pi_W(y)$, by a thin quadrilateral argument.
\end{proof}

We now prove the main proposition of this subsection.

\begin{proof}[Proof of Proposition~\ref{prop:monotone_paths}]
 Recall that, for $\ell\geq0$ and $U\in\mathfrak S$, the set $\mathfrak S^\ell_U$ consists of those $V\in\mathfrak S_U$ with $\ell_U-\ell_V\leq\ell$, and that $\mathfrak T_U^\ell$ consists of those $V\in\mathfrak S_U$ with $\ell_U-\ell_V=\ell$.

We prove by induction on $\ell$ that there exist $\theta,K$ such that
for any $\ell\geq 0$, $x,y\in\cuco X$ and $U\in\mathfrak S$, there is
a path $\alpha$ in $H_{\theta}(x,y)$ connecting $x$ to $y$ such that
$\alpha$ is $(K,K)$--good for $\mathfrak S_U^\ell$.  It then follows 
that for any $x,y\in\cuco X$, there exists a path $\alpha$ in
$H_{\theta}(x,y)$ connecting $x$ to $y$ such that $\alpha$ is
$(K,K)$--good for $\mathfrak S$; this latter statement directly implies
the proposition.

For $a,b\in\cuco X$, denote by
 $[a,b]_W$ a geodesic in $\fontact W$ from $\pi_W(a)$ to $\pi_W(b)$.  Fix $U\in\mathfrak S$.  

\textbf{The case $\ell=0$:} In this case, $\mathfrak S_U^0=\{U\}$.  By
Corollary~\ref{cor:paths_in_hulls}, there exists $\theta_0,K$ and a
$K$--discrete, $K$--efficient path $\alpha_0'\colon\{0,\ldots,k\}\rightarrow
H_{\theta_0}(x,y)$ joining $x$ to $y$.

Similarly, for each $x',y'\in H_{\theta_0}(x,y)$ there exists a $K$--discrete path
$\beta$ contained in $H_{\theta_0}(x',y')$, joining $x'$ to $y'$, and recall that $H_{\theta_0}(x',y')$ is contained in $H_{\theta}(x,y)$ for a suitable $\theta$ in view of Lemma \ref{lem:nested_hulls}.

We use the term \emph{straight path} to refer to a 
path, such as $\beta$, which for each $V\in\mathfrak S$ projects 
uniformly close to a geodesic of $\fontact(V)$.

We now fix $U\in\mathfrak S$, and, using the observation in the last 
paragraph explain how to 
modify $\alpha_0'$ to obtain a $K$--discrete path $\alpha_0$ in
$H_\theta(x,y)$ that is $K$--monotone in $U$; the construction will 
rely on replacing problematic subpaths with straight paths.

A point $t\in\{0,\dots,k\}$ is a \emph{$U$--omen} if there exists
$t'>t$ so that $\dist_{ U}(\alpha_0'(0),\alpha_0'(t))> \dist_{
U}(\alpha_0'(0),\alpha_0'(t'))+5KE$.  If $\alpha_0'$ has no
$U$--omens, then we can take $\alpha_0=\alpha_0'$, so suppose that
there is a $U$--omen and let $t_0$ be the minimal $U$--omen, and let
$t'_0>t_0$ be maximal so that $\dist_{U}(\alpha_0'(0),\alpha_0'(t_0))> \dist_{ U}(\alpha'_{0}(0),\alpha'_{0}(t'_0))$.
Inductively define $t_j$ to be the minimal $U$--omen with $t_j\geq
t'_{j-1}$, if such $t_j$ exists; and when $t_{j}$ exists, we 
define $t'_{j}$ to be maximal in $\{0,\dots,k\}$ satisfying 
$\dist_U(\alpha'_0(0),\alpha'_0(t_j))>\dist_U(\alpha'_0(0),\alpha'_0(t_j'))$.  For each $j\geq0$, let
$x'_j=\alpha_0'(t_j)$ and let $y'_j=\alpha_0'(t'_j)$. See Figure~\ref{fig:omen}.

\begin{figure}[h]
\begin{overpic}[width=0.5\textwidth]{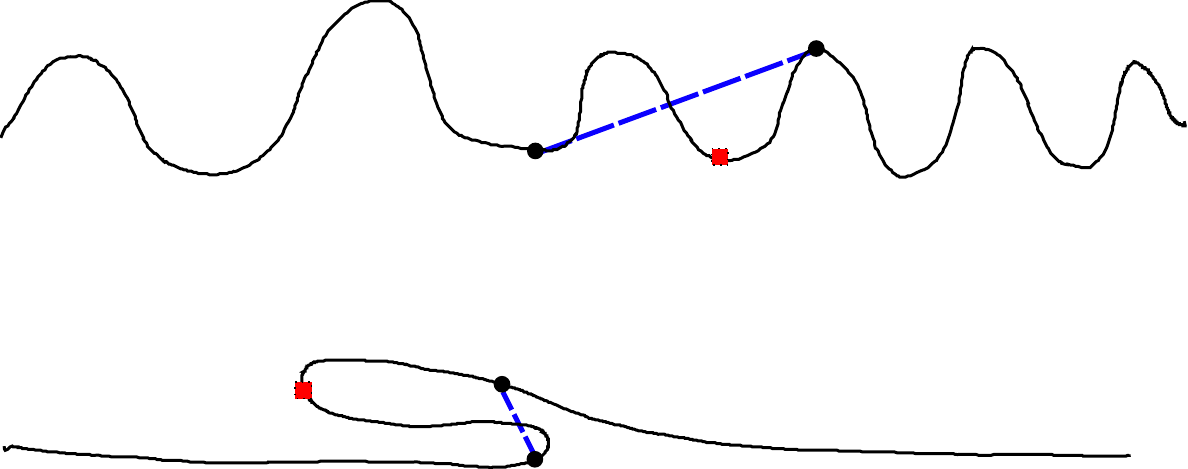}
\put(90,20){$y$}
\put(93,20){$\rightarrow$}
\put(9,20){$x$}
\put(4,20){$\leftarrow$}
\put(40,22){$\alpha_0'(t_j)$}
\put(65,38){$\alpha_0'(t'_j)$}
\end{overpic}
\caption{The picture above shows part of $\alpha_0'$ in $\cuco X$, and
that below shows its projection to $U$.  The point $t_j$ is an omen,
as witnessed by the point marked with a square.  Inserting the dashed
path $\beta_j$, and deleting the corresponding subpath of $\alpha_0'$,
makes $t_j$ cease to be an omen.}\label{fig:omen}
\end{figure}

For each $j$, there exists a $K$--discrete path $\beta_j$ which lies in 
$H_{\theta_0}(x_{j}',y_{j}')\subseteq H_\theta(x,y)$ and is a straight
path from $x'_j$ to $y'_j$.
Let $\alpha_0$ be obtained from
$\alpha_0'$ by replacing each $\alpha_0'([t_j,t'_j])$ with
$\beta_j$.  Clearly, $\alpha_0$ connects $x$ to $y$, is $K$--discrete,
and is contained in $H_{\theta}(x,y)$.  
For each $j$ we have that $\diam_{\fontact U}(\beta_{j})\leq 
\dist_{U}(x_{j}',y_{j}')+2\theta_0$.

Notice that $\dist_{U}(x_{j}',y_{j}')<2KE+10\theta_0$. In fact, since $\alpha'_0(0),\alpha'_0(t_j),\alpha'_0(t'_j)$ lie $\theta_0$-close to a common geodesic and $\dist_U(\alpha'_0(0),\alpha'_0(t_j))\geq\dist_U(\alpha'_0(0),\alpha'_0(t'_j))$, we would otherwise have
$$\dist_U(\alpha'_0(0),\alpha'_0(t_j))-\dist_U(\alpha'_0(0),\alpha'_0(t'_j))\geq \dist_{U}(x_{j}',y_{j}')-5\theta_0\geq 2KE+\theta_0.$$

However, $\dist_U(\alpha'_0(t_j),\alpha'_0(t_j+1))\leq 2KE$ because of $K$--discreteness and the projection map to $\contact U$ being $E$-coarsely Lipschitz. Hence, the inequality above implies
$$\dist_U(\alpha'_0(0),\alpha'_0(t_j))>\dist_U(\alpha'_0(0),\alpha'_0(t'_j))+2KE\geq \dist_U(\alpha'_0(t_j),\alpha'_0(t_j+1)),$$
which contradicts the maximality of $t'_j$.
(Notice that $t'_j\neq k$, and hence $t'_j+1\in\{0,\dots,k\}$ because $\dist_U(\alpha'_0(0),\alpha'_0(t'_j))+\theta_0<\dist_U(\alpha'_0(0),\alpha'_0(t_j))\leq \dist_U(\alpha'_0(0),\alpha'_0(k))+\theta_0$.)

In particular, we get $\diam_{\fontact U}(\beta_{j})\leq 2KE+12\theta_0$, and it is then easy to check that $\alpha_0$ is $\max\{5KE,2KE+12\theta_0\}$--monotone in $U$. By replacing $K$ with $\max\{5KE,2KE+12\theta_0\}$, we thus have a $K$--discrete
path $\alpha_0\subset H_\theta(x,y)$ that joins $x,y$ and is
$K$--monotone in $U$.

We now proceed to the inductive step. Specifically, we fix $\ell\geq 0$ and we assume there 
exist $\theta_{ind},K$ such that there is
a path $\alpha$ in $H_{\theta_{ind}}(x,y)$ connecting $x$ to $y$ such that
$\alpha$ is $(K,K)$--good for $\mathfrak S_U^{\ell-1}$.    

\textbf{The coloring:} For short, we will say that $V\in \mathfrak S$ is $A$--relevant if $\dist_{U}(x,y)\geq A$, see Definition \ref{defn:relevant}. Notice that to prove that a path in $H_\theta(x,y)$ is monotone, it suffices to restrict our attention to only those $W\in\mathfrak S$ which are, say, $10KE$--relevant.

By Lemma~\ref{lem:coloring}, there exists
$\chi\geq 0$, bounded by the complexity of $\cuco X$, and
a $\chi$--coloring $c$ of the $10KE$--relevant elements of $\mathfrak
T_U^\ell$ such that $c(V)=c(V')$ only if $V\transverse V'$. In other words, the set of $10KE$--relevant elements of $\mathfrak
T_U^\ell$ has the form
$\bigsqcup_{i=0}^{\chi-1} c^{-1}(i)$, where $c^{-1}(i)$ is a set
of pairwise-transverse relevant elements of $\mathfrak T_U^\ell$.

\textbf{Induction hypothesis:} Given $p<\chi-1$, assume by
induction (on $\ell$ and $p$) that there exist $\theta_p\geq \theta_{ind},K_p\geq K,$
independent of $x,y,U$, and a path $\alpha_p\co \{0,\dots,k\}\to
H_{\theta_p}(x,y)$, joining $x,y$, that is $(K_{p},K_{p})$--good for
$\bigsqcup_{i=0}^{p}c^{-1}(i)$ and good for $\mathfrak S_U^{\ell-1}$. 

\textbf{Resolving backtracks in the next color:}  Let $\theta_{p+1}$ be provided by Lemma \ref{lem:nested_hulls} with input $\theta_p$. We will modify $\alpha_p$ to construct a $K_{p+1}$--discrete path $\alpha_{p+1}$ in $H_{\theta_{p+1}}(x,y)$, for some $K_{p+1}\geq K_p$, that joins $x,y$ and is $(K_{p+1},K_{p+1})$--good in $\bigsqcup_{i=0}^{p+1}c^{-1}(i)\cup\mathfrak S^{\ell-1}_U$.

Notice that we can restrict our attention to the set $\mathcal C_{p+1}$ of $100(K_pE+\theta_p)$--relevant elements of $c^{-1}(p+1)$.

A point $t\in\{0,\dots,k\}$ is a \emph{$(p+1)$--omen} if there exists
$V\in\calC_{p+1}$ and $t'>t$ so that $\dist_{
V}(\alpha_p(0),\alpha_p(t))> \dist_{ 
V}(\alpha_p(0),\alpha_p(t'))+5K_pE$.
If $\alpha_p$ has no $(p+1)$--omens, then we can take
$\alpha_{p+1}=\alpha_p$, since $\alpha_p$ is good in each $V$ with
$c(V)<p+1$.  Therefore, suppose that there is a $(p+1)$--omen, let
$t_0$ be the minimal $(p+1)$--omen, witnessed by $V_0\in \calC_{p+1}$.
We can assume that $t_0$ satisfies $\dist_{ V_0}(\{x,y\},\alpha_p(t_0))> 10K_pE$.  Let $t'_0>t_0$ be maximal so that $\dist_{
V_0}(\alpha_p(0),\alpha_p(t_0))> \dist_{
V_0}(\alpha_p(0),\alpha_p(t'_0))$.  In particular $\dist_{
V_0}(y,\alpha_p(t_0'))\geq 10 E$.

Let $x'_0=\alpha_0(t_0)$ and $y'_0=\alpha_0(t'_0)$. Inductively, define $t_j$ as the minimal $(p+1)$--omen, witnessed by $V_j\in \calC_{p+1}$, with $t_j\geq t'_{j-1}$, if such $t_j$ exists and let $t'_j$ be maximal so that $\dist_{ V_j}(\alpha_p(0),\alpha_p(t_j))> \dist_{ V_j}(\alpha_p(0),\alpha_p(t'_j))$ and $\dist_{ V_j}(y,\alpha_p(t'_j))> 10 E$. We can assume that $t_j$ satisfies $\dist_{ V_j}(\{x,y\},\alpha_p(t_j))> 10K_pE$. Also, let $x'_j=\alpha_p(t_j),y'_j=\alpha_p(t'_j)$.

Let $\beta_j$ be a path in $H_{\theta_p}(x'_j,y'_j)$ joining $x'_j$ to
$y'_j$ \emph{that is $(K_p,K_p)$--good for each relevant $V$ with
$c(V)\leq p$} and each relevant $V\in\mathfrak S_U^{\ell-1}$.  Such paths can be constructed by induction.  By
Lemma \ref{lem:nested_hulls} $\beta_j$ lies in $H_{\theta_{p+1}}(x,y)$.  Let
$\alpha_{p+1}$ be obtained from $\alpha_p$ by replacing each
$\alpha_p(\{t_j,\dots,t'_j\})$ with $\beta_j$.  Clearly,
$\alpha_{p+1}$ connects $x$ to $y$, is $K_{p}$--discrete, and is contained
in $H_{\theta_{p+1}}(x,y)$.

We observe that the same argument as in the case $\ell=0$ gives $\dist_{V_j}(x'_j,y'_j)\leq 2K_pE+10\theta_p$.

\textbf{Verification that $\alpha_{p+1}$ is good for current colors:}  We next check that each $\beta_j$ is $10^3(K_pE+\theta_p)$--monotone in each $W\in\bigsqcup_{i=0}^{p+1}c^{-1}(i)$. We have to consider the following cases. (We can and shall assume below $W$ is $100(K_pE+\theta_p)$--relevant.)
\begin{itemize}
 \item If $W\nest V_j$, then $W=V_j$, since $\ell_W=\ell_{V_j}$.
 Since the projections on $\fontact W$ of the endpoints of the
 straight path $\beta_j$ coarsely coincide, $\beta_j$ is
 $(2K_pE+12\theta_p)$--monotone in $W$. (See the case $\ell=0$.)
 \item Suppose $V_j\propnest W$.  We claim that the
 projections of the endpoints of $\beta_j$ lie at a uniformly bounded 
 distance in $\fontact W$.  
 
 We claim that $\rho^{V_j}_W$ has to be $E$--close to either $[x,x'_j]_W$ or $
 [y'_j,y]_W$. In fact, if this was not the case,  we would have
 $$\dist_{V_j}(x,y)\leq \dist_{V_j}(x,x'_j)+\dist_{V_j}(x'_j,y'_j)+\dist_{V_j}(y'_j,y)\leq 2E+2K_pE+10\theta_p,$$
 where we applied bounded geodesic
 image
 (Definition~\ref{defn:space_with_distance_formula}.\eqref{item:dfs:bounded_geodesic_image}) to the first and last terms.
 
This is a contradiction with $V_j$ being $100(K_pE+\theta_p)$--relevant.

Suppose for a contradiction that $\dist_{W}(x'_j,y'_j)\geq 500(K_pE+\theta_p)$.  Suppose first that
 $\rho^{V_j}_W$ is $E$--close to $[x,x'_j]_W$.  Then, by monotonicity,
 $\rho^{V_j}_W$ is $E$--far from $[\alpha_p(t'_j),y]_W$.  By bounded
 geodesic image, this contradicts $\dist_{V_j}(y,\alpha_p(t'_j))\geq
 10E$.  If instead $\rho^{V_j}_W$ is $E$--close to $[y'_j,y]_W$, then by
 bounded geodesic image we have
 $\dist_{V_j}(x,\alpha_p(t_j))\leq E$, contradicting that $t_j$ is
 an omen witnessed by $V_j$.  See Figure~\ref{fig:in_W_nested}.
 
  \begin{figure}[h]
  \begin{overpic}[width=0.5\textwidth]{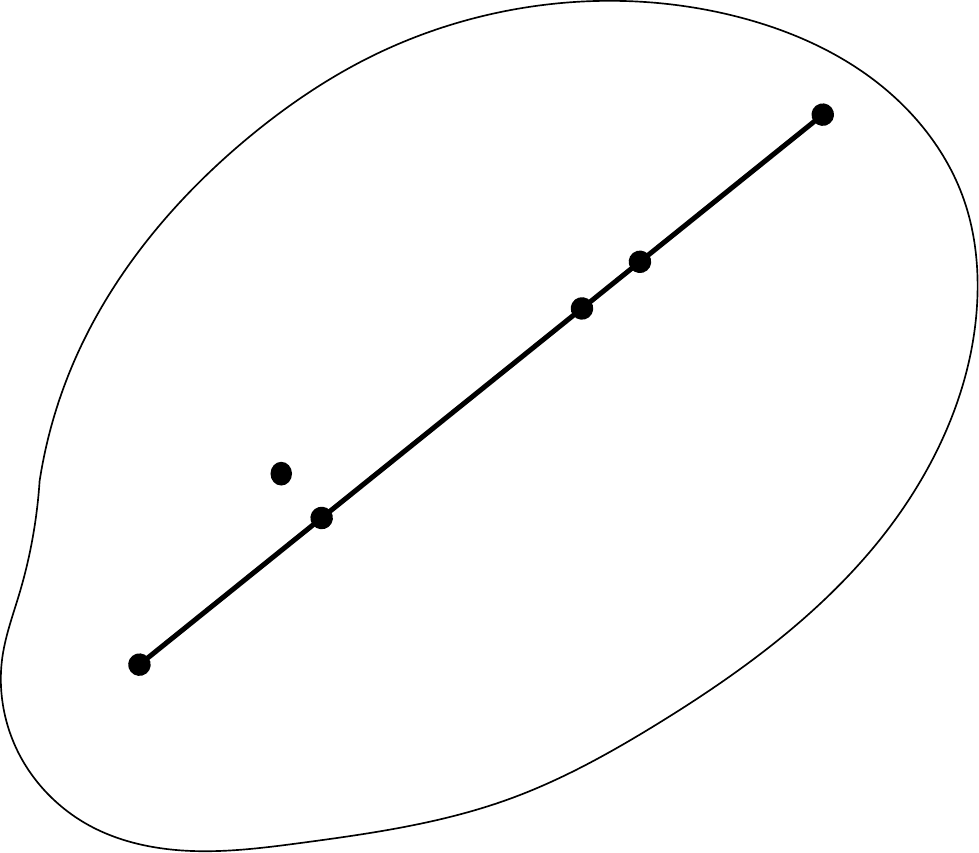}
   \put(80,10){$\fontact W$}
   \put(10,17){$x$}
   \put(86,75){$y$}
   \put(24,42){$\rho^{V_j}_W$}
   \put(35,32){$x'_j$}
   \put(40,59){$\alpha_0(t'_j-1)$}
   \put(70,60){$y'_j=\alpha_0(t'_j)$}
  \end{overpic}

 \caption{The situation in $\fontact W$.}\label{fig:in_W_nested}
 \end{figure}
 
 Hence $\dist_{W}(x'_j,y'_j)\leq 500(K_pE+\theta_p)$ and $\beta_j$ is $10^3(K_pE+\theta_p)$--monotone in $W$.
 
 \item Suppose $W\transverse V_j$.  We again claim that the
 projections of the endpoints of $\beta_j$ are uniformly close in
 $\fontact W$, by showing that they both coarsely
 coincide with $\rho^{V_j}_W$.
 Since $V_{j}$ is relevant, either 
 $\dist_{V_j}(x,\rho^{W}_{V_j})\geq E$ or $\dist_{V_j}(y,\rho^{W}_{V_j})\geq E$. 
 Thus, by consistency, 
 $\dist_{W}(\rho^{V_j}_{W},\{x,y\})\leq E$.  Suppose for a 
 contradiction, that $\dist_{W}(x'_j,y'_j)>100(K_pE+\theta_p)$.  We consider separately the cases where $\dist_W(x,\rho^{V_j}_W)\leq E$ and $\dist_W(y,\rho^{V_j}_W)\leq E$

First, suppose that $\dist_W(x,\rho^{V_j}_W)\leq E$.  Then $\dist_W(y,\rho^{V_j}_W)\geq10K_pE-E>E$, so by consistency, $\dist_{V_j}(y,\rho^W_{V_j})\leq E$.  If $\dist_{V_j}(x,\{x'_j,y'_j\})>E$, then consistency implies that $\dist_W(x_j',\rho^{V_j}_W)\leq E$ and $\dist_W(y'_j,\rho^{V_j}_W)\leq E$, whence $\dist_W(x'_j,y'_j)\leq 2E$, a contradiction.  If $\dist_{V_j}(x,\{x'_j,y'_j\})\leq E$, then since $\dist_{V_j}(x'_j,y'_j)\leq 2K_pE+10\theta_p$, we have $\dist_{V_j}(x,x'_j)\leq5K_pE+10\theta_p$, contradicting that, since $t_j$ was a $(p+1)$--omen witnessed by $V_j$, we must have $\dist_{V_j}(x,x'_j)>5K_pE$.
 
 Second, suppose $\dist_W(y,\rho^{V_j}_W)\leq E$.  Then by relevance of $W$ and consistency, $\dist_{V_j}(x,\rho^W_{V_j})\leq E$.  As above, we have $\dist_{V_j}(x'_j,x)>5K_pE+10\theta_p$, so $\dist_{V_j}(x,\{x'_j,y'_j\})>K_pE>3E$ (since $\dist_{V_j}(x'_j,y'_j)\leq 2K_pE+10\theta_p$ and we may assume $K_p>3$), so $\dist_{V_j}(\rho^W_{V_j},\{x'_j,y'_j\})>E$.  Thus, by consistency, $\pi_W(x'_j),\pi_W(y'_j)$ both lie at distance at most $E$ from $\rho^{V_j}_W$, whence $\dist_W(x'_j,y'_j)\leq 3E$.
 \item Finally, suppose that $W\orth V_j$.  Then either $c(W)<c(V_j)$ and $\beta_j$ is $K_p$--monotone in $W$, or $W$ is irrelevant.
\end{itemize}

Hence, each $\beta_j$ is $10^3(K_pE+\theta_p)$--monotone in each $W\in c^{-1}(\{0,\ldots,p+1\})$.  Moreover, our 
above choice of $\beta_j$ ensures that $\beta_j$ is $K_p$--monotone in each $V\in\mathfrak S_U^{\ell-1}$.

\textbf{Verification that $\alpha_{p+1}$ is monotone:}  Suppose that there exist $t,t'$ such that $t<t'$ and $\dist_{ 
V}(\alpha_{p+1}(0),\alpha_{p+1}(t))> \dist_{ V}(\alpha_{p+1}(0),\alpha_{p+1}(t'))+10^4(K_pE+\theta_p)$ for some $V\in 
c^{-1}(\{0,\ldots,p+1\})\cup\mathfrak S_U^{\ell-1}$. We can assume $t,t'\not\in \cup_i (t_i,t'_i)$. Indeed, if 
$t\in(t_i,t'_i)$ (respectively, $t'\in(t_j,t'_j)$), then since all $\beta_m$ are $10^3(K_pE+\theta_p)$--monotone, we 
can replace $t$ with $t'_i$ (respectively, $t'$ with $t_j$).  After such a replacement, we still have $\dist_{ 
V}(\alpha_{p+1}(0),\alpha_{p+1}(t))> \dist_{ V}(\alpha_{p+1}(0),\alpha_{p+1}(t'))+5K_pE$.

Let $i$ be maximal so that $t'_i\leq t$ (or let $i=-1$ if no such $t'_i$ exists). By definition of $t_{i+1}$, we have 
$t_{i+1}\leq t$, and hence $t_{i+1}=t$. But then $t'_{i+1}>t'$, which is not the case.

\textbf{Conclusion:} Continuing the above procedure while $p<\chi$ produces the desired path $\alpha_\chi$ which is 
$(K_\chi,K_\chi)$--good for  $\mathfrak S_U^\ell$.  In particular, when $U=S$ is $\nest$--maximal and $\ell$ is the 
length of a maximal $\nest$--chain, the proposition follows.
\end{proof}

\subsection{Upper and lower distance bounds}\label{subsec:df_ul}
We now state and prove the remaining lemmas needed to complete the proof of Theorem~\ref{thm:distance_formula} and~\ref{thm:monotone_hierarchy_paths}.

\begin{lem}[Upper bound]\label{lem:df_u}
For every $K,s$ there exists $r$ with the following property. Let 
$\alpha\co\{0,\dots,n\}\to \cuco X$ be a $K$-monotone, $(r,K)$-proper discrete path connecting $x$ to $y$. Then
$$|\alpha|-1\leq \sum_{W\in\mathfrak S}\ignore{\dist_{W}(x,y)}{s}.$$
\end{lem}

\begin{proof}
 Let $r=r(K,E,s)$ be large enough that, for any $a,b\in\cuco X$, if $\dist_{\cuco X}(a,b)\geq r$, then there exists $W\in\mathfrak S$ so that $\dist_{ W}(a,b)\geq 100KEs$. This $r$ is provided by Definition~\ref{defn:space_with_distance_formula}.\eqref{item:dfs_uniqueness}.  
 
For $0\leq j\leq n-1$, choose $V_j\in\mathfrak S$ so that $\dist_{V_j}(\alpha(j),\alpha(j+1))\geq 100KEs$. By monotonicity of $\alpha$ in $V_j$, for any $j'> j$ we have 
 $$\dist_{V_j}(\alpha(0),\alpha(j'))\geq \dist_{V_j}(\alpha(0),\alpha(j))+50KEs.$$  It follows by induction on $j\leq n$ that $\sum_{W\in\mathfrak S}\ignore{\dist_{W}(\alpha(0),\alpha(j))}{s}\geq \min\{j,n-1\}$.
\end{proof}

\begin{lem}[Lower bound]\label{lem:df_l}
There exists $s_0$ such that for all $s\geq s_0$, there exists $C$ with the following property.
$$\dist_{\cuco X}(x,y)\geq \frac{1}{C}\sum_{W\in\mathfrak S}\ignore{\dist_{W}(x,y)}{s}.$$
\end{lem}

\begin{proof}
From Corollary~\ref{cor:paths_in_hulls}, we obtain a $K$--discrete
path $\alpha\co\{0,n\}\to\cuco X$ joining $x,y$ and having the property
that the (coarse) path $\pi_V\circ\alpha\co\{0,\ldots,n\}\to\fontact V$
lies in the $K$--neighborhood of a geodesic from $\pi_V(x)$ to
$\pi_V(y)$.  Moreover, $\alpha$ is $K$--efficient, by the same
corollary.

Fix $s_0\geq 10^3KE$.  A \emph{checkpoint for $x,y$} in $V\in\mathfrak S$ is a ball $Q$ in $\fontact V$ so that $\pi_V\circ\alpha$ intersects $Q$ and $\dist_{ V}(\{x,y\},Q)\geq 10KE+1$.  Note that any ball of radius $10KE$ centered on a geodesic from $\pi_V(x)$ to $\pi_V(y)$ is a checkpoint for $x,y$ in $V$, provided it is sufficiently far from $\{x,y\}$.

For each $V\in\relevant(x,y,10^3KE)$, choose a set $\mathfrak C_V$ of $\left\lceil\frac{\dist_V(x,y)}{10}\right\rceil$ checkpoints for $x,y$ in $V$, subject to the requirement that $\dist_{ V}(C_1,C_2)\geq 10KE$ for all distinct $C_1,C_2\in\mathfrak C_V$.  For each $V\in\relevant(x,y,10^3KE)$, we have $10|\mathfrak C_V|\geq\dist_V(x,y)$, so $$\sum_{V\in\mathfrak S}|\mathfrak C_V|\geq\frac{1}{10}\sum_{W\in\mathfrak S}\ignore{\dist_{W}(x,y)}{s_0}.$$

Each $j\in\{0,\ldots,n\}$ is a \emph{door} if there exists
$V\in\relevant(x,y,10^3KE)$ and $C\in\mathfrak C_V$ such that
$\pi_V(\alpha(j))\in C$ but $\pi_V(\alpha(j-1))\not\in C$.  The
\emph{multiplicity} of a door $j$ is the cardinality of the set
$\mathcal M(j)$ of $V\in\relevant(x,y,10^3KE)$ for which there exists $C\in\mathfrak C_V$ with 
$\pi_V(\alpha(j))\in C$ and $\pi_V(\alpha(j-1))\not\in C$.  Since $\mathfrak C_V$ is a set of
pairwise-disjoint checkpoints, $j$ is a door for at most one element
of $\mathfrak C_V$, for each $V$.  Hence the multiplicity of $j$ is
precisely the total number of checkpoints in
$\cup_{V\in\relevant(x,y,10^3KE)}\mathfrak C_V$ for which $j$ is a
door.

We claim that the set $\mathcal M(j)$ does not contain a pair of transverse elements.  Indeed, suppose that $U,V\in\mathcal M(j)$, satisfy $U\transverse V$.  Let $Q_V\in\mathfrak C_V,Q_U\in\mathfrak C_U$ be the checkpoints containing $\pi_V(\alpha(j)),\pi_U(\alpha(j))$ respectively, so that $\dist_U(\alpha(j),\{x,y\}),\dist_V(\alpha(j),\{x,y\})\geq 10KE+1>10E$, contradicting Corollary~\ref{cor:not_far_from_both}.  Lemma~\ref{lem:ramsey_corollary} thus gives $|\mathcal M_V|\leq\chi$.  Now, $|\alpha|$ is at least the number of doors in $\{0,\ldots,n\}$, whence $|\alpha|\geq\frac{1}{\chi}\sum_{V\in\mathfrak S}|\mathfrak C_V|$.  Since $\alpha$ is $K$--efficient, we obtain $$\dist_{\cuco X}(x,y)\geq\frac{1}{10\chi K}\sum_{W\in\mathfrak S}\ignore{\dist_{W}(x,y)}{s_0}.$$
For $s\geq s_0$, $\sum_{W\in\mathfrak S}\ignore{\dist_{W}(x,y)}{s}\leq\sum_{W\in\mathfrak S}\ignore{\dist_{W}(x,y)}{s_0},$ so the claim follows.
\end{proof}

\section{Hierarchical quasiconvexity and gates}\label{sec:hierarchical_quasiconvexity_and_gates}
We now introduce the notion of hierarchical quasiconvexity, which is essential for the discussion of product regions, the combination theorem of Section~\ref{sec:combination}, and in the forthcoming~\cite{DurhamHagenSisto:HHS_III}.

\begin{defn}[Hierarchical quasiconvexity]\label{defn:hierarchical_quasiconvexity}
Let $(\cuco X,\mathfrak S)$ be a hierarchically hyperbolic space.  
Then $\cuco Y\subseteq\cuco X$ is \emph{$k$--hierarchically quasiconvex}, for some $k\co[0,\infty)\to[0,\infty)$, if the following hold:
\begin{enumerate}
 \item For all $U\in\mathfrak S$, the projection $\pi_U(\cuco Y)$ is 
 a $k(0)$--quasiconvex subspace of the $\delta$--hyperbolic space $\fontact U$.
 \item For all $\kappa\geq0$ and $\kappa$-consistent tuples $\tup
 b\in\prod_{U\in\mathfrak S}2^{\fontact U}$ with $b_U\subseteq\pi_U(\cuco Y)$ for all $U\in\mathfrak S$, each point
 $x\in\cuco X$ for which $\dist_{
 U}(\pi_U(x),b_U)\leq\theta_e(\kappa)$ (where $\theta_e(\kappa)$ is as in Theorem~\ref{thm:realization}) satisfies $\dist(x,\cuco Y)\leq
 k(\kappa)$.
\end{enumerate}
\end{defn}

\begin{rem} Note that condition (2) in the above definition is 
    equivalent to:  For every $\kappa>0$ and every point 
    $x\in\cuco X$ satisfying 
    $\dist_{U}(\pi_U(x),\pi_{U}(\cuco Y))\leq\kappa$ 
    for all $U\in\mathfrak S$, has the property that 
    $\dist(x,\cuco Y)\leq k(\kappa)$.
\end{rem}

\begin{lem}\label{lem:proj_consistent}
 For each $Q$ there exists $\kappa$ so that the following holds. Let $\cuco Y\subseteq\cuco X$ be such that 
$\pi_V(\cuco Y)$ is $Q$--quasiconvex for each $V\in\mathfrak S$. Let $x\in\cuco X$ and, for each $V\in\mathfrak S$, 
let $p_V\in\pi_V(\cuco Y)$ satisfy $\dist_V(x,p_V)\leq \dist_V(x,\cuco Y)+1$. Then $(p_V)$ is $\kappa$--consistent.
\end{lem}

\begin{proof}
 For each $V$, choose $y_V\in\cuco Y$ so that $\pi_V(y_V)=p_V$.

Suppose that $V\transverse W$ or $V\nest W$.  By Lemma~\ref{lem:median_cons} and Theorem~\ref{thm:realization}, there exists $z\in\cuco X$ so that for all $U\in\mathfrak S$, the projection $\pi_U(z)$ lies $C$--close to each of the geodesics $[\pi_U(x),\pi_U(y_V)],$ $[\pi_U(x),\pi_U(y_W)],$~and $[\pi_U(y_W),\pi_U(y_V)]$, where $C$ depends on $\cuco X$.  Hence $\dist_V(p_V,z)$ and $\dist_W(p_W,z)$ are uniformly bounded, by quasiconvexity of $\pi_V(\cuco Y)$ and $\pi_W(\cuco Y)$.  

Suppose that $V\transverse W$.  Since the tuple $(\pi_U(z))$ is consistent, either $y_V$ lies uniformly close in $\fontact V$ to $\rho^W_V$, or the same holds with $V$ and $W$ interchanged, as required.  Suppose that $V\nest W$.  Suppose that $\dist_W(p_W,\rho^V_W)$ is sufficiently large, so that we have to bound $\diam_V(\rho^W_V(p_W)\cup p_V)$.  Since $\dist_W(z,p_W)$ is uniformly bounded, $\dist_W(z,\rho^V_W)$ is sufficiently large that consistency ensures that $\diam_V(\rho^W_V(\pi_W(z))\cup\pi_V(z))$ is uniformly bounded.  Since any geodesic from $p_W$ to $z$ lies far from $\rho^V_W$, the sets $\rho^W_V(\pi_W(z))$ and $\rho^W_V(p_V)$ coarsely coincide. Since $\pi_V(z)$ coarsely coincides with $p_V$ by construction of $z$, we have the required bound.  Hence the tuple with $V$--coordinate $p_V$ is $\kappa$--consistent for uniform $\kappa$.
\end{proof}

\begin{defn}[Gate]\label{defn:gate}
A coarsely Lipschitz map 
$\gate_{\cuco Y}\colon\cuco X\to\cuco Y$ is called a \emph{gate map} 
if for each $x\in\cuco X$ it satisfies:  $\gate_{\cuco
Y}(x)$ is a point $y\in\cuco Y$ 
such that for all $V\in\mathfrak S$,
the set $\pi_V(y)$ (uniformly) coarsely coincides with the projection
of $\pi_V(x)$ to the $k(0)$--quasiconvex set $\pi_V(\cuco Y)$. The 
point $\gate(x)$ is called the \emph{gate of $x$ in $\cuco Y$}. 
The uniqueness axiom implies that when such a map exists it 
is coarsely well-defined.  
\end{defn}

We first establish that, as should be the case for a (quasi)convexity
property, one can coarsely project to hierarchically quasiconvex
subspaces. The next lemma shows that gates exist for 
$k$--hierarchically quasiconvex subsets.

\begin{lem}[Existence of coarse gates]\label{lem:gate}
If $\cuco Y\subseteq\cuco X$ is $k$--hierarchically quasiconvex and non-empty, then there exists a gate map for 
$\cuco Y$, i.e., for each $x\in \cuco X$ 
there exists $y\in\cuco Y$ such that for all $V\in\mathfrak S$,
the set $\pi_V(y)$ (uniformly) coarsely coincides with the projection
of $\pi_V(x)$ to the $k(0)$--quasiconvex set $\pi_V(\cuco Y)$.
\end{lem}

\begin{proof}
For each $V\in\mathfrak S$, let $p_V\in\pi_V(\cuco Y)$ satisfy $\dist_V(x,p_V)\leq \dist_V(x,\cuco Y)+1$. Then 
$(p_V)$ is $\kappa$--consistent for some $\kappa$ independent of $x$ by Lemma \ref{lem:proj_consistent}.  (Note that $(p_v)$ 
is admissible by construction.)

Theorem~\ref{thm:realization} and the definition of hierarchical
quasiconvexity combine to supply $y'\in\neb_{k(\kappa)}(\cuco Y)$ with
the desired projections to all $V\in\mathfrak S$; this point lies at
distance $k(\kappa)$ from some $y\in\cuco Y$ with the desired
property.

We now check that this map is coarsely Lipschitz.  Let
$x_0,x_n\in\cuco X$ be joined by a uniform quasigeodesic $\gamma$.
By sampling $\gamma$, we obtain a discrete path
$\gamma'\colon\{0,\ldots,n\}\to\cuco X$ such that $\dist_{\cuco
X}(\gamma'(i),\gamma'(i+1))\leq K$ for $0\leq i\leq n-1$, where $K$
depends only on $\cuco X$, and such that
$\gamma'(0)=x_0,\gamma'(n)=x_n$.  Observe that $\dist_{\cuco
X}(\gate_{\cuco Y}(x_0),\gate_{\cuco
Y}(x_n))\leq\sum_{i=0}^{n-1}\dist_{\cuco X}(\gate_{\cuco
Y}(\gamma'(i)),\gate_{\cuco Y}(\gamma'(i+1)))$, so it suffices to
exhibit $C$ such that $\dist_{\cuco X}(\gate_{\cuco Y}(x),\gate_{\cuco
Y}(x'))\leq C$ whenever $\dist_{\cuco X}(x,x')\leq K$.  But if
$\dist_{\cuco X}(x,x')\leq K$, then each $\dist_U(x,x')\leq K'$ for
some uniform $K'$, by
Definition~\ref{defn:space_with_distance_formula}.\eqref{item:dfs_curve_complexes},
whence the claim follows from the fact that each $\fontact
U\to\pi_U(\cuco Y)$ is coarsely Lipschitz (with constant depending
only on $\delta$ and $k(0)$) along with the uniqueness axiom
(Definition~\ref{defn:space_with_distance_formula}.\eqref{item:dfs_uniqueness}).
\end{proof}

\subsection{Hierarchically quasiconvex subspaces are hierarchically hyperbolic}\label{sec:hier_conv_hier_hyp}

\begin{prop}\label{prop:sub_hhs}
Let $\cuco Y\subseteq\cuco X$ be a hierarchically $k$-quasiconvex subset of the hierarchically hyperbolic space $(\cuco X,\mathfrak S)$.  Then $(\cuco Y,\dist)$ is a hierarchically hyperbolic space, where $\dist$ is the metric inherited from $\cuco X$. 
\end{prop}

\begin{proof}
There exists $K$ so that any two points in $\cuco Y$ are joined by a uniform quasigeodesic.  Indeed, any two points in $\cuco Y$ are joined by a hierarchy path in $\cuco X$, which must lie uniformly close to $\cuco Y$.

We now define a hierarchically hyperbolic structure.  For each $U$, 
let $r_U\co \fontact U\to\pi_U(\cuco Y)$ be the coarse projection, 
which exists by quasiconvexity.  The index set is $\mathfrak S$, and 
the associated hyperbolic spaces are the various $\fontact U$.  For 
each $U$, define a projection $\pi'_U\co\cuco \to\fontact U$ by $\pi'_U=r_U\circ\pi_U$, and for each non-orthogonal pair $U,V\in\mathfrak S$, the corresponding relative projection $\fontact U\to\fontact V$ is given by $r_V\circ\rho^U_V$.  All of the requirements of Definition~\ref{defn:space_with_distance_formula} involving only the various relations on $\mathfrak S$ are obviously satisfied, since we have only modified the projections.  The consistency inqualities continue to hold since each $r_U$ is uniformly coarsely Lipschitz.  The same is true for bounded geodesic image and the large link lemma.  Partial realization holds by applying the map $\gate_{\cuco Y}$ to points constructed using partial realization in $(\cuco X,\mathfrak S)$.
\end{proof}

\begin{rem}[Alternative hierarchically hyperbolic structures]
In the above proof, one can replace each $\fontact U$ by a thickening $\fontact U_{\cuco Y}$ of $\pi_U(\cuco Y)$ (this set is quasiconvex; the thickening is to make a hyperbolic geodesic space).  This yields a hierarchically hyperbolic structure with coarsely surjective projections.
\end{rem}

\subsection{Standard product regions}\label{subsec:product_regions}
In this section, we describe a class of hierarchically quasiconvex subspaces called \emph{standard product regions} that will 
be useful in future applications. 
We first recall a construction from~\cite[Section~13]{BehrstockHagenSisto:HHS_I}.  

\begin{defn}[Nested partial tuple]\label{defn:nested_partial_tuple}
Recall $\mathfrak S_U=\{V\in\mathfrak S:V\nest U\}$.  Fix
$\kappa\geq\kappa_0$ and let $\mathbf F_U$ be the set of
$\kappa$--consistent tuples in $\prod_{V\in\mathfrak S_U}2^{\fontact
V}$. 
\end{defn}

\begin{defn}[Orthogonal partial tuple]\label{defn:orthogonal_partial_tuple}
Let $\mathfrak S_U^\orth=\{V\in\mathfrak S:V\orth U\}\cup\{A\}$, where
$A$ is a $\nest$--minimal element $A$ such that $V\nest A$ for all
$V\orth U$.  Fix $\kappa\geq\kappa_0$, let $\mathbf E_U$ be the set of
$\kappa$--consistent  tuples in $\prod_{V\in\mathfrak
S_U^\orth-\{A\}}2^{\fontact V}$. 
\end{defn}

\begin{cons}[Product regions in $\cuco X$]\label{const:embedding_product_regions}
Given $\cuco X$ and $U\in\mathfrak S$, there are coarsely well-defined
maps $\phi^\nest,\phi^\orth\co\mathbf F_U,\mathbf E_U\to\cuco X$ whose
images are hierarchically quasiconvex and which extend to a coarsely
well-defined map $\phi_U\co\mathbf F_U\times \mathbf E_U\to\cuco X$
with hierarchically quasiconvex image.  Indeed, for each $(\tup a,\tup b)\in
\mathbf F_U\times \mathbf E_U$, and each $V\in\mathfrak S$, define the
co-ordinate $(\phi_U(\tup a,\tup b))_V$ as follows.  If $V\nest U$,
then $(\phi_U(\tup a,\tup b))_V=a_V$.  If $V\orth U$, then
$(\phi_U(\tup a,\tup b))_V=b_V$.  If $V\transverse U$, then
$(\phi_U(\tup a,\tup b))_V=\rho^U_V$.  Finally, if $U\nest V$, and
$U\neq V$, let $(\phi_U(\tup a,\tup b))_V=\rho^U_V$.

We now verify that the tuple $\phi_U(\tup a,\tup b)$ is consistent.
If $W,V\in\mathfrak S$, and either $V$ or $W$ is transverse to $U$,
then the consistency inequality involving $W$ and $V$ is satisfied in view of Proposition~\ref{prop:rho_consistency}.  The same holds if $U\nest W$ or $U\nest V$.  Hence, 
it remains to consider the cases where $V$ and $W$ are each either nested into or orthogonal to $U$: if $V,W\nest U$ or $V,W\orth U$ then
consistency holds by assumption; otherwise, up to reversing the roles 
of $V$ and $W$ we have  $V\nest U$ and
$W\orth U$, in which case $V\orth W$ and there is nothing to check.
Theorem~\ref{thm:realization} thus supplies the map $\phi_U\co \mathbf F_U\times\mathbf E_U\to\cuco X$.  Fixing any 
$e\in\mathbf E_U$ yields a map $\phi^\nest\co\mathbf F_U\times\{e\}\to\cuco X$, and $\phi^\orth$ is defined 
analogously.  Note that these maps depend on choices of basepoints in $\mathbf E_U,\mathbf F_U$.

Where it will not introduce confusion (e.g., where the basepoints are understood or immaterial), we abuse
notation and regard $\mathbf F_U,\mathbf E_U$ as subspaces of $\cuco
X$, i.e., $\mathbf F_U=\image\phi^\nest,\mathbf E_U=\image\phi^\orth$.
\end{cons}

\begin{prop}\label{prop:hereditary_HHS}
When $\mathbf E_U,\mathbf F_U\subset\cuco X$ are endowed with the subspace metric $\dist$, the spaces $(\mathbf 
F_U,\mathfrak S_U)$ and $(\mathbf E_U,\mathfrak S_U^\orth)$ are hierarchically hyperbolic; if $U$ is not $\nest$--maximal, 
then their complexity is strictly less than that of $(\cuco X,\mathfrak S)$.  Moreover, $\phi^\nest$ and $\phi^\orth$ 
determine hieromorphisms $(\mathbf F_U,\mathfrak S_U),(\mathbf E_U,\mathfrak S_U^\orth)\to(\cuco X,\mathfrak S)$.
\end{prop}

\begin{proof}
For each $V\nest U$ or $V\orth U$, the associated hyperbolic space $\fontact V$ is exactly the one used in the hierarchically hyperbolic structure $(\cuco X,\mathfrak S)$.  For $A$, use an appropriate thickening $\fontact^*A$ of $\pi_A(\image\phi^\orth)$ to a hyperbolic geodesic space.  
All of the projections $\mathbf F_U\to\fontact V,V\in\mathfrak S_U$ and $\mathbf E_U\to\fontact V,V\in\mathfrak S_U^\orth$ are as in $(\cuco X,\mathfrak S)$ (for $A$, compose with a quasi-isometry $\pi_A(\image\phi^\orth)\to\fontact^*A$).  Observe that $(\mathbf F_U,\mathfrak S_U)$ and $(\mathbf E_U,\mathfrak S_U^\orth)$ are hierarchically hyperbolic (this can be seen using a simple version of the proof of Proposition~\ref{prop:sub_hhs}).  If $U$ is not $\nest$--maximal in $\mathfrak S$, then neither is $A$, whence the claim about complexity.

The hieromorphisms are defined by the inclusions $\mathfrak S_U,\mathfrak S_U^\orth\to\mathfrak S$ and, for each $V\in\mathfrak S_U\cup\mathfrak S_U^\orth$, the identity $\fontact V\to\fontact V$, unless $V=A$, in which case we use $\fontact^*A\to\pi_A(\image\phi^\orth)\hookrightarrow\fontact A$.  These give hieromorphisms by definition.
\end{proof}

\begin{rem}[Dependence on $A$]\label{rem:dependence_on_a}
Note that $A$ need not be the unique $\nest$--minimal element of
$\mathfrak S$ into which each $V\orth U$ is nested; the axioms don't require uniqueness of such $\nest$--minimal elements.  Observe that
$\mathbf E_U$ (as a set and as a subspace of $\cuco X$) is defined
independently of the choice of $A$.  It is the hierarchically
hyperbolic structure from Proposition~\ref{prop:hereditary_HHS} that a
priori depends on $A$.  However, note that $A\not\nest U$, since there
exists $V\nest A$ with $V\orth U$, and we cannot have $V\nest U$ and
$V\orth U$ simultaneously.  Likewise, $A\notorth U$ by definition.
Finally, if $U\nest A$, then the axioms guarantee the existence of
$B$, properly nested into $A$, into which each $V\orth U$ is nested,
contradicting $\nest$--minimality of $A$.  Hence $U\transverse A$.  It
follows that $\pi_A(\mathbf E_U)$ is bounded --- it coarsely coincides
with $\rho^U_A$.  Thus the hierarchically hyperbolic structure on
$\mathbf E_U$, and the hieromorphism structure of $\phi^\orth$, is
actually essentially canonical: we can take the hyperbolic space
associated to the $\nest$--maximal element to be a point, whose image
in each of the possible choices of $A$ must coarsely coincide with
$\rho^U_A$.
\end{rem}

\begin{rem}[Orthogonality and product regions]\label{rem:orth_prod}
If $U\orth V$, then we have $\mathbf F_U\subseteq \mathbf E_V$ and
$\mathbf F_V\subseteq \mathbf E_U$, so there is a hierarchically
quasiconvex map $\phi_U\times\phi_V\co\mathbf F_U\times \mathbf
F_V\to\cuco X$ extending to $\phi_U\times\phi^\orth_U$ and
$\phi^\orth_V\times\phi_V$.
\end{rem}

\begin{rem}\label{rem:alternate_gates}
Since $\mathbf F_U,\mathbf E_U$ are hierarchically quasiconvex spaces,
Definition~\ref{defn:gate} provides coarse gates $\gate_{\mathbf
F_U}\co\cuco X\to \mathbf F_U$ and $\gate_{\mathbf E_U}\co\cuco
X\to \mathbf E_U$.  These are coarsely the same as the following maps:
given $x\in\cuco X$, let $\tup x$ be the tuple defined by
$x_W=\pi_W(x)$ when $W\nest U$ and $x_W=\pi_W(x)$ when $W\orth U$ and
$\rho^U_W$ otherwise.  Then $\tup x$ is consistent and coarsely equals
$\gate_{\mathbf F_U\times \mathbf E_U}(x)$.
\end{rem}

\begin{defn}[Standard product region]\label{defn:standard_product}
For each $U\in\mathfrak S$, let $P_U=\image\phi_U$, which is coarsely
$\mathbf F_U\times \mathbf E_U$.  We call this the \emph{standard
product region} in $\cuco X$ associated to $U$.
\end{defn}

The next proposition follows from the definition of the product regions and the fact that, if $U\nest V$, then $\rho^U_W,\rho^V_W$ coarsely coincide whenever $V\nest W$ or $V\transverse W$ and $U\not\orth W$, which holds by Definition~\ref{defn:space_with_distance_formula}.\eqref{item:dfs_transversal}.  

\begin{prop}[Parallel copies]\label{prop:parallel_copies}
There exists $\nu\geq0$ such that for all $U\in\mathfrak S$, all $V\in\mathfrak S_U$, and all $u\in \mathbf E_U$, there exists $v\in \mathbf E_V$ so that $\phi_V(\mathbf F_V\times\{v\})\subseteq\neb_\nu(\phi_U(\mathbf F_U\times\{u\}))$.  
\end{prop}

\subsubsection{Hierarchy paths and product regions}\label{subsubsec:hp_pr}
Recall that a \emph{$D$--hierarchy path} $\gamma$ in $\cuco X$ is a $(D,D)$--quasigeodesic $\gamma:I\to\cuco X$ such that $\pi_U\circ\gamma$ is an unparameterized $(D,D)$--quasigeodesic for each $U\in\mathfrak S$, and that Theorem~\ref{thm:monotone_hierarchy_paths} provides $D\geq1$ so that any two points in $\cuco X$ are joined by a $D$--hierarchy path.  In this section, we describe how hierarchy paths interact with standard product regions.  

In the next proposition and lemma, given $x,y\in\cuco X$, we declare $V\in\mathfrak S$ to be \emph{relevant (for 
$x,y$)} if $\dist_V(x,y)\geq 200DE$.

\begin{prop}[``Active'' subpaths]\label{prop:hierarchy_path_goes_close}
There exists $\nu\geq 0$ so that for all $x,y\in\cuco X$, all $V\in\mathfrak S$ with $V$ relevant for $(x,y)$, and all $D$--hierarchy paths $\gamma$ joining $x$ to $y$, there is a subpath $\alpha$ of $\gamma$ with the following properties:
\begin{enumerate}
 \item $\alpha\subset\neb_\nu(P_V)$;\label{item:HPGC_1}
 \item $\pi_U|_{\gamma}$ is coarsely constant on $\gamma-\alpha$ for all $U\in\mathfrak S_V\cup\mathfrak S_V^\orth$.\label{item:HPGC_2}
\end{enumerate}
\end{prop}

\begin{proof}
We may assume $\gamma\co\{0,n\}\to\cuco X$ is a $2D$--discrete path.  Let $x_i=\gamma(i)$ for $0\leq i\leq n$.  Let $S\in\mathfrak S$ be the $\nest$--maximal element.  Since the proposition holds trivially for $V=S$, assume $V\propnest S$.

First consider the case where $V$ is $\nest$--maximal among relevant elements of $\mathfrak S$.  Lemma~\ref{lem:hierarchy_path_goes_close_in_S} provides $\nu''\geq0$, independent of $x,y$, and also provides $i\leq n$, such that $\dist_S(x_i,\rho^V_S)\leq\nu''$.  Let $i$ be minimal with this property and let $i'$ be maximal with this property.  Observe that there exists $\nu'\geq\nu''$, depending only on $\nu''$ and the (uniform) monotonicity of $\gamma$ in $\fontact S$, such that $\dist_S(x_j,\rho^V_S)\leq\nu'$ for $i\leq j\leq i'$.

For $j\in\{i,\ldots,i'\}$, let $x'_j=\gate_{P_V}(x_j)$.  Let $U\in\mathfrak S$.  By definition, if $U\nest V$ or $U\orth V$, then $\pi_U(x_j)$ coarsely coincides with $\pi_U(x'_j)$, while $\pi_U(x'_j)$ coarsely coincides with $\rho^V_U$ if $V\nest U$ or $V\transverse U$.  We claim that there exist $i_1,i'_1$ with $i\leq i_1\leq i'_1\leq i'$ such that for $i_1\leq j\leq i'_1$ and $U\in\mathfrak S$ with $V\nest U$ or $U\transverse V$, the points $\pi_U(x_j)$ and $\pi_U(x'_j)$ coarsely coincide; this amounts to claiming that $\pi_U(x_j)$ coarsely coincides with $\rho^V_U$.  

If $V\nest U$ and some geodesic $\sigma$ in $\fontact U$ from $\pi_U(x)$ to $\pi_U(y)$ fails to pass through the $E$-neighborhood of $\rho^V_U$, then bounded geodesic image shows that $\rho^U_V(\sigma)$ has diameter at most $E$.  On the other hand, consistency shows that the endpoints of $\rho^U_V(\sigma)$ coarsely coincide with $\pi_V(x)$ and $\pi_V(y)$, contradicting that $V$ is relevant.  Thus $\sigma$ passes through the $E$--neighborhood of $\rho^V_U$.  Maximality of $V$ implies that $U$ is not relevant, so that $\pi_V(x),\pi_V(y)$, and $\pi_V(x_j)$ all coarsely coincide, whence $\pi_V(x_j)$ coarsely coincides with $\rho^V_U$.

If $U\transverse V$ and $U$ is not relevant, then $\pi_U(x_j)$ coarsely coincides with both $\pi_U(x)$ and $\pi_U(y)$, each of which coarsely coincides with $\rho^V_U$, for otherwise we would have $\dist_V(x,y)\leq 2E$ by consistency and the triangle inequality, contradicting that $V$ is relevant.  If $U\transverse V$ and $U$ is relevant, then, by consistency, we can assume that $\pi_U(y)$,$\rho^V_U$ coarsely coincide, as do $\pi_V(x)$,$\rho^U_V$.  Either $\pi_U(x_j)$ coarsely equals $\rho^V_U$, or $\pi_V(x_j)$ coarsely equals $\pi_V(x)$, again by consistency.  If $\dist_V(x,x_j)\leq 10E$ or $\dist_V(y,x_j)\leq 10E$, discard $x_j$.  Our discreteness assumption and the fact that $V$ is relevant imply that there exist $i_1\leq i'_1$ between $i$ and $i'$ so that $x_j$ is not discarded for $i_1\leq j\leq i'_1$.  For such $j$, the distance formula now implies that $\dist(x_j,x'_j)$ is bounded by a constant $\nu$ independent of $x,y$.

We thus have $i_1,i'_1$ such that $x_j\in\neb_\nu(P_V)$ for $i\leq j\leq i'$ and $x_j\not\in\neb_\nu(P_V)$ for $j<i$ or $j>i'$, provided $V$ is $\nest$--maximal relevant.  If $W\nest V$ and $W$ is relevant, and there is no relevant $W'\neq W$ with $W\nest W'\nest V$, then we may apply the above argument to $\gamma'=\gate_{P_V}(\gamma|_{i,\ldots,i'})$ to produce a subpath of $\gamma'$ lying $\nu$--close to $P_W\subseteq P_V$, and hence a subpath of $\gamma$ lying $2\nu$--close to $P_W$.  Finiteness of the complexity (Definition~\ref{defn:space_with_distance_formula}.\eqref{item:dfs_complexity}) then yields assertion~\eqref{item:HPGC_1}.  Assertion~\eqref{item:HPGC_2} is immediate from our choice of $i_1,i_1'$.
\end{proof}

\begin{lem}\label{lem:hierarchy_path_goes_close_in_S}
There exists $\nu'\geq0$ so that for all $x,y\in\cuco X$, all relevant $V\in\mathfrak S$, and all $D$--hierarchy paths $\gamma$ joining $x$ to $y$, there exists $t\in\gamma$ so that $\dist_S(t,\rho^V_S)\leq\nu'$.
\end{lem}

\begin{proof}
Let $\sigma$ be a geodesic in $\fontact S$ joining the endpoints of $\pi_S\circ\gamma$.  Since $\dist_V(x,y)\geq200DE$, the consistency and bounded geodesic image axioms (Definition~\ref{defn:space_with_distance_formula}.\eqref{item:dfs_transversal},\eqref{item:dfs:bounded_geodesic_image}) imply that $\sigma$ enters the $E$--neighborhood of $\rho^V_S$ in $\fontact S$, whence $\pi_S\circ\gamma$ comes uniformly close to $\rho^V_S$.
\end{proof}

\section{Hulls}\label{sec:convex_hulls}
In this section we build ``convex hulls'' in 
hierarchically hyperbolic spaces. This construction is motivated by, 
and generalizes,  the concept in the mapping class group called 
$\Sigma$--hull, as defined by 
Behrstock--Kleiner--Minsky--Mosher \cite{BKMM:consistency}. 
Recall that given a set $A$ of points in a $\delta$--hyperbolic 
space $H$, its \emph{convex hull}, denoted $\hull_{H}(A)$, is the union of 
geodesics between pairs of points in this set. We will make use of 
the fact that the convex hull is $2\delta$--quasiconvex (since, if $p\in [x,y], q\in [x',y']$, then $[p,q]\subseteq \neb_{2\delta}( [p,x] \cup [x,x']\cup [x',q])\subseteq \neb_{2\delta}( [y,x] \cup [x,x']\cup [x',y'])$).

The construction of hulls is based on Proposition~\ref{prop:coarse_retract}, which generalizes 
Lemma~\ref{lem:2_point_retract}; indeed, the construction of hulls in this section generalizes the hulls of pairs of points 
used in Section~\ref{sec:hier_path_df} to prove the distance formula.  The second part of 
Proposition~\ref{prop:coarse_retract} (which is not used in Section~\ref{sec:hier_path_df}) relies on the distance formula.

\begin{defn}[Hull of a set]\label{defn:hull}
 For each set $A\subset \cuco X$ and $\theta\geq 0$, let
 $H_{\theta}(A)$ be the set of all $p\in\cuco X$ so that, for each
 $W\in\mathfrak S$, the set $\pi_W(p)$ lies at distance at most
 $\theta$ from $\hull_{\fontact W}(A)$. 
 Note that $A\subset H_{\theta}(A)$.
\end{defn}

\begin{lem}\label{lem:hull}
 There exists $\theta_0$ so that for each $\theta\geq \theta_0$ there exists $k:\reals^+\to\reals^+$ and each $A\subseteq \cuco X$, we have that $H_{\theta}(A)$ is $k$--hierarchically quasiconvex.
\end{lem}

\begin{proof}
For any $\theta$ and $U\in\mathfrak S$, due to $\delta$--hyperbolicity we have that $\pi_U(H_{\theta}(A))$ is $2\delta$--quasiconvex, so we only have to check the condition on realization points.

 Let $A'$ be the union of all $D_0$--hierarchy paths joining points in $A$, where $D_0$ is the constant from Theorem \ref{thm:monotone_hierarchy_paths}. Then the Hausdorff distance between $\pi_U(A')$ and $\pi_U(A)$ is bounded by $C=C(\delta,D_0)$ for each $U\in\mathfrak S$. Also, $\pi_U(A')$ is $Q=Q(\delta,D_0)$--quasiconvex. Let $\kappa$ be the constant from Lemma \ref{lem:proj_consistent}, and let $\theta_0=\theta_e(\kappa)$ be as in Theorem \ref{thm:realization}.
 
Fix any $\theta\geq\theta_0$, and any $\kappa\geq 0$.  Let $(b_U)$ be
a $\kappa'$--consistent tuple with $b_U\subseteq
\neb_{\theta}(\hull_{\fontact U}(A))$ for each $U\in\mathfrak S$.  Let
$x\in\cuco X$ project $\theta_e(\kappa')$--close to each $b_U$.  We
have to find $y\in H_{\theta}(A)$ uniformly close to $x$.  By Lemma
\ref{lem:proj_consistent}, $(p_U)$ is $\kappa$--consistent, where
$p_U\in \hull_{\fontact W}(A)$ satisfies $\dist_U(x,p_U)\leq
\dist_U(x,\hull_{\fontact W}(A))+1$.  It is readily seen from the
uniqueness axiom (Definition
\ref{defn:space_with_distance_formula}.\ref{item:dfs_uniqueness}) that
any $y\in\cuco X$ projecting close to each $p_U$ has the required
property, and such a $y$ exists by Theorem \ref{thm:realization}.  To
check admissibility, note that each $p_U$ lies $\theta$--close to
$\hull_{\fontact U}(A)$, which in turn lies uniformly close to
$\pi_U(\cuco X)$ by quasiconvexity of $\pi_U(\cuco
X)$.
\end{proof}

We denote the Hausdorff distance in the metric space $Y$ by
$\dist_{Haus,Y}(\cdot,\cdot)$. The next proposition directly generalizes 
\cite[Proposition~5.2]{BKMM:consistency} from mapping class groups to 
general hierarchically hyperbolic spaces. 

\begin{prop}[Retraction onto hulls]\label{prop:coarse_retract}
For each sufficiently large
$\theta$ there exists $C\geq 1$ so that for each set $A\subset\cuco X$
 there is a $(K,K)$--coarsely Lipschitz map
$r\colon\cuco X\to
H_{\theta}(A)$ restricting to the identity on $H_\theta(A)$. 
Moreover, if $A'\subset\cuco X$ lies at finite Hausdorff distance from $A$, then 
$d_{\cuco X}(r_{A}(x), r_{A'}(x))$ is $C$--coarsely Lipschitz in $\dist_{Haus,\cuco X}(A,A')$.
\end{prop}

\begin{proof} By Lemma \ref{lem:hull}, for all sufficiently large $\theta$, $H_{\theta}(A)$ is 
    hierarchically quasiconvex. Thus, by Lemma~\ref{lem:gate} there exists 
    a map $r\co\cuco X \to H_{\theta}(A)$,  which is 
    coarsely Lipschitz and 
    which is the identity on $H_{\theta}(A)$.

We now prove the moreover clause.  By
Definition~\ref{defn:space_with_distance_formula}.\eqref{item:dfs_curve_complexes},
for each $W$ the projections $\pi_{W}$ are each coarsely Lipschitz and
thus $\dist_{Haus,\fontact W}(\pi_{W}(A),\pi_{W}(A'))$ is bounded by a
coarsely Lipschitz function of $\dist_{Haus,\cuco X}(A,A')$.  It is
then easy to conclude using the distance formula (Theorem
\ref{thm:distance_formula}) and the construction of gates (Definition
\ref{defn:gate}) used to produce the map $r$.
\end{proof}

\subsection{Homology of asymptotic cones} In this subsection we make a digression to study homological properties of asymptotic cones of hierarchically hyperbolic spaces. This subsection is not needed for the proof of distance formula, and in fact we will use the distance formula in a proof.

Using Proposition~\ref{prop:coarse_retract}, the identical proof as 
used in 
\cite[Lemma~5.4]{BKMM:consistency} for mapping class groups, yields:

\begin{prop}\label{prop:HullProperties}
    There exists $\theta_0 \ge 0$ depending only on the constants 
    of the hierarchically hyperbolic space $(\cuco X,\mathfrak S)$ such
    that for all $\theta,\theta' \ge \theta_0$ there
    exist $K$, $C$, and $\theta''$ such that given two sets 
    $A ,A' \subset \cuco X$, then:
\begin{enumerate}
\item $\diam(H_{\theta}(A)) \le K \, \diam(A)+C$
\item If $A'\subset H_{\theta}(A)$ then $H_{\theta}(A') \subset
H_{\theta''}(A)$.\label{item:subhull}
\item $d_{Haus,\cuco X}(H_{\theta}(A),H_{\theta}(A')) \le K d_{Haus,\cuco X}(A,A')+C$.
\item \label{item:HullIncreasingTheta} $d_{Haus,\cuco X}(H_{\theta}(A), H_{\theta'}(A)) \le C$
\end{enumerate}
\end{prop}

\begin{rem}
 Proposition \ref{prop:HullProperties} is slightly stronger than the
 corresponding \cite[Lemma 5.4]{BKMM:consistency}, in which $A,A'$ are
 finite sets and the constants depend on their cardinality.  The
 source of the strengthening is just the observation that hulls in
 $\delta$--hyperbolic spaces are $2\delta$--quasiconvex regardless of
 the cardinality of the set (see  \cite[Lemma 5.1]{BKMM:consistency}).
\end{rem}

 It is an easy observation that given a sequence $\seq A$ of sets  $A_{n}\subset \cuco X$ with bounded 
cardinality, the retractions to the corresponding hulls  
$H_{\theta}(A_{n})$ converge in 
any asymptotic cone, $\cuco X_{\omega}$, to a Lipschitz retraction from that asymptotic 
cone to the 
ultralimit of the hulls, $H(\seq A )$. A general argument, see e.g., 
\cite[Lemma~6.2]{BKMM:consistency} implies that the ultralimit of 
the hulls is then contractible. 
The proofs in \cite[Section~6]{BKMM:consistency} then apply in 
the present context using the above Proposition, with the only change 
needed that the reference to the rank theorem for 
hierarchically hyperbolic spaces as proven in 
\cite[Theorem~J]{BehrstockHagenSisto:HHS_I} must replace the 
application of \cite{BehrstockMinsky:dimension_rank}. In particular, this 
yields the following two results:

\begin{cor}\label{cor:AcyclicHulls}
Let $\cuco X$ be a hierarchically hyperbolic space and $\cuco 
X_{\omega}$ one of its asymptotic cones. Let $X\subset \cuco 
X_{\omega}$ be an open subset and suppose that for any sequence, $\seq 
A,$ of 
finite subsets of $X$ we have $H(\seq A)\subset X$. Then $X$ is acyclic.
\end{cor}

\begin{cor}\label{cor:homologydimension}
If $(U,V)$ is an open pair in $\cuco X_{\omega}$, then $H_k(U,V)=\{0\}$ for all 
$k$ greater than the complexity of $\cuco X$.
\end{cor}

\subsection{Relatively hierarchically hyperbolic spaces and the distance formula}\label{subsec:generalized_hull}
In this section, we work in the following context:

\begin{defn}[Relatively hierarchically hyperbolic spaces]\label{defn:RHHS}
The hierarchical space $(\cuco X,\mathfrak S)$ is \emph{relatively hierarchically hyperbolic} if there exists $\delta$ such 
that for all $U\in\mathfrak S$, either $U$ is $\nest$--minimal or $\fontact U$ is $\delta$--hyperbolic.  If $U$ is 
$\nest$--minimal and $\fontact U$ is not hyperbolic, then we insist that $\pi_U$ is $E$--coarsely surjective.
\end{defn}

\begin{rem}One could, more generally, only insist that each
$\pi_U(\cuco X)$ is a uniformly coarsely Lipschitz coarse retract.
For hyperbolic $\fontact U$, this is equivalent to the uniform
quasiconvexity from Definition~\ref{defn:space_with_distance_formula},
and is sufficient for our needs; for the present applications  
Definition~\ref{defn:RHHS} is sufficiently general, as well as for  
applications in \cite{HHS_3}.\end{rem}

Our goal is to prove the following two theorems, which provide hierarchy paths and a distance formula in relatively hierarchically hyperbolic spaces.  We will not use these theorems in the remainder of this paper, but they are required for future applications.

\begin{thm}[Distance formula for relatively hierarchically hyperbolic spaces]\label{thm:rel_hyp_df}
Let $(\cuco X,\mathfrak S)$ be a relatively hierarchically hyperbolic space.  Then there exists $s_0$ such that for all $s\geq s_0$, there exist constants $C,K$ such that for all $x,y\in\cuco X$, $$\dist_{\cuco X}(x,y)\asymp_{K,C}\sum_{U\in\mathfrak S}\ignore{\dist_U(x,y)}{s}.$$
\end{thm}

\begin{proof}
By Proposition~\ref{prop:hulls} below, for some suitably-chosen $\theta\geq0$ and each $x,y\in\cuco X$, there exists a subspace $M_\theta(x,y)$ of $\cuco X$ (endowed with the induced metric) so that $(M_\theta(x,y),\mathfrak S)$ is a hierarchically hyperbolic space (with the same nesting relations and projections from $(\cuco X,\mathfrak S)$, so that for all $U\in\mathfrak S$, we have that $\pi_U(M_\theta(x,y))\subset\neb_\theta(\gamma_U)$, where $\gamma_U$ is an arbitrarily-chosen geodesic in $\fontact U$ from $\pi_U(x)$ to $\pi_U(y)$.  We emphasize that all of the constants from Definition~\ref{defn:space_with_distance_formula} (for $M_\theta(x,y)$) are independent of $x,y$.  The theorem now follows by applying the distance formula for hierarchically hyperbolic spaces (Theorem~\ref{thm:distance_formula}) to $(M_\theta(x,y),\mathfrak S)$.
\end{proof}

\begin{thm}[Hierarchy paths in relatively hierarchically hyperbolic spaces]\label{thm:rhhs_hierarchy_paths}
Let $(\cuco X,\mathfrak S)$ be a relatively hierarchically hyperbolic space.  Then there exists $D\geq0$ such that for all $x,y\in\cuco X$, there exists a $(D,D)$--quasigeodesic $\gamma$ in $\cuco X$ joining $x,y$ so that $\pi_U(\gamma)$ is an unparameterized $(D,D)$--quasigeodesic.
\end{thm}

\begin{proof}
Proceed exactly as in Theorem~\ref{thm:rel_hyp_df}, but apply Theorem~\ref{thm:monotone_hierarchy_paths} instead of Theorem~\ref{thm:distance_formula}.
\end{proof}

We now define hulls of pairs of points in the relatively hierarchically hyperbolic space $(\cuco X,\mathfrak S)$.  Let $\theta$ be a constant to be chosen (it will be the output of the realization theorem for a consistency constant depending on the constants associated to $(\cuco X,\mathfrak S)$), and let $x,y\in\cuco X$.  For each $U\in\mathfrak S$, fix a geodesic $\gamma_U$ in $\fontact U$ joining $\pi_U(x)$ to $\pi_U(y)$.  Define maps $r_U:\fontact U\to\gamma_U$ as follows: if $\fontact U$ is hyperbolic, let $r_U$ be the coarse closest-point projection map.  Otherwise, if $\fontact U$ is not hyperbolic (so $U$ is $\nest$--minimal), define $r_U$ as follows: parametrize $\gamma_U$ by arc length with $\gamma_U(0)=x$, and for each $p\in\fontact U$, let $m(p)=\min\{\dist_U(x,p),\dist_U(x,y)\}$.  Then $r_U(p)=\gamma_U(m(p))$.  This $r_U$ is easily seen to be an $L$--coarsely Lipschitz retraction, with $L$ independent of $U$ and $x,y$.  (When $U$ is minimal, $r_U$ is $1$--Lipschitz.)

Next, define the hull $M_\theta(x,y)$ to be the set of points $x\in\cuco X$ such that $\dist_U(x,\gamma_U)\leq\theta$ for all $U\in\mathfrak S$.  In the next proposition, we show that $M_\theta(x,y)$ is a hierarchically hyperbolic space, with the following hierarchically hyperbolic structure:
\begin{enumerate}
 \item the index set is $\mathfrak S$;\label{item:hull_index}
 \item the nesting, orthogonality, and transversality relations on $\mathfrak S$ are the same as in $(\cuco X,\mathfrak S)$;\label{item:hull_rel}
 \item for each $U\in\mathfrak S$, the associated hyperbolic space is $\gamma_U$;\label{item:hull_space}
 \item for each $U\in\mathfrak S$, the projection $\pi'_U:M_\theta(x,y)\to\gamma_U$ is given by $\pi'_U=r_U\circ\pi_U$;\label{item:hull_proj}
 \item for each pair $U,V\in\mathfrak S$ of distinct non-orthogonal elements, the relative projection $\fontact U\to\fontact V$ is given by $r_V\circ\rho^U_V$.\label{item:hull_relproj}
\end{enumerate}

Since there are now two sets of projections (those defined in the original hierarchical space $(\cuco X,\mathfrak S)$, denoted $\pi_*$, and the new projections $\pi'_*$), in the following proofs we will explicitly write all projections when writing distances in the various $\fontact U$.  

\begin{lem}[Gates in hulls]\label{lem:hull_gate}
Let $M_\theta(x,y)$ be as above.  Then there exists a uniformly coarsely Lipschitz retraction $r:\cuco X\to M_\theta(x,y)$ such that for each $U\in\mathfrak S$, we have, up to uniformly (independent of $x,y$) bounded error, $\pi_U\circ r=r_U\circ\pi_U$.
\end{lem}

\begin{rem}It is crucial in the following proof that $\fontact U$ is $\delta$--hyperbolic for each $U\in\mathfrak S$ that is not $\nest$--minimal.\end{rem}

\begin{proof}[Proof of Lemma~\ref{lem:hull_gate}]
Let $z\in\cuco X$ and, for each $U$, let $t_U=r_U\circ\pi_U(z)$; this defines a tuple $(t_U)\in\prod_{U\in\mathfrak 
S}2^{\fontact U}$ which we will check is $\kappa$--consistent for $\kappa$ independent of $x,y$.  The tuple $(t_U)$ is 
admissible because of quasiconvexity of the images of projections to hyperbolic spaces, and coarse surjectivity of 
projections to non-hyperbolic ones.

Realization 
(Theorem~\ref{thm:realization}) then yields $m\in\cuco X$ such that $\dist_U(\pi_U(m),t_U)\leq\theta$ for all $U\in\mathfrak 
S$.  By definition, $t_U\in\gamma_U$, so $m\in M_\theta(x,y)$ and we define $\gate_{x,y}(z)=m$.  Note that up to perturbing 
slightly, we may take $\gate_{x,y}(z)=z$ when $z\in M_\theta$.  Hence it suffices to check consistency of $(t_U)$.

First let $U,V\in\mathfrak S$ satisfy $U\transverse V$.  Then $\dist_V(\pi_V(x),\pi_V(y))\leq 2E$ (up to exchanging $U$ and $V$), and moreover each of $\pi_U(x),\pi_U(y)$ is $E$--close to $\rho^U_V$.  Since $t_V$ lies on $\gamma_V$, it follows that $\dist_V(t_V,\rho^U_V)\leq2E$.

Next, let $U,V\in\mathfrak S$ satisfy $U\propnest V$.  Observe that in this case, $\fontact V$ is $\delta$--hyperbolic because $V$ is not $\nest$--minimal.  First suppose that $\dist_V(\gamma_V,\rho^U_V)>1$.  Then by consistency and bounded geodesic image, $\dist_U(x,y)\leq 3E$, and $\diam_U(\rho^V_U(\gamma_V))\leq E$.  It follows that $\diam_U(t_U\cup\rho^U_V(t_V))\leq 10E$.  

Hence, suppose that $\dist_V(\rho^U_V,\gamma_V)\leq 10E$ but that $\dist_V(t_V,\rho^U_V)>E$.  Without loss of generality, $\rho^U_V$ lies at distance $\leq E$ from the subpath of $\gamma_V$ joining $t_V$ to $\pi_V(y)$.  Let $\gamma'_V$ be the subpath joining $x$ to $t_V$.  By consistency, bounded geodesic image, and the fact that $\fontact V$ is $\delta$--hyperbolic and $t_V=r_V\circ\pi_V(z)$, the geodesic triangle between $\pi_V(x),\pi_V(z)$, and $t_V$ projects under $\rho^V_U$ to a set, of diameter bounded by some uniform $\xi$, containing $\pi_U(x),\pi_U(z)$, and $\rho^V_U(t_V)$.  Hence, since $t_U=r_U\circ\pi_U(z)$, and $\pi_U(x)\in\gamma_U$, the triangle inequality yields a uniform bound on $\diam_U(t_U\cup\rho^V_U(t_V))$.  Hence there exists a uniform $\kappa$, independent of $x,y$, so that $(t_U)$ is $\kappa$--consistent.  Finally, $\gate_{x,y}$ is coarsely Lipschitz by the uniqueness axiom (Definition~\ref{defn:space_with_distance_formula}.\eqref{item:dfs_uniqueness}), since each $r_U$ is uniformly coarsely Lipschitz.
\end{proof}

\begin{lem}\label{lem:qgs}
Let $m,m'\in M_\theta(x,y)$.  Then there exists $C\geq0$ such that $m,m'$ are joined by a $(C,C)$--quasigeodesic in $M_\theta(x,y)$.
\end{lem}

\begin{proof}
Since $\cuco X$ is a quasigeodesic space, there exists $K\geq0$ so that $m,m'$ are joined by a $K$--discrete $(K,K)$--quasigeodesic $\sigma:[0,\ell]\to\cuco X$ with $\sigma(0)=m,\sigma(\ell)=m'$.  Note that $\gate_{x,y}\circ\sigma$ is a $K'$--discrete, efficient path for $K'$ independent of $x,y$, since the gate map is uniformly coarsely Lipschitz.  A minimal-length $K'$--discrete efficient path in $M_\theta(x,y)$ from $x$ to $y$ has the property that each subpath is $K'$--efficient, and is a uniform quasigeodesic, as needed.
\end{proof}

\begin{prop}\label{prop:hulls}
For all sufficiently large $\theta$, the data \eqref{item:hull_index}--\eqref{item:hull_relproj} above makes $(M_\theta(x,y),\mathfrak S)$ a hierarchically hyperbolic space, where $M_\theta(x,y)$ inherits its metric as a subspace of $\cuco X$.  Moreover, the associated constants from Definition~\ref{defn:space_with_distance_formula} are independent of $x,y$.
\end{prop}

\begin{proof}
By Lemma~\ref{lem:qgs}, $M_\theta(x,y)$ is a uniform quasigeodesic space. We now verify that the enumerated axioms from Definition~\ref{defn:space_with_distance_formula} are satisfied.  Each part of the definition involving only $\mathfrak S$ and the $\nest,\orth,\transverse$ relations is obviously satisfied; this includes finite complexity.  The consistency inequalities hold because they hold in $(\cuco X,\mathfrak S)$ and each $r_U$ is $L$--coarsely Lipschitz.  The same holds for bounded geodesic image and the large link lemma.  We now verify the two remaining claims:

\textbf{Uniqueness:}  Let $m,m'\in M_\theta(x,y)$, so that $\dist_U(\pi_U(m),\gamma_U),\dist_U(\pi_U(m'),\gamma_U)\leq\theta$ for all $U\in\mathfrak S$.  The definition of $r_U$ ensures that $\dist_U(r_U\circ\pi_U(m),r_U\circ\pi_U(m'))\geq\dist_U(\pi_U(m),\pi_U(m'))-2\theta$, and uniqueness follows.

\textbf{Partial realization:}  Let $\{U_i\}$ be a totally orthogonal subset of $\mathfrak S$ and choose, for each $i$, some $p_i\in\gamma_{U_i}$.  By partial realization in $(\cuco X,\mathfrak S)$, there exists $z\in\cuco X$ so that $\dist_{U_i}(\pi_{U_i}(z),p_i)\leq E$ for each $i$ and $\dist_V(\pi_V(z),\rho^{U_i}_V)\leq E$ provided $U_i\propnest V$ or $U_i\transverse V$.  Let $z'=\gate_{x,y}(z)\in M_\theta(x,y)$.  Then, by the definition of the gate map and the fact that each $r_U$ is $L$--coarsely Lipschitz, there exists $\alpha$, independent of $x,y$, so that $\dist_{U_i}(r_{U_i}\circ\pi_{U_i}(z'),p_i)\leq\alpha$, while $\dist_V(r_V\circ\pi_V(z'),\rho^{U_i}_V)\leq\alpha$ whenever $U_i\transverse V$ or $U_i\nest V$.  Hence $z'$ is the required partial realization point.  
This completes the proof that $(M_\theta(x,y),\mathfrak S)$ is an HHS.
\end{proof}

\section{The coarse median property}\label{sec:coarse_media}
In this section, we study the relationship between hierarchically
hyperbolic spaces and spaces that are \emph{coarse median} in the
sense defined by Bowditch in~\cite{Bowditch:coarse_median}.  In
particular, this discussion shows that $\Out(F_n)$ is not a
hierarchically hyperbolic space, and hence not a hierarchically
hyperbolic group, for $n\geq 3$.

\begin{defn}[Median graph]\label{defn:median_graph}
Let $\Gamma$ be a graph with unit-length edges and path-metric
$\dist$.  Then $\Gamma$ is a \emph{median graph} if there is a map
$\median\co \Gamma^3\rightarrow\Gamma$ such that, for all
$x,y,z\in\Gamma$, we have $\dist(x,y)=\dist(x,m)+\dist(m,y)$, and
likewise for the pairs $x,z$ and $y,z$, where $m=\median(x,y,z)$.
Note that if $x=y$, then $\median(x,y,z)=x$.
\end{defn}

Chepoi established in~\cite{Chepoi:cube_median} that $\Gamma$ is a median graph precisely when $\Gamma$ is the $1$-skeleton of a CAT(0) cube complex. 

\begin{defn}[Coarse median space]\label{defn:coarse_median_space}
Let $(M,\dist)$ be a metric space and let $\median\co M^3\rightarrow M$ be a ternary operation satisfying the following:
\begin{enumerate}
 \item\label{cms_item:triples} (Triples)  There exist constants $\kappa,h(0)$ such that for all $a,a',b,b',c,c'\in M$, $$\dist(\median(a,b,c,),\median(a',b',c'))\leq\kappa\left(\dist(a,a')+\dist(b,b')+\dist(c,c')\right)+h(0).$$
 \item\label{cms_item:tuples} (Tuples)  There is a function 
 $h\co\naturals\cup\{0\}\rightarrow[0,\infty)$ such that for any 
 $A\subseteq M$ with $1\leq|A|=p<\infty$, there is a CAT(0) cube 
 complex $\cuco F_p$ and maps $\pi\co A\rightarrow\cuco F_p^{(0)}$ and 
 $\lambda\co\cuco F_p^{(0)}\rightarrow M$ such that $\dist(a,\lambda(\pi(a)))\leq h(p)$ for all $a\in A$ and such that $$\dist(\lambda(\median_p(x,y,z)),\median(\lambda(x),\lambda(y),\lambda(z)))\leq h(p)$$ for all $x,y,z\in\cuco F_p$, where $\median_p$ is the map that sends triples from $\cuco F_p^{(0)}$ to their median.
\end{enumerate}
Then $(M,\dist,\median)$ is a \emph{coarse median space}.  The \emph{rank} of $(M,\dist,\median)$ is at most $d$ if each $\cuco F_p$ above can be chosen to satisfy $\dimension\cuco F_p\leq d$, and the rank of $(M,\dist,\median)$ is exactly $d$ if $d$ is the minimal integer such that $(M,\dist,\median)$ has rank at most $d$.
\end{defn}

The next fact was observed by Bowditch~\cite{Bowditch:large_scale}; we include a proof for completeness.

\begin{thm}[Hierarchically hyperbolic implies coarse median]\label{thm:hier_hyp_coarse_median}
Let $(\cuco X,\mathfrak S)$ be a hierarchically hyperbolic space.
Then $\cuco X$ is coarse median of rank at most the complexity
of $(\cuco X,\mathfrak S)$.
\end{thm}

\begin{proof}
Since the spaces $\fontact U,U\in\mathfrak S$ are $\delta$-hyperbolic 
for some $\delta$ independent of $U$, there exists for each $U$ a 
ternary operation $\median^U\co\fontact U^3\rightarrow\fontact U$ so 
that $(\fontact U,\dist_{ U},\median^U)$ is a coarse median space of rank $1$, and the constant $\kappa$ and function $h\co\naturals\cup\{0\}\rightarrow[0,\infty)$ from Definition~\ref{defn:coarse_median_space} can be chosen to depend only on $\delta$ (and not on $U$).

\textbf{Definition of the median:}  Define a map $\median\co\cuco X^3\rightarrow\cuco X$ as follows.  Let $x,y,z\in\cuco X$ 
and, for each $U\in\mathfrak S$, let $b_U=\median^U(\pi_U(x),\pi_U(y),\pi_U(z))$.  By Lemma~\ref{lem:median_cons}, the tuple 
$\tup b\in\prod_{U\in\mathfrak S}2^{\fontact U}$ whose $U$-coordinate is $b_U$ is $\kappa$--consistent for an 
appropriate choice of $\kappa$.  Hence, by the realization theorem (Theorem~\ref{thm:realization}), there exists $\theta_e$ 
and $\median=\median(x,y,z)\in\cuco X$ such that $\dist_U(\median,b_U)\leq\theta_u$ for all $U\in\mathfrak S$.  Moreover, 
this is coarsely well-defined (up to the constant $\theta_e$ from the realization theorem). 

\textbf{Application
of~\cite[Proposition~10.1]{Bowditch:coarse_median}:} Observe that, by
Definition \ref{defn:space_with_distance_formula}.\ref{item:dfs_curve_complexes}, the
projections $\pi_U\colon\cuco X\rightarrow\fontact U,\,U\in\mathfrak
S$ are uniformly coarsely Lipschitz.  Moreover, for each $U\in\mathfrak S$,
the projection $\pi_U\colon\cuco X\rightarrow \fontact U$ is a
``quasimorphism'' in the sense
of~\cite[Section~10]{Bowditch:coarse_median}, i.e., $\dist_{
U}(\median^U(\pi_U(x),\pi_U(y),\pi_U(z)),\pi_U(\median(x,y,z)))$ is
uniformly bounded, by construction, as $U$ varies over $\mathfrak S$
and $x,y,z$ vary in $\cuco X$.  Proposition~10.1
of~\cite{Bowditch:coarse_median} then implies that $\median$ is a coarse median on
$\cuco X$, since that the hypothesis~$(P1)$ 
of that proposition holds in our situation by the distance
formula.
\end{proof}

The following is a consequence of 
Theorem~\ref{thm:hier_hyp_coarse_median} and work of Bowditch 
\cite{Bowditch:coarse_median, Bowditch:embeddings}:

\begin{cor}[Contractibility of asymptotic cones]\label{cor:HHS_contractible_cone}
Let $\cuco X$ be a hierarchically hyperbolic space. Then all the asymptotic cones of $\cuco X$ are contractible, and in fact bi-Lipschitz equivalent to CAT(0) spaces.
\end{cor}

\begin{cor}[HHGs have quadratic Dehn function]\label{cor:HHG_quadratic_Dehn_function}
Let $G$ be a finitely generated group that is a
hierarchically hyperbolic space.  Then $G$ is finitely presented and
has quadratic Dehn function.  In particular, this conclusion holds
when $G$ is a hierarchically hyperbolic group.
\end{cor}

\begin{proof}
This follows from Theorem~\ref{thm:hier_hyp_coarse_median} and~\cite[Corollary~8.3]{Bowditch:coarse_median}.
\end{proof}

\begin{cor}\label{cor:out_fn}
For $n\geq 3$, the group $\Out(F_n)$ is not a hierarchically hyperbolic space, and in particular is not a hierarchically hyperbolic group.
\end{cor}

\begin{proof}
This is an immediate consequence of
Corollary~\ref{cor:HHG_quadratic_Dehn_function} and 
the exponential lower bound on the Dehn function
of $\Out(F_n)$ given by the combined results of 
\cite{BridsonVogtmann:dehn_1,HandelMosher:dehn,BridsonVogtmann:dehn_2}.
\end{proof}

We also recover a special case of Theorem~I of~\cite{BehrstockHagenSisto:HHS_I}, using Corollary~\ref{cor:HHG_quadratic_Dehn_function} and a theorem of Gersten-Holt-Riley~\cite[Theorem~A]{GerstenRiley:nilpotent}:

\begin{cor}\label{cor:nilpotent}
Let $N$ be a finitely generated virtually nilpotent group.  Then $G$ is quasi-isometric to a hierarchically hyperbolic space if and only if $N$ is virtually abelian.
\end{cor}

\begin{cor}\label{cor:symmetric}
Let $S$ be a symmetric space of non-compact type, or a thick affine
building.  Suppose that the spherical type of $S$ is not $A^r_1$.
Then $S$ is not hierarchically hyperbolic.
\end{cor}

\begin{proof}
This follows from Theorem~\ref{thm:hier_hyp_coarse_median} and
Theorem~A of~\cite{Haettel:symmetric}.
\end{proof}

Finally, Theorem~9.1 of~\cite{Bowditch:embeddings} combines with Theorem~\ref{thm:hier_hyp_coarse_median} to yield:

\begin{cor}[Rapid decay]\label{cor:rapid_decay}
Let $G$ be a group whose Cayley graph is a hierarchically hyperbolic
space.
Then $G$ has
the rapid decay property.
\end{cor}

\subsection{Coarse media and hierarchical quasiconvexity}\label{subsec:coarse_median_quasiconvex}
The natural notion of quasiconvexity in the coarse median setting is related to hierarchical quasiconvexity.  

\begin{defn}[Coarsely convex]\label{defn:median_closed}
Let $(\cuco X,\mathfrak S)$ be a hierarchically hyperbolic space and let $\median\co\cuco X^3\rightarrow\cuco X$ be the coarse median map constructed in the proof of Theorem~\ref{thm:hier_hyp_coarse_median}.  A closed subspace $\cuco Y\subseteq\cuco X$ is \emph{$\mu$--convex} if for all $y,y'\in\cuco Y$ and $x\in\cuco X$, we have $\median(y,y',x)\in\neb_\mu(\cuco Y)$.  
\end{defn}

\begin{rem}
We will not use $\mu$--convexity in the remainder of the paper.  However, it is of independent interest since it parallels a characterization of convexity in median spaces: a subspace $\cuco Y$ of a median space is convex exactly when, for all $y,y'\in\cuco Y$ and $x$ in the ambient median space, the median of $x,y,y'$ lies in $\cuco Y$.
\end{rem}

\begin{prop}[Coarse convexity and hierarchical quasiconvexity]\label{prop:coarse_median_hier_convex}
Let $(\cuco X,\mathfrak S)$ be a hierarchically hyperbolic space and let $\cuco Y\subseteq\cuco X$.  If $\cuco Y$ is hierarchically $k$-quasiconvex, then there exists $\mu\geq 0$, depending only on $k$ and the constants from Definition~\ref{defn:space_with_distance_formula}, such that $\cuco Y$ is $\mu$--convex.
\end{prop}

\begin{proof}Let $\cuco Y\subseteq\cuco X$ be $k$--hierarchically quasiconvex, let $y,y'\in\cuco Y$ and $x\in\cuco X$.  Let $m=\median(x,y,y')$.  For any $U\in\mathfrak S$, the projection $\pi_U(\cuco Y)$ is by definition $k(0)$-quasiconvex, so that, for some $k'=k'(k(0),\delta)$, we have $\dist_{ U}(m_U,\pi_U(\cuco Y))\leq k'$, where $m_U$ is the coarse median of $\pi_U(x),\pi_U(y),\pi_U(y')$ coming from hyperbolicity of $\fontact U$.  The tuple $(m_U)_{U\in\mathfrak S}$ was shown above to be $\kappa$--consistent for appropriately-chosen $\kappa$ (Lemma~\ref{lem:median_cons}), and $\dist_U(m_U,\median(x,y,y'))\leq\theta_e(\kappa)$, so, by hierarchical quasiconvexity $\dist_{\cuco X}(\median(x,y,y'),\cuco Y)$ is bounded by a constant depending on $k(\kappa)$ and $k'$. 
\end{proof}

\section{Combination theorems for hierarchically hyperbolic spaces}\label{sec:combination}
The goal of this section is to prove Theorem~\ref{thm:combination}, which enables the construction of new hierarchically hyperbolic spaces and groups from a tree of given ones.  We postpone the statement of the theorem until after the relevant definitions.

\begin{defn}[Quasiconvex hieromorphism, full hieromorphism]\label{defn:quasiconvex_full_hieromorphism}
Let 
$(f,f\inducedS,\{f\induced(U)\}_{U\in\mathfrak S})$ be a 
hieromorphism $(\cuco X,\mathfrak S)\rightarrow(\cuco
X',\mathfrak S')$. 
We say $f$ is
\emph{$k$--hierarchically quasiconvex} if its image is $k$--hierarchically quasiconvex and $f\co\cuco X\to\cuco X'$ is a quasi-isometric embedding.  The
hieromorphism is \emph{full} if:
\begin{enumerate}
 \item there exists $\xi\geq 0$ such that each $f\induced(U)\co\fontact
 U\to\fontact(f\inducedS (U))$ is a $(\xi,\xi)$--quasi-isometry, and
 \item for each $U\in\mathfrak S$, if $V'\in\mathfrak S'$ satisfies
 $V'\nest f\inducedS (U)$, then there exists $V\in\mathfrak S$ such that
 $V\nest U$ and $f\inducedS (V)=V'$.\label{item:2_full}
\end{enumerate}  
\end{defn}

\begin{rem}\label{rem:full}
Observe that Definition~\ref{defn:quasiconvex_full_hieromorphism}.\eqref{item:2_full} holds automatically unless $V'$ is bounded.
\end{rem}

\begin{defn}[Tree of hierarchically hyperbolic spaces]\label{defn:tree_of_HHS}
    Let $\mathcal V,\mathcal E$ denote the vertex and edge-sets, respectively, of the
    simplicial tree $T$. 
    A \emph{tree of hierarchically hyperbolic spaces} is a quadruple
$$\mathcal T=\left(T,\{\cuco X_v\},\{\cuco X_e\},
\{\phi_{e_{\pm}}:v\in\mathcal V,e\in\mathcal E\}\right)$$ satisfying:
\begin{enumerate}
 \item $\{\mathcal X_v\}$ and $\{\mathcal X_e\}$ are uniformly hierarchically hyperbolic: each $\mathcal X_v$ has index set $\mathfrak S_v$, and each $\mathcal X_e$ has index set $\mathfrak S_e$.  In particular, there is a uniform bound on the complexities of the hierarchically hyperbolic structures on the $\mathcal X_v$ and $\mathcal X_e$.
 \item Fix an orientation on each $e\in\mathcal E$ and let  $e_+,e_-$ 
 denote the initial and terminal vertices of $e$.  Then, each
 $\phi_{e_{\pm}}\co\cuco X_e\to \cuco X_{e_\pm}$ is a hieromorphism with
 all constants bounded by some uniform $\xi\geq 0$.  (We adopt the
 hieromorphism notation from Definition~\ref{defn:hieromorphism}.
 Hence we actually have maps $\phi_{e_\pm}\co\cuco X_e\to\cuco
 X_{e_\pm}$, and maps $\phi_{e_\pm}\inducedS \co\mathfrak S_e\to\mathfrak
 S_{e_\pm}$ preserving nesting, transversality, and orthogonality, and
 coarse $\xi$--Lipschitz maps $\phi_{e_\pm}\induced (U)\co\fontact
 U\to\fontact(\phi_{e_\pm}\inducedS (U))$ satisfying the conditions of that
 definition.)
\end{enumerate}
Given a tree $\mathcal T$ of hierarchically hyperbolic spaces, denote
by $\cuco X(\mathcal T)$ the metric space constructed from
$\bigsqcup_{v\in\mathcal V} \cuco X_v$ by adding edges of length $1$ as
follows: if $x\in \cuco X_e$, we declare $\phi_{e_-}(x)$ to be joined by an
edge to $\phi_{e_+}(x)$.  Given $x,x'\in\cuco X$ in the same vertex
space $\cuco X_v$, define $\dist'(x,x')$ to be $\dist_{\cuco X_v}(x,x')$.  Given $x,x'\in\cuco X$ joined by an edge, define
$\dist'(x,x')=1$.  Given a sequence $x_0,x_1,\ldots,x_k\in\cuco X$,
with consecutive points either joined by an edge or in a common vertex
space, define its length to be $\sum_{i=1}^{k-1}\dist'(x_i,x_{i+1})$.
Given $x,x'\in\cuco X$, let $\dist(x,x')$ be the infimum of the
lengths of such sequences $x=x_0,\ldots,x_k=x'$.
\end{defn}

\begin{rem}\label{rem:tree_quasigeodesic}
Since the vertex spaces are (uniform) quasigeodesic spaces, $(\cuco X,\dist)$ is a quasigeodesic space.
\end{rem}

\begin{defn}[Equivalence, support, bounded support]\label{defn:equivalence}
Let $\mathcal T$ be a tree of hierarchically hyperbolic spaces.  For
each $e\in\mathcal E$, and each $W_{e_-}\in\mathfrak
S_{e_-},W_{e_+}\in\mathfrak S_{e_+}$, write $W_{e_-}\sim_d W_{e_+}$ if
there exists $W_e\in\mathfrak S_e$ so that
$\phi_{e_\pm}\inducedS (W_e)=W_{e_{\pm}}$. 
The transitive closure
$\sim$ of $\sim_d$ is an equivalence relation on $\bigcup_v\mathfrak
S_v$.  The $\sim$--class of $W\in\bigcup_v\mathfrak S_v$ is denoted
$[W]$.

The \emph{support} of an equivalence class $[W]$ is the induced subgraph $T_{[W]}$ of $T$ whose vertices are those $v\in T$ so that $\mathfrak S_v$ contains a representative of $[W]$.  Observe that $T_{[W]}$ is connected.  The tree $\mathcal T$ of hierarchically hyperbolic spaces has \emph{bounded supports} if there exists $n\in\naturals$ such that each $\sim$--class has support of diameter at most $n$.
\end{defn}

We can now state the main theorem of this section:

\begin{thm}[Combination theorem for hierarchically hyperbolic spaces]\label{thm:combination}
 Let $\mathcal T$ be a tree of hierarchically hyperbolic spaces.  Suppose that:
 \begin{enumerate}
  \item there exists a function $k$ so that each edge--hieromorphism is $k$--hierarchically quasiconvex;
  \item each edge--hieromorphism is full;
  \item $\mathcal T$ has bounded supports of diameter at most~$n$;\label{item:bounded_proj} 
  \item \label{item:hypothesis_4} if $e$ is an edge of $T$ and $S_e$
  is the $\nest$--maximal element of $\mathfrak S_e$, then for all
  $V\in\mathfrak S_{e^\pm}$, the elements $V$ and
  $\phi_{e^\pm}\inducedS (S_e)$ are not orthogonal in $\mathfrak
  S_{e^\pm}$.  Moreover, there exists $K\ge0$ such that for all
  vertices $v$ of $\mathcal T$ and edges $e$ incident to $v$, we have
  $\dist_{Haus}(\phi_v(\cuco X_e)),\mathbf
  F_{\phi_v\inducedS(S_e)}\times\{\star\})\leq K$, where
  $S_e\in\mathfrak S_e$ is the unique maximal element and
  $\star\in\mathbf E_{\phi_v\inducedS(S_e)}$.
 \end{enumerate}
Then $\cuco X(\mathcal T)$ is hierarchically hyperbolic.
\end{thm}

We postpone the proof until after the necessary lemmas and
definitions.  For the remainder of this section, fix a tree of hierarchically hyperbolic spaces $\mathcal
T=\left(T,\{\cuco X_v\},\{\cuco X_e\},\{\phi_{e^\pm}\}\right)$
 satisfying the hypotheses of
Theorem~\ref{thm:combination}; let $n$ be the constant implicit in assumption \ref{item:bounded_proj}.

Let $\mathfrak S^0=\{T\}\cup\left(\bigcup_v\mathfrak
S_v/\sim\right)$.

\begin{defn}[Nesting, orthogonality, transversality in $\mathfrak S^0$]\label{defn:nest_etc_on_classes}
For all $[W]\in\mathfrak S$, declare $[W]\nest T$.  If $[V],[W]$ are $\sim$--classes, then $[V]\nest[W]$ if and only 
if there exists $v\in T$ such that $[V],[W]$ are respectively represented by $V_v,W_v\in\mathfrak S_v$ and $V_v\nest 
W_v$; this relation is \emph{nesting}.  For convenience, for $A\in\mathfrak S$, we write $\mathfrak S_{A}$ to  denote 
the set of $B\in\mathfrak S^0$ such that $B\nest A$.

Likewise, $[V]\orth[W]$ if and only if there exists a vertex $v\in T$ such that $[V],[W]$ are respectively 
represented by $V_v,W_v\in\mathfrak S_v$ and $V_v\orth W_v$; this relation is \emph{orthogonality}.  If 
$[V],[W]\in\mathfrak S$ are not orthogonal and neither is nested into the other, then they are \emph{transverse}, 
written $[V]\transverse[W]$.  Equivalently, $[V]\transverse[W]$ if for all $v\in T_{[V]}\cap T_{[W]}$, the 
representatives $V_v,W_v\in\mathfrak S_v$ of $[V],[W]$ satisfy $V_v\transverse W_v$.
\end{defn}

Fullness (Definition~\ref{defn:quasiconvex_full_hieromorphism}.\eqref{item:2_full}) was introduced to enable the following two lemmas:

\begin{lem}\label{lem:no_cycle}
Let $\mathcal T$ be a tree of hierarchically hyperbolic spaces, let $v$ be a vertex of the underlying tree $T$, and let $U,U'\in\mathfrak S_v$ satisfy $U\nest U'$.  Then either $U=U'$ or $U\not\sim U'$.
\end{lem}

\begin{proof}
Suppose that $U\sim U'$, so that there is a closed path
$v=v_0,v_1,\ldots,v_n=v$ in $T$ and a sequence
$U=U_0,U_1,\ldots,U_n=U'$ such that $U_i\in\mathfrak S_{v_i}$ and
$U_i\sim_d U_{i+1}$ for all $i$.  If $U\neq U'$, then
Condition~\eqref{item:2_full} (fullness) from
Definition~\ref{defn:quasiconvex_full_hieromorphism} and the fact that
hieromorphisms preserve nesting yields $U''\in\mathfrak S_v$,
different from $U'$, such that $U''\sim U$ and $U''\propnest
U'\propnest U$ (where $\propnest$ denotes proper nesting).  Repeating
this argument contradicts finiteness of complexity.
\end{proof}

\begin{lem}\label{lem:comb_nest}
The relation $\nest$ is a partial order on $\mathfrak S^0$, and $T$ is the unique $\nest$--maximal element.  
Moreover, if $[V]\orth [W]$ and $[U]\nest[V]$, then $[U]\orth[W]$ and $[V],[W]$ are not $\nest$--comparable. 
\end{lem}

\begin{proof}
Reflexivity is clear.  Suppose that $[V_v]\nest[U_u]\nest[W_w]$.  Then there are vertices $v_1,v_2\in\mathcal V$ and 
representatives $V_{v_1}\in[V_v],U_{v_1}\in[U_u],U_{v_2}\in[U_u],W_{v_2}\in[W_w]$ so that $V_{v_1}\nest U_{v_1}$ and 
$U_{v_2}\nest W_{v_2}$.  Since edge--hieromorphisms are full, induction on $\dist_T(v_1,v_2)$ yields $V_{v_2}\nest 
U_{v_2}$ so that $V_{v_2}\sim V_{v_1}$.  Transitivity of the nesting relation in $\mathfrak S_{v_2}$ implies that 
$V_{v_2}\nest W_{v_2}$, whence $[V_v]\nest[W_w]$.   

Suppose that $[U_u]\nest[V_v]$ and $[V_v]\nest[U_u]$, and suppose 
by contradiction that $[U_u]\neq[V_v]$.  Choose $v_1,v_2\in\mathcal V$ and representatives 
$U_{v_1},U_{v_2},V_{v_1},V_{v_2}$ so that $U_{v_1}\nest V_{v_1}$ and $V_{v_2}\nest U_{v_2}$.  The definition of 
$\sim$ again yields $U_{v_2}\sim U_{v_1}$ with $U_{v_2}\nest V_{v_2}\neq U_{v_2}$.  This contradicts 
Lemma~\ref{lem:no_cycle}.  Hence $\nest$ is antisymmetric, whence it is a partial order.  The underlying tree 
$T$ is the unique $\nest$--maximal element by definition.

Suppose that $[V]\orth[W]$ and $[U]\nest[V]$.  Then there are vertices $v_1,v_2$ and representatives 
$V_{v_1},W_{v_1},U_{v_2},V_{v_2}$ such that $V_{v_1}\orth W_{v_1}$ and $U_{v_2}\orth V_{v_2}$.  Again by fullness of 
the edge--hieromorphisms, there exists $U_{v_1}\sim U_{v_2}$ with $U_{v_1}\nest V_{v_1}$, whence $U_{v_1}\orth 
W_{v_1}$.  Thus $[U]\orth [W]$ as required.  Also, $\nest$--incomparability of $[V],[W]$ follows from fullness and 
the fact that edge-hieromorphisms preserve orthogonality and nesting.
\end{proof}

\begin{lem}\label{lem:orth_comp}
Let $[W]\in\mathfrak S^0$ and let $[U]\nest[W]$.  Suppose moreover that  $$\left\{[V]\in\mathfrak 
S_{[W]}:[V]\orth[U]\right\}\neq\emptyset.$$  Then there exists $[A]\in\mathfrak S_{[W]}-\{[W]\}$ such that $[V]\nest 
[A]$ for all $[V]\in\mathfrak S_{[W]}$ with $[V]\orth[U]$.
\end{lem}

\begin{proof}
Choose some $v\in\mathcal V$ so that there exists $V_v\in\mathfrak S_v$ and $U_v\in\mathfrak S_v$ with $[U_v]=[U]$ 
and $V_v\orth U_v$.  Then by definition, there exists $A_v\in\mathfrak S_v$ so that $B_v\nest A_v$ whenever $B_v\orth 
U_v$ and so that $[B_v]\nest[W]$.  It follows from the fact that the edge-hieromorphisms are full and preserve 
(non)orthogonality that $[B]\nest[A_v]$ whenever $[B]\orth[U]$.  
\end{proof}

The set $\mathfrak S^0$ is not quite large enough to satisfy the orthogonality axiom, for the following reason: in 
Lemma~\ref{lem:orth_comp}, we needed $[W]$ to be a $\sim$--class, but since $T\in\mathfrak S^0$, we need to be able 
to satisfy the axiom with $[W]$ replaced by $T$.  To this end, we add some new elements to $\mathfrak S^0$, and 
extend the $\nest,\orth,\transverse$ relations, as follows.

\newcommand{\container}{\mathrm{K}^{\orth}}

\begin{defn}[Containers and $\mathfrak S$]\label{defn:containers_S}
We now define the index set $\mathfrak S$ for the HHS structure we will construct in order to prove 
Theorem~\ref{thm:combination}.  First, $\mathfrak S$ contains $\mathfrak S^0$.  Next, for each $[W]$ for which there 
exists $[U]$ with $[U]\orth[W]$, let $\container_0([W])$ be a new element of $\mathfrak S$, which we call the 
\emph{container} of $[W]$.  We make the following declarations:

\begin{itemize}
     \item $\container_0([W])\nest T$;
     \item $[U]\nest\container_0([W])$ if and only if $[U]\orth[W]$;
     \item $\container_0([W])\transverse\container_0([U])$ if $[U]\neq[W]$;
     \item $\container_0([W])\orth[V]$ if and only if $[V]\nest[W]$;
     \item for all other $[U]$, we have $[U]\transverse\container_0([W])$.
\end{itemize}

Let $\mathcal K_0$ be the set of all $\container_0([W])$ as $[W]$ varies among those $\sim$--classes for which there 
is at least one orthogonal $\sim$--class.

Next, for each $\container_0([W])\in\mathcal K_0$, consider a $\sim$--class $[U]\nest\container_0([W])$ such that 
$[U]\orth[V]$ for some other $[V]\nest\container_0([W])$.  Let $\container_1([W],[U])$ be a new element of 
$\mathfrak S$, and let $\mathcal K_1$ be the set of such containers, as $[W]$ varies and as $[U]$ varies over those 
$\sim$--classes nested in $\container_0([W])$ (i.e. orthogonal to $[W]$) that are orthogonal to some other 
$\sim$--class nested in $\container_0([W])$.  

We now make the following declarations:
\begin{itemize}
     \item $\container_1([W],[U])\nest\container_0([W])$ and $\container_1([W],[U])$ is transverse to every other 
element of $\mathcal K_0\cup\mathcal K_1$.
    \item $\container_1([W],[U])\nest T$.
    \item $[V]\nest\container_1([W],[U])$ if and only if $[V]\nest\container_0([W])$ and $[V]\orth[U]$.
    \item $[V]\orth\container_1([W],[U])$ if and only if either $[V]\nest [W]$ (i.e. $[V]\orth\container_0([W])$) 
or $[V]\nest[U]$.  
\item If none of the two preceding conditions is satisfied by $[V]$, then 
$[V]\transverse\container_1([W],[U])$.
\end{itemize}
We now proceed as above to inductively construct sets $\mathcal K_\eta,\eta\ge1$ of new ``containers'', where each 
$\container_\eta([W]_1,\ldots,[W_\eta])$ is nested in $\container_i([W]_1,\ldots,[W_i])$ for $i\leq \eta-1$, and 
also nested 
in $T$.  Our inductive construction ensures that $[W_1],\ldots,[W_\eta]$ are pairwise-orthogonal.  The 
$\sim$--classes 
$[U]$ nested in $\container_\eta([W]_1,\ldots,[W_\eta])$ are precisely those that are orthogonal to each of 
$[W_1],\ldots,[W_\eta]$.  The $\sim$--classes $U$ orthogonal to $\container_\eta([W_1],\ldots,[W_\eta])$ are 
precisely those 
$[U]$ nested into some $[W_i]$.

Let $\mathfrak S=\mathfrak S^0\cup\bigcup_{\eta\ge0}\mathcal K_\eta$.
\end{defn}

\begin{rem}[Extension of $\nest,\orth,\transverse$ satisfies the axioms]\label{rem:extending_relations}
Lemma~\ref{lem:comb_nest} shows that $\nest$ is a partial order on $\mathfrak S^0$, and 
Definition~\ref{defn:containers_S} shows how to extend $\nest$ to all of $\mathfrak S$.  By construction, the 
extended $\nest$ continues to be transitive.  This follows from Lemma~\ref{lem:comb_nest}, the definition, and 
induction on the $\eta$ in $\mathcal K_\eta$.  By definition, $T$ is still the unique $\nest$--maximal element.

Now suppose that $[U]\nest\container_\eta([W_1],\ldots,[W_\eta])$ and $[V]\orth 
\container_\eta([W_1],\ldots,[W_\eta])$.  Then  
$[V]\nest[W_i]$ for some $i$, and $[U]\orth[W_j]$ for all $j$.  Lemma~\ref{lem:comb_nest} implies $[U]\orth[V]$.  On 
the other hand, $\container_\eta([W_1],\ldots,[W_\eta])$ is never nested into any $\sim$--class or orthogonal to any 
element of $\bigcup_\eta\mathcal K_\eta$.
\end{rem}

\begin{lem}\label{lem:comb_complexity}
There exists $\chi\geq0$ such that if $\{V_1,\ldots,V_c\}\subset\mathfrak S$ consists of pairwise orthogonal or 
pairwise $\nest$--comparable elements, then $c\leq\chi$.  In particular, $\bigcup_{\eta\ge0}\mathcal 
K_\eta=\bigcup_{\eta=0}^{(\chi-1)/2}\mathcal K_\eta$.
\end{lem}

\begin{proof}
For each $v\in T$, let $\chi_v$ be the complexity of $(\cuco
X_v,\mathfrak S_v)$ and let $\chi=2\max_v\chi_v+1$.  Let
$[V_1],\ldots,[V_c]\in\mathfrak S-\{T\}$ be $\sim$--classes that are
pairwise orthogonal or pairwise $\nest$--comparable.  The Helly
property for trees yields a vertex $v$ lying in the support of each
$[V_i]$; let $V^v_i\in\mathfrak S_{v}$ represent $[V_i]$.  Since
edge--hieromorphisms preserve nesting, orthogonality, and
transversality, $c\leq\chi_v$.  

Any pairwise-orthogonal set in $\mathfrak S$ either has cardinality $\le1$ or contains at most one element that is 
not a $\sim$--class, so the bound on pairwise-orthogonal sets is $\max_v\chi_v+1$.  

Hence it suffices to bound $\nest$--chains in $\mathfrak S$.  Any $\nest$--chain $V_1\nest V_2\nest\cdots\nest V_k$ 
has the property that, for some $0\le m\le k$, the first $m$ elements are $\sim$--classes, and the remaining 
elements lie in $\{T\}\cup\bigcup_\eta\mathcal K_\eta$.  Hence it suffices to show that any $\nest$--chain in 
$\bigcup_{\eta\ge0}\mathcal K_\eta$ has length at most $\chi$.  But by definition, any such chain has the form 
$$\container_0([W_0])\sqsupset\container_1([W_0],[W_1])\sqsupset\cdots\sqsupset\container_\eta([W_0],\ldots,[W_{
\eta-1}] ),
$$ where the $[W_i]$ are pairwise orthogonal.  Hence $\eta\leq\max_v\chi_v\leq(\chi-1)/2$, as required.  This also 
proves the final assertion.
\end{proof}

\begin{defn}[Favorite representative, hyperbolic spaces associated to elements of $\mathfrak 
S$]\label{defn:favorites}
Let $\fontact T=T$.  For each $\sim$--class $[W]$, choose a \emph{favorite vertex} $v$ of $T_{[W]}$ and let 
$W_v\in\mathfrak S_{W_v}$ be the \emph{favorite representative} of $[W]$.  Let $\fontact [W]=\fontact W_v$.  Note 
that each $\fontact [W]$ is $\delta$--hyperbolic, where $\delta$ is the uniform hyperbolicity constant for $\mathcal 
T$.

Finally, for each $\container\in\bigcup_\eta\mathcal K_\eta$, let $\fontact\container$ be a single 
point.
\end{defn}

\begin{defn}[Gates in vertex spaces]\label{defn:tree_gate}
For each vertex $v$ of $T$, define a \emph{gate map} $\gate_v\co\cuco
X \to \cuco X_v$ as follows.  Let $x\in\cuco X_u$ for some vertex $u$
of $T$.  We define $\gate_v(x)$ inductively on $\dist_T(u,v)$.  If
$u=v$, then set $\gate_v(x)=x$.  Otherwise, $u=e^{-}$ for some edge
$e$ of $T$ so that $d_T(e^+,v)=d_T(u,v)-1$.  Then set
$\gate_{v}(x)=\gate_{v}(\phi_{e^+}(\phi_{e^-}^{-1}(\gate_{\phi_{e^-}(\cuco
X_e)}(x))))$.  We also have a map $\beta_{V_v}\co\cuco X \to \fontact
V_v$, defined by $\beta_{V_v}(x)= \pi_{V_v}(\gate_v(x))$.  (Here,
$\gate_{\phi_{e^-}(\cuco X_e)}\co \cuco X_{e^-}=\cuco X_u\to\phi_{e^-}(\cuco
X_e)$ is the usual gate map to a hierarchically quasiconvex subspace,
described in Definition~\ref{defn:gate}, and $\phi_{e^\pm}^{-1}$ is a
quasi-inverse for the edge-hieromorphism.)
\end{defn}

\begin{lem}\label{lem:far_gates}
There exists $K$, depending only on $E$ and $\xi$, such that the
following holds.  Let $e,f$ be edges of $T$ and $v$ a vertex so that
$e^-=f^-=v$.  Suppose for some $V\in\mathfrak S_v$ that there exist
$x,y\in\phi_{e^-}(\cuco X_e)\subseteq\cuco X_v$ with
$\dist_{V}(\gate_{\phi_{f^-}(\cuco X_{f})}(x),\gate_{\phi_{f^-}(\cuco
X_{f})}(y))>10K$.  Then $V\in\phi_{e^-}\inducedS (\mathfrak S_e)\cap
\phi_{f^{-}}\inducedS (\mathfrak S_f)$. 
\end{lem}

\begin{proof}
Let $\cuco Y_e=\phi_{e^-}(\cuco X_e)$ and let $\cuco
Y_f=\phi_{f^-}(\cuco X_f)$; these spaces are uniformly hierarchically
quasiconvex in $\cuco X_v$.  Moreover, by fullness of the
edge-hieromorphisms, we can choose $K\geq100E$ so that the map
$\pi_V\co\cuco Y_e\to\fontact V$ is $K$--coarsely surjective for each
$V\in\phi_{e^-}\inducedS (\mathfrak S_e)$, and likewise for
$\phi_{f^-}\inducedS (\mathfrak S_f)$ and $\cuco Y_v$.  If $V\in\mathfrak
S_v-\phi_{f^-}\inducedS (\mathfrak S_f)$, then $\pi_V$ is $K$--coarsely
constant on $\cuco Y_f$, by the distance formula, since $\cuco X_f$ is
quasi-isometrically embedded.  Likewise, $\pi_V$ is coarsely constant
on $\cuco Y_e$ if $V\not\in\phi_{e^-}\inducedS (\mathfrak S_e)$.  (This
also follows from consistency when $V$ is transverse to some unbounded
element of $\phi_{e^-}\inducedS (\mathfrak S_e)$ and from consistency and
bounded geodesic image otherwise.)

Suppose that there exists $V\in\mathfrak S_v$ such that
$\dist_{V}(\gate_{\phi_{f^-}(\cuco X_{f})}(x),\gate_{\phi_{f^-}(\cuco
X_{f})}(y))>10K$.  Since $\gate_{\phi_{f^-}(\cuco
X_{f})}(x),\gate_{\phi_{f^-}(\cuco X_{f})}(y)\in\cuco X_f$, we
therefore have that $V\in\phi_{f^-}\inducedS (\mathfrak S_f)$.  On the
other hand, the definition of gates implies that $\dist_{V}(x,y)>8K$,
so $V\in\phi_{e^-}\inducedS (\mathfrak S_e).$
\end{proof}

\begin{lem}\label{lem:equiv}
There exists a constant $K'$ such that the following holds.  Let $e,f$
be edges of $T$ and suppose that there do not exist $V_e\in\mathfrak
S_e,V_f\in\mathfrak S_f$  for which 
$\phi_{e^-}\inducedS (V_e)\sim\phi_{f^-}\inducedS (V_f)$.  Then
$\gate_{e^-}(\cuco X_f)$ has diameter at most $K'$.
In particular, the conclusion holds if $\dist_T(e,f)>n$, where $n$ 
bounds the diameter of the supports.
\end{lem}

\begin{proof}
The second assertion follows immediately from the first in light of how $n$ was chosen.  

We now prove the first assertion by induction on the number $k$ of vertices on the geodesic in $T$ from $e$ to $f$.  
The base case, $k=1$, follows from Lemma~\ref{lem:far_gates}.  

For $k\geq 1$, let $v_0,v_1,\ldots,v_k$ be the vertices on a geodesic
from $e$ to $f$, in the obvious order.  Let $b$ be the edge joining
$v_{k-1}$ to $v_k$, with $b^-=v_k$.

\begin{figure}[h]
\begin{overpic}[width=0.4\textwidth]{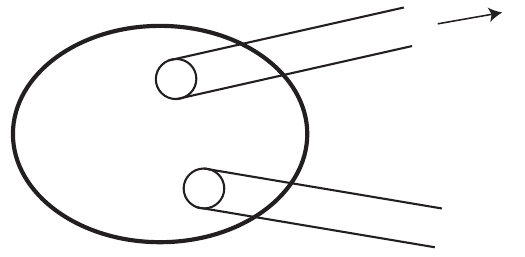}
\put(100,45){$\cuco X_{e}$}
\put(65,6){$\cuco X_{f}$}
\put(15,20){$\cuco X_{b_{-}}$}
\put(65,40){$\cuco X_{b}$}
\end{overpic}
\caption{Schematic of the subset of $\cuco X$ near $\cuco X_{b^{-}}$ .}\label{fig:gates}
\end{figure}

It follows from the definition of gates that $\gate_{e^-}(\cuco X_f)$
has diameter (coarsely) bounded above by that of
$\gate_{\phi_{b^-}(\cuco X_{b})}(\cuco X_f)$ and that of
$\gate_{e^-}(\cuco X_{b})$.  Hence suppose that
$\diam(\gate_{\phi_{b^-}(\cuco X_{b})}(\cuco X_f))>10K$ and
$\diam(\gate_{e^-}(\cuco X_{b}))>10K$.  Then, by induction and
Lemma~\ref{lem:far_gates}, we see that there exists $V_e\in\mathfrak
S_e,V_f\in\mathfrak S_f$ for which $\phi_{e^-}\inducedS
(V_e)\sim\phi_{f^-}\inducedS (V_f)$, a contradiction.
\end{proof}

\begin{lem}\label{lem:vertex_gate_lipschitz}
The map $\gate_v\colon\cuco X\to\cuco X_v$ is coarsely Lipschitz, with constants independent of $v$.
\end{lem}

\begin{proof}
Let $x,y\in\cuco X$.  If the projections of $x,y$ to $T$ lie in the ball of radius $2n+1$ about $v$, then this 
follows since $\gate_v$ is the composition of a bounded number of maps, each of which is uniformly coarsely 
Lipschitz by Lemma~\ref{lem:gate}.  Otherwise, by Remark~\ref{rem:tree_quasigeodesic}, it suffices to consider $x,y$ 
with $\dist_{\cuco X}(x,y)\leq C$, where $C$ depends only on the metric $\dist$.  In this case, let $v_x,v_y$ be the 
vertices in $T$ to which $x,y$ project.  Let $v'$ be the median in $T$ of $v,v_x,v_y$.  Observe that there is a 
uniform bound on $\dist_T(v_x,v')$ and $\dist_T(v_y,v')$, so it suffices to bound 
$\dist_v(\gate_v(\gate_{v'}(x)),\gate_v(\gate_{v'}(y)))$.  Either $\dist_T(v,v')\leq 2n+1$, and we are done, or 
Lemma~\ref{lem:equiv} gives the desired bound, since equivalence classes have support of diameter at most $n$.
\end{proof}

\begin{defn}[Projections]\label{defn:comb_proj}
For each $[W]\in\mathfrak S$, define the \emph{projection} $\pi_{[W]}\colon\cuco X\to\fontact [W]$ by 
$\pi_{[W]}(x)=\beta_{W_v}(x)$, where $W_v$ is the favorite representative of $[W]$.  Note that these projections 
take 
points to uniformly bounded sets, since the collection of vertex spaces is uniformly hierarchically hyperbolic.  
Define $\pi_T\colon\cuco X\to T$ to be the usual projection to $T$.  Finally, for each 
$\container\in\bigcup_\eta\mathcal K_\eta$, just let $\pi_\container:\cuco X\to\fontact\container$ be a constant 
map.
\end{defn}

\begin{lem}[Comparison maps]\label{lem:comparison_map}
There exists a uniform constant $\xi\geq1$ such that for all $W_v\in\mathfrak S_v,W_w\in\mathfrak S_w$ with $W_v\sim 
W_w$, there exists a $(\xi,\xi)$--quasi-isometry $\mathfrak c\co W_v\to W_w$ such that $\mathfrak 
c\circ\beta_v=\beta_w$ up to uniformly bounded error.
\end{lem}

\begin{defn}\label{defn:comparison_map}A map $\mathfrak c$ as given 
    by Lemma~\ref{lem:comparison_map} is called a \emph{comparison map}.\end{defn}

\begin{proof}[Proof of Lemma~\ref{lem:comparison_map}]
We first clarify the situation by stating some consequences of the definitions.  Let $e^+,e^-$ be vertices of $T$ 
joined by an edge $e$.  Suppose that there exists $W^+\in\mathfrak S_{e^+},W^-\in\mathfrak S_{e^-}$ such that 
$W^+\sim W^-$, so that there exists $W\in\mathfrak S_e$ with $\pi(\phi_{e^\pm})(W)=W^\pm$.  Then the following 
diagram coarsely commutes (with uniform constants):
\begin{center}
$
\begin{diagram}
\node{}\node{\cuco X}\arrow{sw,l}{\gate_{e^-}}\arrow{se,l}{\gate_{e^+}}\node{}\\
\node{\cuco X_{e^-}}\arrow{se,t}{}\arrow{s,l}{\pi_{W^-}}\node{}\node{\cuco 
X_{e^+}}\arrow{sw,t}{}\arrow{s,r}{\pi_{W^-}}\\
\node{\fontact W^-}\arrow{se}\node{\cuco X_e}\arrow{nw,t}{}\arrow{ne,t}{}\arrow{s,r}{\pi_W}\node{\fontact 
W^+}\arrow{sw}\\
\node{}\node{\fontact W}\arrow{nw}{}\arrow{ne}{}
\end{diagram}
$
\end{center}
where $\cuco X_e\to\cuco X_{e^\pm}$ is the uniform quasi-isometry
$\phi_{e^\pm}$, while $\cuco X_{e^{\pm}}\to\cuco X_e$ is the
composition of a quasi-inverse for $\phi_{e^\pm}$ with the gate map
$\cuco X_{e^\pm}\to\phi_{e^\pm}(\cuco X_e)$, and the maps $\fontact
W\leftrightarrow\fontact W^{\pm}$ are the quasi-isometries implicit in
the edge hieromorphism or their quasi-inverses.  The proof essentially
amounts to chaining together a sequence of these diagrams as $e$
varies among the edges of a geodesic from $v$ to $w$; an important 
ingredient is 
played by the fact that such a geodesic has length at most
$n$.

Let $v=v_0,v_1,\ldots,v_m,v_{m+1}=w$ be the geodesic sequence in $T$
from $v$ to $w$ and let $e^i$ be the edge joining $v_i$ to $v_{i+1}$.
For each $i$, choose $W_i\in\mathfrak S_{e^i}$ and
$W_i^\pm\in\mathfrak S_{e^i_\pm}$ so that (say) $W_0^-=W_v$ and
$W_{m}^+=W_{w}$ and so that $\phi_{e^i_{\pm}}\inducedS (W_i)=W_i^{\pm}$ for
all $i$.  For each $i$, let $q_i^{\pm}\co\fontact W_i\to\fontact
W_i^{\pm}$ be $q_i=\phi_{e^i_{\pm}}\induced(W_i)$ be the
$(\xi',\xi')$--quasi-isometry packaged in the edge--hieromorphism, and
let $\bar q_i^{\pm}$ be a quasi-inverse; the constant $\xi'$ is
uniform by hypothesis, and $m\leq n$ since $\mathcal T$ has bounded
supports.  The hypotheses on the edge--hieromorphisms ensure that the
$W_i^{\pm}$ are uniquely determined by $W_v,W_{w}$, and we define
$\mathfrak c$ by $$\mathfrak c=q_m^{\epsilon_m}\bar
q_m^{\epsilon_m'}\cdots q_1^{\epsilon_1}\bar q_1^{\epsilon_1'},$$
where $\epsilon_i,\epsilon_i'\in\{\pm\}$ depend on the orientation of
$e^i$, and $\epsilon_i'=+$ if and only if $\epsilon_i=-$.  This is a
$(\xi,\xi)$--quasi-isometry, where $\xi=\xi(\xi'_n)$.

If $v=w$, then $\mathfrak c$ is the identity and $\mathfrak c\circ\beta_v=\beta_v$.  Let $d\geq 1=\dist_T(v,w)$ and 
let $w'$ be the penultimate vertex on the geodesic of $T$ from $v$ to $w$.  Let $\mathfrak c':\fontact 
W_v\to\fontact W_{w'}$ be a comparison map, so that, by induction, there exists $\lambda'\geq0$ so that 
$\dist_{\fontact W_{w'}}(\mathfrak c'\circ\beta_v(x),\beta_{w'}(x))\leq\lambda'$ for all $x\in\cuco X$.  Let 
$\mathfrak c''=\bar q_k^+q_k^-\co\fontact W_{w'}\to\fontact W_w$ be the $(\xi',\xi')$--quasi-isometry packaged in 
the edge--hieromorphism, so that the following diagram coarsely commutes:
\begin{center}
$
\begin{diagram}
\node{}\node{\cuco X}\arrow{sw,l}{\gate_v}\arrow{s,r}{\gate_{w'}}\arrow{se,l}{\gate_w}\\
\node{\cuco X_v}\arrow{r,b}{\gate_{w'}}\arrow{s,l}{\pi_{W_v}}\node{\cuco 
X_{w'}}\arrow{r,b}{\gate_w}\arrow{s,l}{\pi_{W_{w'}}}\node{\cuco X_w}\arrow{s,l}{\pi_{W_w}}\\
\node{\fontact W_v}\arrow{e,b}{\mathfrak c'}\node{\fontact W_{w'}}\arrow{e,b}{\mathfrak c''}\node{\fontact W_w}
\end{diagram}
$
\end{center}
Since $\mathfrak c=\mathfrak c''\circ\mathfrak c'$ and the constants implicit in the coarse commutativity of the 
diagram depend only on the constants of the hieromorphism and on $d\leq n$, the claim follows.
\end{proof}

\begin{lem}\label{lem:projection_coarse_lipschitz}
There exists $K$ such that each $\pi_{[W]}$ is $(K,K)$--coarsely Lipschitz.
\end{lem}

\begin{proof}
For each vertex $v$ of $T$ and each $V\in\mathfrak S_v$, the projection $\pi_V:\cuco X_v\to\fontact V$ is uniformly 
coarsely Lipschitz, by definition.  By Lemma~\ref{lem:vertex_gate_lipschitz}, each gate map $\gate_v:\cuco X\to\cuco 
X_v$ is uniformly coarsely Lipschitz.  The lemma follows since $\pi_{[W]}=\pi_{W_v}\circ\gate_v$, where $v$ is the 
favorite vertex carrying the favorite representative $W_v$ of $[W]$.   
\end{proof}

\begin{defn}[Projections between hyperbolic spaces]\label{defn:comb_rho}
If $[V]\nest[W]$, then choose vertices $v,v',w\in\mathcal V$ so that $V_v,W_w$ are respectively the favorite 
representatives of $[V],[W]$, while $V_{v'},W_{v'}$ are respectively representatives of $[V],[W]$ with 
$V_{v'},W_{v'}\in\mathfrak S_{v'}$ and $V_{v'}\nest W_{v'}$.  Let $\mathfrak c_V\co\fontact V_{v'}\to\fontact V_v$ 
and $\mathfrak c_W\co\fontact W_{v'}\to\fontact W_w$ be comparison maps.  Then define 
$$\rho^{[V]}_{[W]}=c_W\left(\rho^{V_{v'}}_{W_{v'}}\right),$$ which is a uniformly bounded set, and define 
$\rho^{[W]}_{[V]}\co\fontact [W]\to\fontact [V]$ by $$\rho^{[W]}_{[V]}=\mathfrak 
c_V\circ\rho^{W_{v'}}_{V_{v'}}\circ\bar{\mathfrak c}_W,$$ where $\bar{\mathfrak c}_W$ is a quasi--inverse of 
$\mathfrak c_W$ and $\rho^{W_{v'}}_{V_{v'}}\co\fontact W_{v'}\to\fontact V_{v'}$ is the map provided by 
Definition~\ref{defn:space_with_distance_formula}.\eqref{item:dfs_nesting}.  Similarly, if $[V]\transverse [W]$, and 
there exists $w\in T$ so that $\mathfrak S_w$ contains representatives $V_w,W_w$ of $[V],[W]$, then let 
$$\rho^{[V]}_{[W]}=\mathfrak c_W\left(\rho^{V_{w}}_{W_{w}}\right).$$  Otherwise, choose a closest pair $v,w$ so that 
$\mathfrak S_v$ (respectively $\mathfrak S_w$) contains a representative of $[V]$ (respectively $[W]$).  Let $e$ be 
the first edge of the geodesic in $T$ joining $v$ to $w$, so $v=e^-$ (say).  Let $S$ be the $\nest$--maximal element 
of $\mathfrak S_e$, and let $$\rho^{[W]}_{[V]}=\mathfrak c_V\left(\rho^{\phi_{e^-}\inducedS(S)}_{V_{v}}\right).$$  This 
is well-defined by hypothesis~\eqref{item:hypothesis_4}.

For each $\sim$--class $[W]$, let $\rho^{[W]}_T$ be the support of $[W]$ (a uniformly bounded set since $\mathcal T$ 
has bounded supports).  Define $\rho^T_{[W]}\colon T\to\fontact W$ as follows: given $v\in T$ not in the support of 
$[W]$, let $e$ be the unique edge with $e^-$ (say) separating $v$ from the support of $[W]$.  Let $S\in\mathfrak 
S_e$ 
be $\nest$--maximal.  Then $\rho^T_{[W]}(v)=\rho^{[\phi_{e^-}\inducedS(S)]}_{[W]}$.  If $v$ is in the support of $[W]$, then 
let $\rho^T_{[W]}(v)$ be chosen arbitrarily.

Finally, let $\container,{\container}'\in\bigcup_\eta\mathcal K_\eta$ and let $[W]$ be a $\sim$--class.  If 
$\container\transverse{\container}'$, then $\rho^{{\container}'}_{\container}$ is the single point 
$\fontact\container$.  If $\container\nest{\container}'$, then $\rho^{{\container}'}_{\container}$ is a constant map 
and $\rho_{{\container}'}^{\container}$ is the obvious single point.  We never have $\container\nest[W]$.  If 
$[W]\nest\container$, then $\rho^{[W]}_{\container}$ is again the obvious single point, and we can define 
$\rho_{[W]}^{\container}\co\fontact\container\to\fontact[W]$ to be an arbitrary constant map.  Finally, 
$\rho^T_{\container}\co T\to\fontact\container$ is the constant map, and $\rho^{\container}_T$ is the bounded set 
defined as follows.  By definition, there is a unique pairwise orthogonal set $[W_1],\ldots,[W_\eta]$ so that 
$\container=\container_\eta([W_1],\ldots,[W_\eta])$.  By the proof of Lemma~\ref{lem:comb_complexity}, the supports 
of the 
various $[W_i]$ all intersect in a subtree of $T$, which necessarily has diameter at most $n$, by the bounded 
supports hypothesis; we take this intersection to be $\rho^{\container}_T$.
\end{defn}

We are now ready to complete the proof of the combination theorem.

\begin{proof}[Proof of Theorem~\ref{thm:combination}]
We claim that $(\cuco X(\mathcal T),\mathfrak S)$ is hierarchically 
hyperbolic.  We take the nesting, orthogonality, and
transversality relations for a tree of spaces given by
Definition~\ref{defn:nest_etc_on_classes} and Definition~\ref{defn:containers_S}. In 
Lemmas~\ref{lem:comb_nest}, ~\ref{lem:orth_comp}, and Remark~\ref{rem:extending_relations}, it is shown that 
these relations satisfy all of
the conditions \eqref{item:dfs_nesting} and~\eqref{item:dfs_orthogonal} of
Definition~\ref{defn:space_with_distance_formula}
not involving the projections.  Moreover, the complexity of $(\cuco
X(\mathcal T),\mathcal S)$ is finite, by
Lemma~\ref{lem:comb_complexity}, verifying
Definition~\ref{defn:space_with_distance_formula}.\eqref{item:dfs_complexity}.
The set of $\delta$--hyperbolic spaces $\{\fontact A:A\in\mathfrak
S\}$ is provided by Definition~\ref{defn:favorites}, while the
projections $\pi_{[W]}\co\cuco X\to\fontact [W]$ required by
Definition~\ref{defn:space_with_distance_formula}.\eqref{item:dfs_curve_complexes}
are defined in Definition~\ref{defn:comb_proj} and are uniformly coarsely Lipschitz by 
Lemma~\ref{lem:projection_coarse_lipschitz}.  Since $\pi_{[W]}(\cuco X)$ uniformly coarsely coincides with the image of an 
appropriately-chosen vertex space $\cuco X_v$, it is quasiconvex since $\pi_{W_v}(\cuco X_v)$ is uniformly quasiconvex by 
Definition~\ref{defn:space_with_distance_formula}.\eqref{item:dfs_curve_complexes}.
The projections
$\rho^{[V]}_{[W]}$ when $[V],[W]$ are non-orthogonal are described in
Definition~\ref{defn:comb_rho}.  To complete the proof, it thus
suffices to verify the consistency inequalities
(Definition~\ref{defn:space_with_distance_formula}.\eqref{item:dfs_transversal}),
the bounded geodesic image axiom
(Definition~\ref{defn:space_with_distance_formula}.\eqref{item:dfs:bounded_geodesic_image}),
the large link axiom
(Definition~\ref{defn:space_with_distance_formula}.\eqref{item:dfs_large_link_lemma}),
partial realization
(Definition~\ref{defn:space_with_distance_formula}.\eqref{item:dfs_partial_realization}),
and uniqueness
(Definition~\ref{defn:space_with_distance_formula}.\eqref{item:dfs_uniqueness}).

\textbf{Consistency:}  Any consistency inequalities involving elements of $\bigcup_\eta\mathcal K_\eta$ hold 
trivially since in that case, at least one of the two associated hyperbolic spaces in question is a point.  Suppose 
that $[U]\transverse [V]$ or $[U]\nest[V]$ and let $x\in\cuco X$.  Choose representatives $U_u\in\mathfrak 
S_u,V_u\in\mathfrak S_v$ of $[U],[V]$ so that $\dist_T(u,v)$ realizes the distance between the supports of 
$[U],[V]$. 
 By composing the relevant maps in the remainder of the argument with comparison maps, we can assume that $U_u,V_v$ 
are favorite representatives.  Without loss of generality, there exists a vertex $w\in T$ so that $x\in\cuco X_w$.  
If $u=v$, then consistency follows since it holds in each vertex space, so assume that $u,v$ have disjoint supports 
and in particular $[U]\transverse[V]$.

If $w\not\in[u,v]$, then (say) $u$ separates $w$ from $v$.  Then $\pi_{[V]}(x)=\pi_{V_v}(\gate_{v}(x))$.  Let $e$ be 
the edge of the geodesic $[u,v]$ emanating from $u$, so that $\rho^{[V]}_{[U]}=\rho^{S}_{U_u}$, where $S$ is the 
image in $\mathfrak S_u$ of the $\nest$--maximal element of $\mathfrak S_e$.  If $\dist_{\fontact 
U_u}(\gate_u(x),\rho^S_{U_u})\leq E$, then we are done.  Otherwise, by consistency in $\mathfrak S_u$, we have 
$\dist_{\fontact S}(\gate_u(x),\rho^{U_u}_S)\leq E$, from which consistency follows.  Hence suppose that 
$w\in[u,v]$.  Then without loss of generality, there is an edge $e$ containing $v$ and separating $w$ from $v$.  As 
before, projections to $V$ factor through the $\nest$-maximal element of $\mathfrak S_e$, from which consistency 
follows.  

We verify consistency for $T,[W]$ for each $\sim$--class $[W]$.  Choose $x\in\cuco X_v$.  If $\dist_T(v,T_{[W]})\geq 
n+1$, then let $e$ be the edge incident to $T_{[W]}$ separating it from $v$, so that (up to a comparison map) 
$\rho^T_{[W]}(v)=\rho^{S}_{W}$, where $W\in\mathfrak S_{e^+}$ represents $W$, and $e^+\in T_{[W]}$, and $S$ is the 
image in $\mathfrak S_{e^+}$ of the $\nest$--maximal element of $\mathfrak S_e$.  (Note that $W\transverse S$ by 
hypothesis~\eqref{item:hypothesis_4} of the theorem and the choice of $e$).  On the other hand (up to a comparison 
map) $\pi_{[W]}(x)=\pi_W(\gate_{e^+}(x))\asymp\pi_W(\gate_{e^+}(\cuco X_v))\asymp\pi_W(\mathbf F_S)=\rho^S_W$, as 
desired.  (The final coarse equality holds by hypothesis~\eqref{item:group:hypothesis_4}.)

Finally, suppose that $[U]\nest[V]$ and that either $[V]\propnest[W]$ or $[V]\transverse [W]$ and $[U]\not\orth[W]$. 
 We claim that $\dist_{[W]}(\rho^{[U]}_{[W]},\rho^{[V]}_{[W]})$ is uniformly bounded.  By definition, $T_{[U]}\cap 
T_{[V]}\neq\emptyset$ and we fix representatives $U_u\in\mathfrak S_u,V_u\in\mathfrak S_u$ of $[U],[V]$ with 
$U_u\nest V_u$.

Next, suppose that $T_{[V]}\cap T_{[W]}\neq\emptyset$ and $T_{[U]}\cap T_{[W]}\neq\emptyset$, so that we can choose 
vertices $v,w$ of $T$ and representatives $V_w,W_w\in\mathfrak S_w$ so that $V_w\nest W_w$ or 
$V_w\transverse W_w$ according to whether $[V]\nest[W]$ or $[V]\transverse [W]$, and choose representatives 
$U_v,W_v\in\mathfrak S_v$ of $[U],[W]$ so that $U_v\nest W_v$ or $U_v\transverse W_v$ according to whether 
$[U]\nest[W]$ or $[U]\transverse [W]$.  Let $m\in T$ be the median of $u,v,w$.  Since $u,w$ lie in the support of 
$[U],[W]$, so does $m$, since supports are connected.  Likewise, $m$ lies in the support of $[V]$.  Let 
$U_m,V_m,W_m$ be the representatives of $[U],[V],[W]$ in $m$.  Since edge-maps are full hieromorphisms, we have 
$U_m\nest V_m$ and $U_m\not\orth W_m$ and either $V_m\nest W_m$ or $V_m\transverse W_m$.  Hence 
Definition~\ref{defn:space_with_distance_formula}.\eqref{item:dfs_transversal} implies that 
$\dist_{W_m}(\rho^{U_m}_{W_m},\rho^{V_m}_{W_m})$ is uniformly bounded.  Since the comparison maps are uniform 
quasi-isometries, it follows that $\dist_{[W]}(\rho^{[U]}_{[W]},\rho^{[V]}_{[W]})$ is uniformly bounded, as desired.

Next, suppose that $T_{[U]}\cap T_{[W]}=\emptyset$. Then $[U]\transverse[W]$.  If there is an edge $e$ separating 
$T_{[W]}$ from $T_{[U]}\cup T_{[V]}$, then $\rho^{[U]}_{[W]}=\rho^{[V]}_{[W]}$ by definition.  Otherwise, 
$[V]\transverse [W]$ (by transitivity of $\nest$ and the fact that $T_{[U]}\cap T_{[W]}=\emptyset$) but there is 
some 
$v\in T_{[V]}\cap T_{[W]}$ and representatives $V_v,W_v\in\mathfrak S_v$ of $[V],[W]$ with $V_v\transverse W_v$.  
But 
by fullness of the hieromorphism and induction on $\dist_T(u,v)$, we find that $v\in T_{[U]}$, contradicting that 
$T_{[U]}\cap T_{[W]}=\emptyset$.

\textbf{Bounded geodesic image and large link axiom in $T$ and $\bigcup_\eta\mathcal K_\eta$:}  For bounded geodesic 
image, in the case where one of the two nested elements of $\mathfrak S$ in question is in $\bigcup_\eta\mathcal 
K_\eta$, 
the claim holds trivially since the associated hyperbolic space has diameter $0$ and the associated map between 
hyperbolic spaces has either domain or codomain a single point.

Let $\gamma$ be a geodesic in $T$ and let $[W]$ be a $\sim$--class so that $\dist_T(\gamma,\rho^{[W]}_T)>1$, which 
is to say that $\gamma$ does not contain vertices in the support of $[W]$.  Let $e$ be the terminal edge of the 
geodesic joining $\gamma$ to the support of $[W]$.  Then for all $u\in\gamma$, we have by definition that 
$\rho^T_{[W]}(u)=\rho^{[S]}_{[W]}$ for some fixed $\sim$--class $[S]$.  This verifies the bounded geodesic image 
axiom for $T,[W]$.  

By Lemma~\ref{lem:equiv}, there exists a constant $K''$ such that if $x,x'\in\cuco X$ respectively project to 
vertices $v,v'$, then any $[W]\in\mathfrak S-\{T\}$ with $\dist_{[W]}(x,x')\geq K''$ is supported on a vertex 
$v_{[W]}$ on the geodesic $[v,v']$ and is hence nested into $[S_{v_W}]$, where $S_{v_W}$ is maximal in $\mathfrak 
S_{v_{[W]}}$.  Indeed, choose $w$ in the support of $[W]$.  Then either $\dist_{[W]}(x,x')$ is smaller than some 
specified constant, or $\dist_{\cuco X_w}(\gate_w(x),\gate_w(x')>K'$.  Thus $\gate_{\cuco X_w}(\cuco X_m)$ has 
diameter at least $K'$, where $m$ is the median of $v,v',w$.  Hence $m$ lies in the support of $[W]$, and 
$m\in[v,v']$, and $[W]\nest [S]$, where $S$ is $\nest$--maximal in $\mathfrak S_m$.  Finally, for each such 
$S_{v_W}$, it is clear that $\dist_T(x,\rho^{[S_{v_W}]}_T)\leq\dist_T(x,x')$, verifying the conclusion of the large 
link axiom for $T$.

Finally, we have to check that if $\container=\container([W_1],\ldots,[W_\eta])$ and $x,y\in\cuco X$, then there 
exist 
a uniformly bounded number of elements $[U_j]$ so that for any $[V]\nest\container$ with $\dist_{[V]}(x,y)\ge E$, 
the class $[V]$ is nested into some $U_j$.  We have shown above that any such $[V]$ is nested into the 
$\nest$--maximal $S_v\in\mathfrak S_v$ for some $v$ on the geodesic of $T$ between $\pi_T(x)$ and $\pi_T(y)$.  Now, 
since $[V]\nest\container$, we have $[V]\orth[W_i]$ for all $i$, so that the support of $[V]$ uniformly coarsely 
coincides with the support of $[W_i]$ for each $i$.  Hence $v$ must be among the boundedly many vertices on 
$[\pi_T(x),\pi_T(y)]$ that lie in the intersection of the supports of the $[W_i]$.  Thus we can take our set of 
$U_j$ to be the set of such $S_v$, which has uniformly bounded cardinality (bounded in terms of $n$).

\textbf{Bounded geodesic image and large link axiom in $W\sqsubset T$:}  Let $[W]$ be non-$\nest$--maximal, let 
$[V]\nest[W]$, and let $\gamma$ be a geodesic in $\fontact [W]$.  Then $\gamma$ is a geodesic in $\fontact W_w$, by 
definition, where $w$ is the favorite vertex of $[W]$ with corresponding representative $W_w$.  Let $V_w$ be the 
representative of $[V]$ supported on $w$, so that $\rho^{[V]}_{[W]}=\rho^{V_w}_{W_w}$, so that $\gamma$ avoids the 
$E$--neighborhood of $\rho^{[V]}_{[W]}$ exactly when it avoids the $E$--neighborhood of $\rho^{V_w}_{W_w}$.  The 
bounded geodesic image axiom now follows from bounded geodesic image in $\mathfrak S_w$, although the constant $E$ 
has been changed according to the quasi-isometry constant of comparison maps.  

Now suppose $x,x'\in\cuco X_v,\cuco X_{v'}$ and choose $w$ to be the favorite vertex in the support of $[W]$.  
Suppose for some $[V]\nest[W]$ that $\dist_{[V]}(x,x')\geq E'$, where $E'$ depends on $E$ and the quasi-isometry 
constants of the edge-hieromorphisms.  Then $\dist_{V_w}(\gate_w(x),\gate_w(x'))\geq E$, for some representative 
$V_w\in\mathfrak S_w$ of $[V]$, by our choice of $E'$.  Hence, by the large link axiom in $\mathfrak S_w$, we have 
that $V_w\nest T_i$, where $\{T_i\}$ is a specified set of 
$N=\lfloor\dist_{W_w}(\gate_w(x),\gate_w(x'))\rfloor=\lfloor\dist_{[W]}(x,x')\rfloor$ elements of $\mathfrak S_w$, 
with each $T_i\sqsubset W_w$.  Moreover, the large link axiom in $\mathfrak S_w$ implies that 
$\dist_{[W]}(x,\rho^{[T_i]}_{[W]})=\dist_{W_w}(\gate_w(x),\rho^{T_i}_{W_w})\leq N$ for all $i$.  This verifies the 
large link axiom for $(\cuco X(\mathcal T),\mathfrak S)$.

\textbf{Partial realization:}  Let $[V_1],\ldots,[V_k]\in\mathfrak S$ be pairwise--orthogonal, and, for each $i\leq 
k$, let $p_i\in\fontact [V_i]$.  For each $i$, let $T_i\subseteq T$ be the induced subgraph spanned by the vertices 
$w$ such that $[V_i]$ has a representative in $\mathfrak S_w$.  The definition of the $\sim$--relation implies that 
each $T_i$ is connected, so by the Helly property of trees, there exists a vertex $v\in T$ such that for each $i$, 
there exists $V^i_v\in\mathfrak S_v$ representing $[V_i]$.  Moreover, we have $V^i_v\orth V^j_v$ for $i\neq j$, 
since 
the edge--hieromorphisms preserve the orthogonality relation.  Applying the partial realization axiom 
(Definition~\ref{defn:space_with_distance_formula}.\eqref{item:dfs_partial_realization}) to $\{p_i'\in\fontact 
V^i_v\}$, where $p_i'$ is the image of $p_i$ under the appropriate comparison map, yields a point $x\in\cuco X_v$ 
such that $\pi_{V^i_v}(x)$ is coarsely equal to $p_i'$ for all $i$, whence $\dist_{[V_i]}(x,p_i)$ is uniformly 
bounded.  If $[V_i]\nest[W]$, then $W$ has a representative $W_v\in\mathfrak S_v$ such that $V^i_v\nest W$, whence 
$\dist_{[W]}(x,\rho^{[V_i]}_W)$ is uniformly bounded since $x$ is a partial realization point for $\{V^i_v\}$ in 
$\mathfrak S_v$.  Finally, if $[W]\transverse[V_i]$, then either the subtrees of $T$ supporting $[W]$ and $[V_i]$ 
are 
disjoint, in which case $\dist_{[W]}(x,\rho^{[V_i]}_{[W]})$ is bounded, or $[W]$ has a representative in $\mathfrak 
S_v$ transverse to $V^i_v$, in which case the same inequality holds by our choice of $x$.  There is nothing to check 
regarding projections onto $\fontact\container$ for $\container\in\bigcup_\eta\mathcal K_\eta$, since those spaces 
are 
single points.

It remains to consider pairwise orthogonal collections that include elements of $\bigcup_\eta\mathcal K_\eta$.  Since 
no 
two of these elements can be orthogonal, we must consider $\container=\container_\eta([W_1],\ldots,[W_\eta])$, which 
is 
orthogonal to $\sim$--classes $[V_1],\ldots,[V_k]$, which themselves form an orthogonal collection.  Let $p$ be the 
unique point of $\fontact\container$, and for each $i\leq k$, let $p_i\in\fontact V_i$.  Then the previous 
discussion provides a point $x$ so that for any $i$, we have $\pi_{[V_i]})(x)\asymp p_i$, and 
$\pi_T(x)\asymp\rho^{[V_i]}_T$.  Moreover, for any $[U]$ so that, for some $i$, we have $[U]\transverse[V_i]$ or 
$[V_i]\nest[U]$, we have $\pi_{[U]}(x)\asymp\rho^{[V_i]}_{[U]}$.  We claim that $x$ also satisfies the conclusion of 
partial realisation for the pairwise-orthogonal set $\container,[V_1],\ldots,[V_k]$.  Again, there is nothing to 
check regarding projections onto $\fontact\container$ for $\container\in\bigcup_\eta\mathcal K_\eta$, since those 
spaces 
are single points, and this includes the statement about $\pi_{\container}(x)$.  So, it just remains to check that 
$\pi_T(x)$ uniformly coarsely coincides with $\rho^{\container}_T$.  But $\pi_T(x)$ coarsely coincides with 
$\rho_T^{[V_i]}$ for any $i$, by the construction of $x$.  Since $[V_i]\orth\container$, we have $[V_i]\nest[W_j]$ 
for some $j$, so $\rho^{[V_i]}_T$ coarsely coincides with $\rho^{[W_j]}_T$.  But $\rho^{[W_j]}_T$ coarsely 
coincides, by definition, with $\rho^{\container}_T$, as required.

\textbf{Uniqueness of realization points:} Suppose $x,y\in\cuco X$
satisfy $\dist_{[V]}(x,y)\leq K$ for all $[V]\in\mathfrak S$.  Then,  
for each vertex $v\in T$,  
applying the uniqueness axiom in $\cuco X_v$
to $\gate_v(x),\gate_v(y)$ shows that 
$\dist_{\cuco
X_v}(\gate_v(x),\gate_v(y))\leq\zeta=\zeta(K)$. 
Indeed, otherwise we
would have $\dist_{V}(\gate_v(x),\gate_v(y))>\xi K+\xi$ for some
$V\in\mathfrak S_v$, whence $\dist_{[V]}(x,y)>K$.  There exists $k\leq
K$ and a sequence $v_0,\ldots,v_k$ of vertices in $T$ so that
$x\in\cuco X_{v_0},y\in\cuco X_{v_k}$.  For each $j$, let
$x_j=\gate_{v_j}(x)$ and $y_j=\gate_{v_j}(y)$.  Then
$x=x_0,y_0,x_1,y_1,\ldots,y_{j-1},x_j,y_j,\ldots,x_k,y_k=y$ is a path
of uniformly bounded length joining $x$ to $y$.  Indeed, $\dist_{\cuco
X_{v_j}}(x_j,y_j)\leq\zeta$ and $k\leq K$ by the preceding discussion,
while $x_j$ coarsely coincides with a point on the opposite side of an
edge-space from $y_{j-1}$ by the definition of the gate of an
edge-space in a vertex-space and the fact that $x_{j-1}$ and $y_{j-1}$
coarsely coincide.  This completes the proof.
\end{proof}

\subsection{Equivariant definition of $(\cuco X(\mathcal T),\mathfrak
S)$}\label{subsec:tree_of_HHS_automorphisms} 
Let $\mathcal T$ denote the tree of hierarchically hyperbolic spaces  
$(\mathcal T,\{\cuco X_v\},\{\cuco X_e\},\{\pi_{e^\pm}\})$, 
and let $(\cuco X(\mathcal T),\mathfrak S)$ be the hierarchically hyperbolic structure defined in the proof of 
Theorem~\ref{thm:combination}.  Various arbitrary choices were made in defining the constituent hyperbolic spaces 
and projections in this hierarchically hyperbolic structure, and we now insist on a specific way of making these 
choices in order to describe automorphisms of $(\cuco X(\mathcal T),\mathfrak S)$.

Recall that an automorphism of $(\cuco X(\mathcal T),\mathfrak S)$ is determined by a bijection $g\co\mathfrak 
S\to\mathfrak S$ and a set of isometries $g\co\fontact Q\to\fontact gQ$, for $Q\in\mathfrak S$.  Via the 
distance formula, this determines a uniform quasi-isometry $\cuco X(\mathcal T)\to\cuco X(\mathcal T)$.

A bijection $g\co\bigsqcup_{v\in\mathcal V}\mathfrak S_{v}\to\bigsqcup_{v\in\mathcal V}\mathfrak S_v$ is 
\emph{$T$--coherent}
if there is an induced isometry $g$ of the underlying tree, $T$, so that $fg=gf$, where
$f\colon\bigsqcup_{v\in\mathcal V}\mathfrak S_{v}\to T$
sends each $V\in\mathfrak S_v$ to $v$, for all $v\in\mathcal V$.  The
$T$--coherent bijection $g$ is said to be  \emph{$\mathcal T$--coherent} 
if it also 
preserves the relation~$\sim$.  Noting that the composition of
$\mathcal T$--coherent bijections is $\mathcal T$--coherent, denote by
$\mathcal P_{\mathcal T}$ the group of $\mathcal T$--coherent
bijections.  For each $g\in\mathcal P_{\mathcal T}$, there is an
induced bijection $g\co\mathfrak S^0\to\mathfrak S^0$.

Recall that the hierarchically hyperbolic structure $(\cuco X(\mathcal
T),\mathfrak S)$ was completely determined except for the following
three types of choices which were made arbitrarily.
\begin{enumerate}
 \item For each $[V]\in\mathfrak S$, we chose an arbitrary
 \emph{favorite vertex} $v$ in the support of $[V]$; and, 
 \item we chose an arbitrary \emph{favorite representative} $V_v\in\mathfrak
 S_v$ with $[V]=[V_v]$.  (Note that if, as is often the case in
 practice, edge-hieromorphisms $\mathfrak S_e\to\mathfrak S_v$ are
 injective, then $V_v$ is the unique representative of its
 $\sim$--class that lies in $\mathfrak S_v$, and hence our choice is 
 uniquely determined.)
 \item For each $[W]\in\mathfrak S$, the
 point $\rho^T_{[W]}(v)$ is chosen arbitrarily in $\fontact W$, where $W$
 is the favorite representative of $[W]$ and $v$ is a vertex in the support of $[W]$.
\end{enumerate}

We now constrain these choices so that they are equivariant.  For each $\mathcal P_{\mathcal T}$--orbit in 
$\mathfrak S$, choose a representative $[V]$ of that orbit, choose a favorite vertex $v$ arbitrarily in its support, 
and choose a favorite representative $V_v\in\mathfrak S_v$ of $[V]$.  Then declare $gV_v\in\mathfrak S_{gv}$ to be 
the favorite representative, and $gv$ the favorite vertex, associated to $g[V]$, for each $g\in\mathcal P_{\mathcal 
T}$.  

Recall that, for each $[W]\in\mathfrak S$, we defined $\fontact [W]$ to be $\fontact W$, where $W$ is the favorite 
representative of $[W]$.  Suppose that we have specified a subgroup $G\leq\mathcal P_{\mathcal T}$ and, for each 
$[W]\in\mathfrak S$ and $g\in\mathcal P_{\mathcal T}$, an isometry $g\co\fontact [W]\to\fontact g[W]$.  Then we 
choose $\rho^T_{[W]}$ in such a way that $\rho^T_{g[W]}=g\rho^T_{[W]}$ for each $[W]\in\mathfrak S$ and $g\in G$. 

\subsection{Graphs of hierarchically hyperbolic groups}\label{subsec:combination_theorem_for_groups}
Recall that the finitely generated group $G$ is \emph{hierarchically hyperbolic} if there is a hierarchically 
hyperbolic space $(\cuco X,\mathfrak S)$ such that $G\leq\Aut(\mathfrak S)$ and the action of $G$ on $\cuco X$ is 
metrically proper and cobounded and the action of $G$ on $\mathfrak S$ is co-finite (this, together with the 
definition of an automorphism, implies that only finitely many isometry types of hyperbolic space are involved in 
the HHS structure).  Endowing $G$ with a word-metric, we see that $(G,\mathfrak S)$ is a hierarchically hyperbolic 
space.

If $(G,\mathfrak S)$ and $(G',\mathfrak S')$ are hierarchically hyperbolic groups, then a \emph{homomorphism of 
hierarchically hyperbolic groups} $\phi\co(G,\mathfrak S)\to(G',\mathfrak S')$ is a homomorphism $\phi\co G\to G'$ 
that is also a $\phi$--equivariant hieromorphism as in Definition~\ref{defn:homomorphism_of_HHGs}.

Recall that a \emph{graph of groups} $\mathcal G$ is a graph $\Gamma=(V,E)$ together with a set $\{G_v:v\in V\}$ of 
vertex groups, a set $\{G_e:e\in E\}$ of edge groups, and monomorphisms $\phi_{e}^\pm\co G_e\to G_{e^\pm}$, where 
$e^\pm$ are the vertices incident to $e$.  As usual, the \emph{total group} $G$ of $\mathcal G$ is the quotient of 
$\left(*_{v\in V}G_v\right)*F_E$, where $F_E$ is the free group generated by $E$, obtained by imposing the following 
relations:
\begin{itemize}
 \item $e=1$ for all $e\in E$ belonging to some fixed spanning tree $T$ of $\Gamma$;
 \item $\phi_{e}^+(g)=e\phi_{e}^-(g)e^{-1}$ for $e\in E$ and $g\in G_e$.
\end{itemize}

We are interested in the case where $\Gamma$ is a finite graph and,
for each $v\in V,e\in E$, we have sets $\mathfrak S_v,\mathfrak S_e$
so that $(G_v,\mathfrak S_v)$ and $(G_e,\mathfrak S_e)$ are
hierarchically hyperbolic group structures for which  $\phi_{e}^\pm\colon
G_e\to G_{e^\pm}$ is a homomorphism of hierarchically hyperbolic
groups.  In this case, $\mathcal G$ is a \emph{finite graph of
hierarchically hyperbolic groups}.  If in addition each $\phi_e^\pm$
has hierarchically quasiconvex image, then $\mathcal G$ has
\emph{quasiconvex edge groups}.

Letting $\widetilde\Gamma$ denote the Bass-Serre tree, observe that $\mathcal T=\widetilde{\mathcal 
G}=(\widetilde\Gamma,\{G_{\tilde v}\},\{G_{\tilde e}\},\{\phi_{\tilde 
e}^\pm\})$ is a tree of hierarchically 
hyperbolic spaces, where $\tilde v$ ranges over the vertex set of $\widetilde\Gamma$, and each $G_{\tilde v}$ is a 
conjugate in the total group $G$ to $G_v$, where $\tilde v\mapsto v$ under $\widetilde\Gamma\to\Gamma$, and an 
analogous statement holds for edge-groups.  Each $\phi_{\tilde e}^\pm$ is conjugate to an edge-map in $\mathcal G$ 
in the obvious way.  We say $\mathcal G$ has \emph{bounded supports} if $\mathcal T$ does.

\begin{cor}[Combination theorem for HHGs]\label{cor:combination_theorem_for_HHG}
Let $\mathcal G=(\Gamma,\{G_v\},\{G_e\},\{\phi_{e}^\pm\})$ be a finite graph of hierarchically hyperbolic groups, with 
Bass-Serre tree $\widetilde\Gamma$.  
Suppose that:
\begin{enumerate}
  \item $\mathcal G$ has quasiconvex edge groups;
  \item each $\phi_e^\pm$, as a hieromorphism, is full;
  \item $\mathcal G$ has bounded supports;\label{item:atoroidal}
  \item \label{item:group:hypothesis_4}  if $e$ is an edge of $\Gamma$ and $S_e$ the $\nest$--maximal element of 
$\mathfrak S_e$, then for all $V\in\mathfrak S_{e^\pm}$, the elements $V$ and $\phi_{e^\pm}\inducedS(S_e)$ are not 
orthogonal in $\mathfrak S_{e^\pm}$;
\item for each vertex $v$ of $\Gamma$, there are finitely many $G_v$--orbits of subsets $\mathcal 
U\subset\mathfrak S_v$ for which the elements of $\mathcal U$ are 
pairwise-orthogonal;\label{item:pwo}
\item there exists $K\ge0$ such that for all vertices $v$ of $\widetilde\Gamma$ and edges $e$ incident to $v$, 
we have $\dist_{Haus}(\phi_v(G_e)),\mathbf F_{\phi_v\inducedS(S_e)}\times\{\star\})\leq K$, where $S_e\in\mathfrak 
S_e$ is the unique maximal element and $\star\in\mathbf E_{\phi_v\inducedS(S_e)}$.
\end{enumerate}
Then the total group $G$ of $\mathcal G$ is a hierarchically hyperbolic group.
\end{cor}

\begin{rem}
We have added hypothesis~\eqref{item:pwo} because it is exactly what's required.  In fact, it should follow from 
a stronger but more natural condition, namely that for each $V\in\mathfrak S_v$, the stabiliser in $G_v$ of the 
standard product region $\mathbf P_V$ acts cocompactly on $\mathbf P_V$.  This holds, for example, in the mapping 
class group.  On the other hand, this stronger condition it is not a consequence of the definition of an HHG 
since, for example, one can put exotic HHG structures on a free group where this fails.
\end{rem}

\begin{proof}[Proof of Corollary~\ref{cor:combination_theorem_for_HHG}]
By Theorem~\ref{thm:combination}, $(G,\mathfrak S)$ is a hierarchically hyperbolic space.  Observe that 
$G\leq\mathcal P_{\mathcal G}$, since $G$ acts on the Bass-Serre tree $\widetilde\Gamma$, and this action is induced 
by an action on $\bigcup_{v\in\mathcal V}\mathfrak S_v$ preserving the $\sim$--relation.  Hence the hierarchically 
hyperbolic structure $(G,\mathfrak S)$ can be chosen according to the constraints in 
Section~\ref{subsec:tree_of_HHS_automorphisms}, whence it is easily checked that $G$ acts on $\mathfrak S^0$ by HHS 
automorphisms.  Moreover, for any $[V]$, there are finitely many $\stabilizer_G([V])$--orbits of $\sim$--classes 
nested in $[V]$.  

The action on $\mathfrak S^0$ is co-finite since each $G_v$ is a hierarchically hyperbolic group.

Moreover, since $G$ preserves nesting and orthogonality in $\mathfrak S^0$, we have an induced action of $G$ on 
$\bigcup_\eta\mathcal K_\eta$ defined by 
$g\container([W_1],\ldots,[W_\eta])=\container([gW_1],\ldots,[gW_n])$.  We must show that this action 
(and hence the action of $G$ on $\mathfrak S$ obtained by combining this with the action on $\mathfrak S^0$) is 
cofinite.  

Since each element of $\mathcal K_n$ corresponds to a $n$--element pairwise-orthogonal set in 
$\mathfrak S^0$, and this correspondence is injective, it suffices to show that there are only finitely many 
$G$--orbits of such sets.  This follows from hypothesis~\eqref{item:pwo}.

Finally, the maps of the form $\pi_{\container}:G\to\fontact\container$ and $\rho^*_{\container}$ obviously satisfy 
the conditions required of an action by HHS automorphisms, since they are constant maps.  Finally, for all 
$\container\in\bigcup_\eta\mathcal K_\eta$ and $[W]$, and $g\in G$, we can choose the arbitrary constant map 
$\rho^{\container}_{[W]}$ so that $g(\rho^{\container}_{[W]})=\rho^{g\container}_{g[W]}$, where 
$g:\fontact[W]\to\fontact g[W]$ is the isometry from the automorphism action on $\mathfrak S^0$.  The same holds 
with $W$ replaced by $T$, since $\rho^{\container}_T$ was defined to be an intersection of support 
trees associated to $\container$, and $\rho^{g\container}_T$ is, by the definition of the $G$--action on 
$\bigcup_\eta\mathcal K_\eta$ and the $G$--equivariance of the assignment of each $\sim$--class to its support tree, 
the 
intersection of the $g$--translates of these support trees, i.e. $g\rho^{\container}_T$.  This completes the proof.
\end{proof}

\begin{rem}[Examples where the combination theorem does not apply]\label{prob:hhs_combination_examples}
Examples where one cannot apply Theorem~\ref{thm:combination} or Corollary~\ref{cor:combination_theorem_for_HHG} are 
likely to yield examples of groups that are not 
hierarchically hyperbolic groups, or even hierarchically hyperbolic spaces.  
\begin{enumerate}
 \item Let $\mathcal G$ be a finite graph of groups with $\integers^2$ vertex 
groups and $\integers$ edge groups, i.e., a \emph{tubular group}.  
In~\cite{Wise:tubular}, Wise completely characterized the tubular groups that 
act freely on CAT(0) cube complexes, and also characterized the (rare) tubular groups 
that admit cocompact such actions; Woodhouse recently gave a necessary and 
sufficient condition for the particular cube complex constructed 
in~\cite{Wise:tubular} to be finite-dimensional~\cite{Woodhouse:tubular}. These 
results suggest that there is little hope of producing 
hierarchically hyperbolic structures for tubular groups via cubulation, except 
in particularly simple cases.  

This is because the obstruction to cocompact cubulation is very similar to the obstruction to building a 
hierarchically hyperbolic structure using Theorem~\ref{thm:combination}.  
Indeed, if some vertex-group $G_v\cong\integers^2$ has more than $2$ 
independent incident edge-groups, then, if $\mathcal G$ satisfied the 
hypotheses of Theorem~\ref{thm:combination}, the hierarchically hyperbolic 
structure on $G_v$ would include $3$ pairwise-orthogonal unbounded elements, 
contradicting partial realization.  This shows that such a tubular group does 
not admit a hierarchically hyperbolic structure by virtue of the obvious 
splitting, and in fact shows that there is no hierarchically hyperbolic 
structure in which $G_v$ and the incident edge-groups are hierarchically 
quasiconvex.  

 \item \label{item:free_by_z}Let $G=F\rtimes_{\phi}\integers$, where $F$ is a 
finite-rank free group and $\phi:F\to F$ an automorphism.  When $F$ is 
atoroidal, $G$ is a hierarchically hyperbolic group simply by virtue of being 
hyperbolic~\cite{BestvinaFeighn:combination,Brinkmann:free_aut}.  There is also 
a more refined hierarchically hyperbolic structure in this case, in which all 
of the hyperbolic spaces involved are quasi-trees.  Indeed, by combining 
results in~\cite{HagenWise:freebyz} and~\cite{Agol:virtual_haken}, one finds 
that $G$ acts freely, cocompactly, and hence virtually co-specially on a CAT(0) 
cube complex, which therefore contains a $G$-invariant \emph{factor system} in 
the sense of~\cite{BehrstockHagenSisto:HHS_I} and is hence a hierarchically 
hyperbolic group; the construction in~\cite{BehrstockHagenSisto:HHS_I} ensures 
that the hierarchically hyperbolic structure for such cube complexes always 
uses a collection of hyperbolic spaces uniformly quasi-isometric to trees. However, the situation is presumably quite different when $G$ is not hyperbolic.  In this case, it seems that $G$ is rarely hierarchically hyperbolic.
\end{enumerate}
\end{rem}

\subsection{Products}\label{sec:product_HHS}
In this short section, we briefly describe a hierarchically hyperbolic structure on products of hierarchically hyperbolic spaces.

\begin{prop}[Product HHS]\label{prop:product_HHS}
Let $(\cuco X_0,\mathfrak S_0)$ and $(\cuco X_1,\mathfrak S_1)$ be
hierarchically hyperbolic spaces.  Then $\cuco X=\cuco X_0\times\cuco
X_1$ admits a hierarchically hyperbolic structure $(\cuco X,\mathfrak
S)$ such that for each of $i\in\{0,1\}$ the inclusion map $\cuco X_i\to\cuco
X$ induces a quasiconvex hieromorphism.
\end{prop}

\begin{proof}
Let $(\cuco X_i,\mathfrak S_i)$ be hierarchically hyperbolic spaces for $i\in\{0,1\}$.  Let
$\mathfrak S$ be a hierarchically hyperbolic structure consisting of
the disjoint union of $\mathfrak S_0$ and $\mathfrak S_1$ (together
with their intrinsic hyperbolic spaces, projections, and nesting,
orthogonality, and transversality relations), along with the following
domains whose associated hyperbolic spaces are points: $S$, into which
everything will be nested; $U_i$, for $i\in\{0,1\}$, into which
everything in $\mathfrak S_i$ is nested; for each $U\in\mathfrak S_i$
a domain $V_{U}$, with $|\fontact V_{U}|=1$, into which is nested
everything in $\mathfrak S_{i+1}$ and everything in $\mathfrak S_i$
orthogonal to $U$.  The elements $V_U$ are all transverse to $U_0$ and $U_1$.  Given $U,U'$, the elements $V_U,V_{U'}$ are 
transverse unless $U\nest U'$, in which case $V_U\nest V_{U'}$.  Projections $\pi_U\colon\cuco
X_0\times\cuco X_1\to U\in\mathfrak S$ are defined in the obvious way
when $U\not\in\mathfrak S_0\cup\mathfrak S_1$; otherwise, they are the
compositions of the existing projections with projection to the
relevant factor.  Projections of the form $\rho^U_V$ are either
defined already, uniquely determined, or are chosen to coincide with
the projection of some fixed basepoint (when $V\in\mathfrak
S_0\cup\mathfrak S_1$ and $U$ is not).  It is easy to check that this
gives a hierarchically hyperbolic structure on $\cuco X_1\times\cuco
X_2$.

The hieromorphisms $(\cuco X_i,\mathfrak S_i)\to(\cuco X,\mathfrak S)$ are inclusions on $\cuco X_i$ and $\mathfrak S$; for each $U\in\mathfrak S_i$, the map $\mathfrak S_i\ni \fontact U\to \fontact U\in\mathfrak S$ is the identity.  It follows immediately from the definitions that the diagrams from Definition~\ref{defn:hieromorphism} coarsely commute, so that these maps are indeed hieromorphisms.  Hierarchical quasiconvexity likewise follows from the definition.
\end{proof}

Product HHS will be used in defining hierarchically hyperbolic structures on graph manifolds in Section~\ref{sec:hier_hyp_3_manifolds}.  The next result follows directly from the proof of the previous 
proposition.
\begin{cor}\label{cor:product_HHG}
    Let $G_{0}$ and $G_{1}$ be hierarchically hyperbolic groups. Then $G_{0}\times G_{1}$ is a hierarchically hyperbolic 
    group. 
\end{cor}

\section{Hyperbolicity relative to HHGs}\label{sec:hyperbolicity_rel_HHS}

Relatively hyperbolic groups possess natural 
hierarchically hyperbolic structures:

\begin{thm}[Hyperbolicity relative to HHGs] \label{thm:rel_hyp}
Let the group $G$ be hyperbolic relative to a finite collection
$\mathcal P$ of peripheral subgroups.  If each $P\in\mathcal
P$ is a hierarchically hyperbolic space, then $G$ is a hierarchically
hyperbolic space. Further, if each $P\in\mathcal
P$ is a hierarchically hyperbolic group, then so is $G$.
\end{thm}

\begin{proof}
    We prove the statement about hierarchically hyperbolic groups; 
    the statement about spaces follows {\it a fortiori}. 
    
For each $P\in\mathcal P$, let $(P,\mathfrak S_P)$ be a hierarchically
hyperbolic group structure.  For convenience, assume that the
$P\in\mathcal P$ are pairwise non-conjugate (this will avoid
conflicting hierarchically hyperbolic structures).  For each $P$ and
each left coset $gP$, let $\mathfrak S_{gP}$ be a copy of $\mathfrak
S_P$ (with associated hyperbolic spaces and projections), so
that there is a hieromorphism $(P,\mathfrak S_P)\to(gP,\mathfrak
S_{gP})$, equivariant with respect to the conjugation
isomorphism $P\to P^g$.

Let $\widehat G$ be the usual hyperbolic space formed from $G$ by
coning off each left coset of each $P\in\mathcal P$.  Let $\mathfrak
S=\{\widehat G\}\cup\bigsqcup_{gP\in G\mathcal P}\mathfrak S_{gP}$.
The nesting, orthogonality, and transversality relations on each
$\mathfrak S_{gP}$ are as defined above; if $U,V\in\mathfrak
S_{gP},\mathfrak S_{g'P'}$ and $gP\neq g'P'$, then declare
$U\transverse V$.  Finally, for all $U\in\mathfrak S$, let
$U\nest\widehat G$.  The hyperbolic space $\fontact\widehat G$ is $\widehat G$, while the hyperbolic space $\fontact U$ associated to
each $U\in\mathfrak S_{gP}$ was defined above.

The projections are defined as follows: $\pi_{\widehat G}\co G\to\widehat
G$ is the inclusion, which is coarsely surjective and hence has quasiconvex image.  For 
each $U\in \mathfrak S_{gP}$, let
$\gate_{gP}\co G\to gP$ be the closest-point projection onto $gP$ and let
$\pi_{U}=\pi_U\circ\gate_{gP}$, to extend the domain of $\pi_U$ from
$gP$ to $G$.  Since each $\pi_U$ was coarsely Lipschitz on $U$ with quasiconvex image, and the closest-point projection is 
uniformly coarsely Lipschitz, the projection $\pi_U$ is uniformly coarsely Lipschitz and has quasiconvex image.  For each 
$U,V\in\mathfrak S_{gP}$, the coarse maps
$\rho_U^V$ and $\rho^U_V$ were already defined.  If $U\in\mathfrak
S_{gP}$ and $V\in\mathfrak S_{g'P'}$, then
$\rho^U_V=\pi_V(\gate_{g'P'}(gP))$, which is a uniformly bounded set
(here we use relative hyperbolicity, not just the weak relative
hyperbolicity that is all we needed so far).  Finally, for
$U\neq\widehat G$, we define $\rho^U_{\widehat G}$ to be the
cone-point over the unique $gP$ with $U\in\mathfrak S_{gP}$, and
$\rho^{\widehat G}_U\co \widehat G\to\fontact U$ is defined as follows:
for $x\in G$, let $\rho^{\widehat G}_U(x)=\pi_U(x)$.  If $x\in\widehat
G$ is a cone-point over $g'P'\neq gP$, let $\rho^{\widehat
G}_U(x)=\rho^{S_{g'P'}}_U$, where $S_{g'P'}\in\mathfrak S_{g'P'}$ is
$\nest$--maximal.  The cone-point over $gP$ may be sent anywhere in
$U$.

By construction, to verify that $(G,\mathfrak S)$ is a hierarchically
hyperbolic group structure, it suffices to verify that it satisfies
the remaining axioms for a hierarchically hyperbolic space given in
Definition~\ref{defn:space_with_distance_formula}, since the
additional $G$--equivariance conditions hold by construction.
Specifically, it remains to verify consistency, bounded geodesic image
and large links, partial realization, and uniqueness.

\textbf{Consistency:} The nested consistency inequality holds
automatically within each $\mathfrak S_{gP}$, so it remains to verify
it only for $U\in\mathfrak S_{gP}$ versus $\widehat G$, but this
follows directly from the definition: if $x\in G$ is far in $\widehat
G$ from the cone-point over $gP$, then $\rho^{\widehat
G}_U(x)=\pi_U(x)$, by definition.  To verify the transverse
inequality, it suffices to consider $U\in\mathfrak
S_{gP},V\in\mathfrak S_{g'P'}$ with $gP\neq g'P'$.  Let $x\in G$ and
let $z=\gate_{g'P'}(x)$.  Then, if $\dist_U(x,z)$ is sufficiently
large, then $\dist_{gP}(x,z)$ is correspondingly large, so that by
Lemma~1.15 of~\cite{Sisto-distformrelhyp}, $\gate_{g'P'}(x)$ and
$\gate_{g'P'}(gP)$ coarsely coincide, as desired.

The last part of the consistency axiom,
Definition~\ref{defn:space_with_distance_formula}.\eqref{item:dfs_transversal},
holds as follows.  Indeed, if $U\nest V$, then either $U=V$, and there
is nothing to prove.  Otherwise, if $U\nest V$ and either $V\propnest
W$ or $W\transverse V$, then either $U,V\in\mathfrak S_{gP}$ for some
$g,P$, or $U\in\mathfrak S_{gP}$ and $V=\widehat G$.  The latter
situation precludes the existence of $W$, so we must be in the former
situation.  If $W\in\mathfrak S_{gP}$, we are done since the axiom
holds in $\mathfrak S_{gP}$.  If $W=\widehat G$, then $U,V$ both
project to the cone-point over $gP$, so $\rho^U_W=\rho^V_W$.  In the
remaining case, $W\in\mathfrak S_{g'P'}$ for some $g'P'\neq gP$, in
which case $\rho^U_W,\rho^V_W$ both coincide with
$\pi_W(\gate_{g'P'}(gP))$.

\textbf{Bounded geodesic image:} Bounded geodesic image holds within
each $\mathfrak S_{gP}$ by construction, so it suffices to consider
the case of $U\in\mathfrak S_{gP}$ nested into $\widehat G$.  Let
$\hat\gamma$ be a geodesic in $\widehat G$ avoiding $gP$ and the cone
on $gP$.  Lemma~1.15 of~\cite{Sisto-distformrelhyp} ensures that any
lift of $\hat\gamma$ has uniformly bounded projection on $gP$, so
$\rho^{\widehat G}_U\circ\hat\gamma$ is uniformly bounded.

\textbf{Large links:} The large link axiom
(Definition~\ref{defn:space_with_distance_formula}.\eqref{item:dfs_large_link_lemma})
can be seen to hold in $(G,\mathfrak S)$ by combining the large link
axiom in each $gP$ with malnormality of $\mathcal P$ and Lemma~1.15
of~\cite{Sisto-distformrelhyp}.

\textbf{Partial realization:} This follows immediately from partial
realization within each $\mathfrak S_{gP}$ and the fact that no new
orthogonality was introduced in defining $(G,\mathfrak S)$, together
with the definition of $\widehat G$ and the definition of projection
between elements of $\mathfrak S_{gP}$ and $\mathfrak S_{g'P'}$ when
$gP\neq g'P'$.  More precisely, if $U\in\mathfrak S_{gP}$ and
$p\in\fontact U$, then by partial realization within $gP$, there
exists $x\in gP$ so that $\dist_U(x,p)\leq\alpha$ for some fixed
constant $\alpha$ and $\dist_V(x,\rho^U_V)\leq\alpha$ for all
$V\in\mathfrak S_{gP}$ with $U\nest V$ or $U\transverse V$.  Observe
that $\dist_{\widehat G}(x,\rho^U_{\widehat G})=1$, since $x\in gP$
and $\rho^U_{\widehat G}$ is the cone-point over $gP$.  Finally, if
$g'P'\neq gP$ and $V\in\mathfrak S_{g'P'}$, then
$\dist_V(x,\rho^U_V)=\dist_V(\pi_V(\gate_{g'P'}(x)),\pi_V(\gate_{g'P'}(gP)))=0$
since $x\in gP$.

\textbf{Uniqueness:} If $x,y$ are uniformly close in $\widehat G$,
then either they are uniformly close in $G$, or they are uniformly
close to a common cone-point, over some $gP$, whence the claim follows
from the uniqueness axiom in $\mathfrak S_{gP}$.
\end{proof}

\begin{rem}
    The third author established a characterization of relative
    hyperbolicity in terms of projections in \cite{Sisto-distformrelhyp}. 
    Further, there it was proven that like mapping class groups, 
    there was a natural way to compute distances in 
    relatively hyperbolic groups from certain related spaces, namely:
    if $(G,\mathcal
    P)$ is relatively hyperbolic, then distances in $G$ are coarsely
    obtained by summing the corresponding distance in the coned-off
    Cayley graph $\widehat G$ together with the distances between
    projections in the various $P\in\mathcal P$ and their cosets. 
    We recover a  
    new proof of Sisto's formula as a consequence of Theorem~\ref{thm:rel_hyp} 
    and Theorem~\ref{thm:distance_formula}.
\end{rem}

Theorem~\ref{thm:rel_hyp} will be used in our analysis of $3$--manifold groups in Section~\ref{sec:hier_hyp_3_manifolds}.  However, there is a more general statement in the context of metrically relatively hyperbolic spaces (e.g., what Drutu--Sapir call 
asymptotically tree-graded \cite{DrutuSapir:TreeGraded}, or spaces 
that satisfy the 
equivalent condition on projections formulated in \cite{Sisto:metric_rel_hyp}).  For instance, arguing exactly as in the proof of Theorem~\ref{thm:rel_hyp} shows that if the space $\cuco X$ is hyperbolic relative to a collection of uniformly hierarchically hyperbolic spaces, then $\cuco X$ admits a hierarchically hyperbolic structure (in which each peripheral subspace embeds hieromorphically).  

More generally, let the geodesic metric space $\cuco X$ be hyperbolic relative to a collection $\mathcal P$ of subspaces, and let $\widehat{\cuco X}$ be the hyperbolic space obtained from $\cuco X$ by coning off each $P\in\mathcal P$.  Then we can endow $\cuco X$ with a hierarchical space structure as follows:
\begin{itemize}
 \item the index-set $\mathfrak S$ consists of $\mathcal P$ together with an additional index $S$;
 \item for all $P,Q\in\mathcal P$, we have $P\transverse Q$, while $P\propnest S$ for all $P\in\mathcal P$ (the orthogonality relation is empty and there is no other nesting);
 \item for each $P\in\mathcal P$, we let $\fontact P=P$;
 \item we declare $\fontact S=\widehat{\cuco X}$;
 \item the projection $\pi_S:\cuco X\to\widehat{\cuco X}$ is the inclusion;
 \item for each $P\in\mathcal P$, let $\pi_P:\cuco X\to P$ be the closest-point projection onto $P$ (which is surjective);
 \item for each $P\in\mathcal P$, let $\rho^P_S$ be the cone-point in $\widehat{\cuco X}$ associated to $P$;
 \item for each $P\in\mathcal P$, let $\rho^S_P:\widehat{\cuco X}\to P$ be defined by $\rho^S_P(x)=\pi_P(x)$ for $x\in\cuco X$, while $\rho^S_P(x)=\pi_P(Q)$ whenever $x$ lies in the cone on $Q\in\mathcal P$.
 \item for distinct $P,Q\in\mathcal P$, let $\rho^P_Q=\pi_Q(P)$ (which is uniformly bounded since $\cuco X$ is hyperbolic relative to $\mathcal P$).
\end{itemize}

The above definition yields:

\begin{thm}\label{thm:rel_hyp_space}
Let the geodesic metric space $\cuco X$ be hyperbolic relative to the collection $\mathcal P$ of subspaces.  Then, with $\mathfrak S$ as above, we have that $(\cuco X,\mathfrak S)$ is a hierarchical space, and is moreover relatively hierarchically hyperbolic.
\end{thm}

\begin{proof}
By definition, for each $U\in\mathfrak S$, we have that either $U=S$ and $\fontact S=\widehat{\cuco X}$ is hyperbolic, or $U$ is $\nest$--minimal.  The rest of the conditions of Definition~\ref{defn:space_with_distance_formula} are verified as in the proof of Theorem~\ref{thm:rel_hyp}.
\end{proof}

\section{Hierarchical hyperbolicity of 3-manifold groups}\label{sec:hier_hyp_3_manifolds}

In this section we show that fundamental groups of most  
$3$--manifolds admit hierarchical 
hyperbolic structures. More precisely, we prove:

\begin{thm}[$3$--manifolds are hierarchically hyperbolic]\label{thm:3mflds}
    Let $M$ be a closed $3$--manifold. 
    Then $\pi_{1}(M)$ is a hierarchically hyperbolic space if and 
    only if $M$ does not have a Sol or Nil component in its prime decomposition.
\end{thm}

\begin{proof}
    It is well known that for a closed irreducible $3$--manifold $N$ the Dehn function of 
    $\pi_{1}(N)$ is linear if $N$ is hyperbolic, 
    cubic if $N$ is Nil, exponential if $N$ is Sol, and quadratic 
    in all other cases. Hence by 
    Corollary~\ref{cor:HHG_quadratic_Dehn_function}, if $\pi_{1}(M)$ 
    is a hierarchically hyperbolic space, then $M$ does not contain Nil or 
    Sol manifolds in its prime decomposition. It remains to prove the converse.
    
    Since the fundamental group of any reducible $3$--manifold is the
    free product of irreducible ones, the reducible case immediately
    follows from the irreducible case by Theorem~\ref{thm:rel_hyp}.
    
    When $M$ is geometric and not Nil or Sol, then 
    $\pi_{1}(M)$ is quasi-isometric to one of the following: 
    \begin{itemize}
    \item $\reals^3$ is hierarchically hyperbolic by Proposition~\ref{prop:product_HHS};
    \item    $\hyperbolic^3$, $\mathbb S^3$, $\mathbb S^2\times\reals$ are (hierarchically) hyperbolic; 
    \item $\hyperbolic^{2}\times\reals$ and $PSL_2(\reals)$: the first is hierarchically hyperbolic by Proposition~\ref{prop:product_HHS}, whence the second is also since it is quasi-isometric to the first by~\cite{Rieffel:H2crossR}.
    \end{itemize}
        
    We may now assume $M$ is not geometric. 
    Our main step is to show 
    that any irreducible non-geometric graph manifold group is 
    a hierarchically hyperbolic space.
    
    Let $M$ be an irreducible non-geometric graph manifold. 
    By \cite[Theorem
    2.3]{KapovichLeeb:3manifolds}, by replacing $M$ by a manifold 
    whose fundamental group is quasi-isometric to that of $M$, we may assume that 
    our manifold is a \emph{flip 
    graph manifold}, i.e., each Seifert fibered space component is a trivial
    circle bundles over a surfaces of genus at least 2 and each pair 
    of adjacent Seifert fibered spaces are glued by flipping the 
    base and fiber directions.

    Let $X$ be the universal cover of $M$.  
    The decomposition of $M$ into geometric
    components induces a decomposition of $X$ into subspaces $\{S_v\}$,
    one for each vertex $v$ of the Bass-Serre tree $T$ of $M$.  Each such
    subspace $S_v$ is bi-Lipschitz homeomorphic to the product of a copy
    $R_v$ of the real line with the universal cover $\Sigma_v$ of a
    hyperbolic surface with totally geodesic boundary, and there are
    maps $\phi_v\co S_v\to\Sigma_v$ and $\psi_v\co S_v\to R_v$. Notice that $\Sigma_v$ is hyperbolic, and in particular hierarchically hyperbolic. However, for later purposes, we endow $\Sigma_v$ with the hierarchically hyperbolic structure originating from the fact that $\Sigma_v$ is hyperbolic relative to its boundary components, see Theorem \ref{thm:rel_hyp}.
    
    By Proposition~\ref{prop:product_HHS} each $S_{v}$ is a 
    hierarchically hyperbolic space and thus we have a tree of hierarchically hyperbolic spaces. 
    Each edge space is a product  $\partial_{0}\Sigma_v 
    \times R_{v}$, where $\partial_{0}\Sigma_v$ is a particular 
    boundary component of $\Sigma_{v}$ determined by the adjacent 
    vertex. Further, since the graph manifold is flip, we also 
    have that for each vertices $v,w$ of the tree, 
    the edge-hieromorphism between $S_{v}$ and $S_{w}$ sends 
    $\partial_{0}\Sigma_v$ to $R_{w}$ and $R_{v}$ to $\partial_{0}\Sigma_w$.

    We now verify the hypotheses of Theorem \ref{thm:combination}.
    The first hypothesis is that 
    there exists $k$ so that 
    each edge-hieromorphism is $k$--hierarchically quasiconvex. This is easily seen 
    since the edge-hieromorphisms have the 
    simple form described above. The second hypothesis of 
    Theorem~\ref{thm:combination}, fullness of edge-hieromorphisms, also follows immediately from the explicit description of 
    the edges here and the simple hierarchically hyperbolic structure of the 
    edge spaces.
    
    The third hypothesis of Theorem~\ref{thm:combination} requires that the 
    tree has bounded supports.  We can assume that the product regions $S_v$ are maximal in the sense that each edge-hieromorphism sends the fiber direction $R_v$ to $\partial_0\Sigma_w$ in each adjacent $S_w$.  It follows that the support of each $\sim$--class (in the language of Theorem~\ref{thm:combination}) consists of at most $2$ vertices.  The last hypothesis of Theorem~\ref{thm:combination} is about 
    non-orthogonality of maximal elements and again follows directly 
    from the explicit hierarchically hyperbolic structure.  Moreover, the part of the hypothesis about edge-spaces 
coinciding coarsely with standard product regions in vertex spaces follows from the explicit hierarchically hyperbolic 
structure.
    
    All the hypotheses of Theorem~\ref{thm:combination} are 
    satisfied, so $\pi_1M$ (with any word metric) is a hierarchically hyperbolic space for all irreducible 
    non-geometric graph manifolds $M$.
    
    The general case that the fundamental
    group of any non-geometric $3$--manifold is an
    hierarchically hyperbolic space now follows immediately 
    by Theorem~\ref{thm:rel_hyp},
    since any $3$--manifold group is hyperbolic relative to
    its maximal graph manifold subgroups.
\end{proof}

\begin{rem}[(Non)existence of HHG structures for 3--manifold groups]\label{rem:graph_HHS_cube}
The proof of Theorem~\ref{thm:3mflds} shows that for many $3$--manifolds $M$, the group $\pi_1M$ is not merely a hierarchically hyperbolic space (when endowed with the word metric arising from a finite generating set), but is actually a hierarchically hyperbolic group.  Specifically, if $M$ is virtually compact special, then $\pi_1M$ acts freely and cocompactly on a CAT(0) cube complex $\cuco X$ that is the universal cover of a compact special cube complex.  Hence $\cuco X$ contains a $\pi_1M$--invariant factor system (see~\cite[Section~8]{BehrstockHagenSisto:HHS_I}) consisting of a $\pi_1M$--finite set of convex subcomplexes.  This yields a hierarchically hyperbolic structure $(\cuco X,\mathfrak S)$ where $\pi_1M\leq\Aut(\mathfrak S)$ acts cofinitely on $\mathfrak S$ and geometrically on $\cuco X$, i.e., $\pi_1M$ is a hierarchically hyperbolic group.

The situation is quite different when $\pi_1M$ is not virtually \textbf{compact} special.  Indeed, when $M$ is a nonpositively-curved graph manifold, $\pi_1M$ virtually acts freely, but not necessarily cocompactly, on a CAT(0) cube complex $\cuco X$, and the quotient is virtually special; this is a result of Liu~\cite{Liu:graph_man} which was also shown to hold in the case where $M$ has nonempty boundary by Przytycki and Wise~\cite{PrzytyckiWise:graph_man}.  Moreover, $\pi_1M$ acts with finitely many orbits of hyperplanes.  Hence the $\pi_1M$--invariant factor system on $\cuco X$ from~\cite{BehrstockHagenSisto:HHS_I} yields a $\pi_1M$--equivariant HHS structure $(\cuco X,\mathfrak S)$ with $\mathfrak S$ $\pi_1M$--finite.  However, this yields an HHG structure on $\pi_1M$ only if the action on $\cuco X$ is cocompact.  In~\cite{HagenPrzytycki:graph}, the second author and Przytycki showed that $\pi_1M$ virtually acts freely and cocompactly on a CAT(0) cube complex, with special quotient, only in the very particular situation where $M$ is \emph{chargeless}.  This essentially asks that the construction of the hierarchically hyperbolic structure on $\widetilde M$ from the proof of Theorem~\ref{thm:3mflds} can be done $\pi_1M$--equivariantly.  In general, this is impossible: recall that we passed from $\widetilde M$ to the universal cover of a flip manifold using a (nonequivariant) quasi-isometry.  Motivated by this observation and the fact that the range of possible HHS structures on the universal cover of a JSJ torus is very limited, we conjecture that $\pi_1M$ is a hierarchically hyperbolic group if and only if $\pi_1M$ acts freely and cocompactly on a CAT(0) cube complex. 
\end{rem}

\section{A new proof of the distance formula for mapping class 
groups}\label{sec:MCG_HHS}

We now describe the hierarchically hyperbolic structure of mapping 
class groups. In \cite{BehrstockHagenSisto:HHS_I} we gave a proof of 
this result using several of the main results of \cite{Behrstock:asymptotic, 
BKMM:consistency, MasurMinsky:I, MasurMinsky:II}. Here we give an elementary proof which is independent 
of the Masur--Minsky ``hierarchy machinery.''  
One consequence of this is a new and concise proof of the celebrated 
Masur--Minsky distance formula \cite[Theorem 6.12]{MasurMinsky:II}, 
which we obtain  
by combining Theorem \ref{thm:MCGhh} and Theorem \ref{thm:distance_formula}.

\begin{enumerate}
 \item Let $S$ be closed connected oriented surface of finite type and let $\mathcal M(S)$ be its marking complex.
 \item Let $\mathfrak S$ be the collection of isotopy classes of essential subsurfaces of $S$, and for each $U\in\mathfrak S$ let $\fontact U$ be its curve complex.
 \item The relation $\nest$ is nesting, $\perp$ is disjointness and $\transverse$ is overlapping.
 \item For each $U\in\mathfrak S$, let $\pi_U\co\mathcal 
 M(S)\to\fontact U$ be the (usual) subsurface projection. For 
 $U,V\in\mathfrak S$ satisfying either $U\nest V$ or $U\transverse V$,
 denote $\rho^U_V=\pi_V(\partial U)\in \fontact V$, while for $V\nest
 U$ let $\rho^U_V\co\fontact U\to 2^{\fontact V}$ be the subsurface
 projection.
\end{enumerate}

\begin{thm}\label{thm:MCGhh}
 Let $S$ be closed connected oriented surface of finite type.  Then,
 $(\mathcal M(S),\mathfrak S)$ is a
 hierarchically hyperbolic space, for $\mathfrak S$ as above.  In particular the mapping class
 group $\MCG(S)$ is a hierarchically hyperbolic group.
\end{thm}

\begin{proof}
Hyperbolicity of curve graphs is the main result of
\cite{MasurMinsky:I}; more recent proofs of this were 
found in 
\cite{Aougab:uniform,Bowditch:uniform,ClayRafiSchleimer:uniform,
HenselPrzytyckiWebb:unicorn,PrzytyckiSisto:universe}, some of which are elementary.

 Axioms \ref{item:dfs_curve_complexes}, \ref{item:dfs_nesting}, \ref{item:dfs_orthogonal} and 
\ref{item:dfs_complexity} are clear (an elementary exposition of the Lipschitz condition for subsurface projections is provided in \cite[Lemma 
2.5]{MasurMinsky:II}, and the projections have quasiconvex image because they are coarsely surjective).  Both parts of axiom
\ref{item:dfs_transversal} can be found in
 \cite{Behrstock:asymptotic}.  The nesting part is elementary, and a
 short elementary proof in the overlapping case was obtained by 
 Leininger and can be found in \cite{Mangahas:UU}.
 
 Axiom \ref{item:dfs:bounded_geodesic_image} was proven in \cite{MasurMinsky:II}, and an elementary proof is available in \cite{Webb:BGI}. In fact, in the aforementioned papers it is proven that there exists a constant $C$ so that for any subsurface $W$, markings $x,y$ and geodesic from $\pi_W(x)$ to $\pi_W(y)$ the following holds. If $V\nest W$ and $V\neq W$ satisfies $d_V(x,y)\geq C$ then some curve along the given geodesic does not intersect $\partial V$. This implies Axiom \ref{item:dfs_large_link_lemma}, since we can take the $T_i$ to be the complements of curves appearing along the aforementioned geodesic.

Axiom \ref{item:dfs_partial_realization} follows easily from the
following statement, which clearly holds:  For any given collection of
disjoint subsurfaces and curves on the given subsurfaces, there exists
a marking on $S$ that contains the given curves as base curves (or, up
to bounded error, transversals in the case that the corresponding
subsurface is an annulus).

Axiom \ref{item:dfs_uniqueness} is hence the only delicate one. We 
are finished modulo this last axiom which we verify below in 
Proposition~\ref{prop:mcg_uniqueness_axiom} (see also~\cite[Proposition 5.11]{BBF:quasi_tree}).
\end{proof}

\begin{prop}\label{prop:mcg_uniqueness_axiom}
 $(\mathcal M(S),\mathfrak S)$ satisfies the uniqueness axiom, i.e., for each $\kappa\geq 0$, there exists $\theta_u=\theta_u(\kappa)$ such that if $x,y\in\mathcal M(S)$ satisfy $\dist_U(x,y)\leq \kappa$ for each $U\in\mathfrak S$ then $\dist_{\mathcal M(S)}(x,y)\leq\theta_u$.
\end{prop}

\begin{proof} 
    Note that when the complexity (as measured by the quantity $3g+p-3$
    where $g$ is the genus and $p$ the number of punctures) is less 
    than 2 then $\mathcal M(S)$ is hyperbolic and thus the axiom 
    holds. We will proceed by inducting on complexity: thus we will
    fix $S$ to have complexity at least 2 and assume that all the axioms
    for a hierarchically hyperbolic space, including the uniqueness 
    axiom, hold for each proper subsurface
    of $S$.

Now, having fixed our surface $S$, the proof is by induction on $\dist_{\fontact S}(base(x),base(y))$.

If $\dist_{\fontact S}(base(x),base(y))=0$, then $x$ and $y$ share
some non-empty multicurve $\sigma=c_1\cup\dots\cup c_k$.  For $x',y'$
the restrictions of $x,y$ to $S- \sigma$ we have that, by
induction, $\dist_{\mathcal M(S- \sigma)}(x',y')$ is bounded
in terms of $\kappa$.  We then take the markings in a geodesic
in $\mathcal M(S- \sigma)$ from $x'$ to $y'$ and extend these 
all in the same way 
to obtain markings in $\mathcal M(S)$ which yield a path in $\mathcal M(S)$ from $x$ to $\hat{y}$ whose length
is bounded in terms of $\kappa$, where $\hat{y}$ is the marking for 
which:

\begin{itemize}
 \item $\hat{y}$ has the same base curves as $y$,
 \item the transversal for each $c_i$ is the same as the corresponding transversal for $x$, and
 \item the transversal for each curve in $base(y)- \{c_i\}$ is the same as the corresponding transversal for $y$.
\end{itemize}

Finally, it is readily seen that $\dist_{\mathcal M(S)}(\hat{y},y)$ is
bounded in terms of $\kappa$ because the transversals of each $c_i$ in
the markings $x$ and $y$ are within distance $\kappa$ of each other.
This completes the proof of the base case of the Proposition.

Suppose now that that the statement holds whenever $\dist_{\fontact
S}(base(x),base(y))\leq n$, and let us prove it in the case
$\dist_{\fontact S}(base(x),base(y))= n+1$.  Let $c_x\in base(x)$ and
$c_y\in base(y)$ satisfy $\dist_{\fontact S}(c_x,c_y)=n+1$.  Let
$c_x=\sigma_0,\dots,\sigma_{n+1}=c_y$ be a tight geodesic (hence, each
$\sigma_i$ is a multicurve).  Let $\sigma$ be the union of $\sigma_0$
and $\sigma_1$.  Using the realization theorem in the subsurface $S- \sigma$
we can find a marking $x'$ in $S- \sigma$ whose projections
onto each $\fontact U$ for $U\subseteq S- \sigma$ coarsely
coincide with $\pi_U(y)$.  Let $\hat x$ be the marking for which: 
\begin{itemize}
 \item $base(\hat x)$ is the union of $base(x')$ and $\sigma$,
 \item the transversal in $\hat x$ of curves in $base(\hat x)\cap base(x')$ are the same as those in $x'$,
 \item the transversal of $c_x$ in $\hat x$ is the same as the one in $x$,
 \item the transversal in $\hat x$ of a curve $c$ in $\sigma_1$ is $\pi_{A_c}(y)$, where $A_c$ is an annulus around $c$.
\end{itemize}
Note that $\dist_{\fontact S}(base(\hat x),base(y))=n$. Hence, the following two claims conclude the proof.
\renewcommand{\qedsymbol}{$\blacksquare$}

\par\medskip
{\bf Claim 1.} $\dist_{\mathcal M(S)}(x,\hat x)$ is bounded in terms of $\kappa$.
\par\medskip

\begin{proof}
It suffices to bound $\dist_{\fontact U}(x,\hat x)$ in terms of $\kappa$ for each $U\subseteq S- c_x$. In fact, once we do that, by induction on complexity we know that we can bound $\dist_{\mathcal M(S- c_x)}(z,\hat{z})$, where $z,\hat{z}$ are the restrictions of $x,\hat{x}$ to $S- c_x$, whence the conclusion easily follows.

If $U$ is contained in $S- \sigma$, then the required bound follows since $\pi_U(\hat{x})$ coarsely coincides with $\pi_U(x')$ in this case.

If instead $\partial U$ intersects $\sigma_1$, then $\pi_U(\hat{x})$ coarsely coincides with $\pi_U(\sigma_1)$.

At this point, we only have to show that $\pi_U(\sigma_1)$ coarsely coincides with $\pi_U(y)$, and in order to do so we observe that we can apply the bounded geodesic image theorem to the geodesic $\sigma_1,\dots,\sigma_{n+1}$. In fact, $\sigma_1$ intersects $\partial U$ by hypothesis and $\sigma_i$ intersects $\partial U$ for $i\geq 3$ because of the following estimate that holds for any given boundary component $c$ of $\partial U$
$$\dist_{\fontact(S)}(\sigma_i,c)\geq \dist_{\fontact(S)}(\sigma_i,\sigma_0)-\dist_{\fontact(S)}(\sigma_0,c)\geq i-1>1.$$
Finally, $\sigma_2$ intersects $\partial U$ because of the definition of tightness: $\partial U$ intersects $\sigma_1$, hence it must intersect $\sigma_0\cup\sigma_2$. However, it does not intersect $\sigma_0$, whence it intersects $\sigma_2$.
\end{proof}

\par\medskip
{\bf Claim 2.} There exists $\kappa'$, depending on $\kappa$, so that for each subsurface $U$ of $S$ we have $\dist_{\fontact U}(\hat x,y)\leq \kappa'$.
\par\medskip

\begin{proof} If $\sigma_0$ intersects $\partial U$, then $\pi_{U}(\hat x)$ coarsely coincides with $\pi_U(\sigma_0)$. In turn, $\pi_U(\sigma_0)$ coarsely coincides with $\pi_U(x)$, which is $\kappa$--close to $\pi_U(y)$.

On the other hand, if $U$ does not intersect $\sigma$, then we are done by the definition of $x'$.

Hence, we can assume that $U$ is contained in $S- \sigma_0$ and that $\sigma_1$ intersects $\partial U$. In particular, $\pi_U(\hat x)$ coarsely coincides with $\pi_U(\sigma_1)$. But we showed in the last paragraph of the proof of Claim 1 that $\pi_U(\sigma_1)$ coarsely coincides with $\pi_U(y)$, so we are done.
\end{proof}

\renewcommand{\qedsymbol}{$\Box$}
As explained above, the proofs of the above two claims complete the proof.
\end{proof}

\bibliographystyle{alpha}
\bibliography{hier_hyp_II}

\def\cprime{$'$}
\begin{thebibliography}{BKMM12}

\bibitem[AB]{AbbottBehrstock:linear}
Carolyn Abbott and Jason Behrstock.
\newblock Conjugator lengths in hierarchically hyperbolic groups.
\newblock {\em {preprint arXiv:1808.09604}}.

\bibitem[ABD]{AbbottBehrstock}
Carolyn Abbott, Jason Behrstock, and Matthew~G Durham.
\newblock Largest acylindrical actions and stability in hierarchically
  hyperbolic groups.
\newblock {\em {preprint arXiv:1705.06219}}.

\bibitem[Ago13]{Agol:virtual_haken}
Ian Agol.
\newblock The virtual {H}aken conjecture.
\newblock {\em Doc. Math.}, 18:1045--1087, 2013.
\newblock With an appendix by Agol, Daniel Groves, and Jason Manning.

\bibitem[Aou13]{Aougab:uniform}
Tarik Aougab.
\newblock Uniform hyperbolicity of the graphs of curves.
\newblock {\em Geometry \& Topology}, 17(5):2855--2875, 2013.

\bibitem[BBF15]{BBF:quasi_tree}
Mladen Bestvina, Ken Bromberg, and Koji Fujiwara.
\newblock Constructing group actions on quasi-trees and applications to mapping
  class groups.
\newblock {\em Publications math{\'e}matiques de l'IH{\'E}S}, 122(1):1--64,
  2015.

\bibitem[BDS11]{BehrstockDrutuSapir:MCGsubgroups}
J.~Behrstock, C.~Dru\c{t}u, and M.~Sapir.
\newblock Median structures on asymptotic cones and homomorphisms into mapping
  class groups.
\newblock {\em Proc. London Math. Soc.}, 102(3):503--554, 2011.

\bibitem[Beh06]{Behrstock:asymptotic}
J.~Behrstock.
\newblock Asymptotic geometry of the mapping class group and {T}eichm\"{u}ller
  space.
\newblock {\em Geometry \& Topology}, 10:2001--2056, 2006.

\bibitem[BF92]{BestvinaFeighn:combination}
M.~Bestvina and M.~Feighn.
\newblock A combination theorem for negatively curved groups.
\newblock {\em J. Differential Geometry}, 35:85--101, 1992.

\bibitem[BF02]{BestvinaFujiwara:boundedcohom}
M.~Bestvina and K.~Fujiwara.
\newblock Bounded cohomology of subgroups of mapping class groups.
\newblock {\em Geometry \& Topology}, 6:69--89 (electronic), 2002.

\bibitem[BF14a]{BestvinaFeighn:hyperbolicCFF}
Mladen Bestvina and Mark Feighn.
\newblock Hyperbolicity of the complex of free factors.
\newblock {\em Adv. Math.}, 256:104--155, 2014.

\bibitem[BF14b]{BestvinaFeighn:projections}
Mladen Bestvina and Mark Feighn.
\newblock Subfactor projections.
\newblock {\em J. Topol.}, 7(3):771--804, 2014.

\bibitem[BH16]{BehrstockHagen:cubulated1}
J.~Behrstock and M.F.\ Hagen.
\newblock Cubulated groups: thickness, relative hyperbolicity, and simplicial
  boundaries.
\newblock {\em Geometry, Groups, and Dynamics}, 10(2):649--707, 2016.

\bibitem[BHS17a]{BehrstockHagenSisto:HHS_I}
J.~Behrstock, M.F. Hagen, and A.~Sisto.
\newblock {Hierarchically hyperbolic spaces I: curve complexes for cubical
  groups}.
\newblock {\em Geom. Topol.}, 21:1731--1804, 2017.

\bibitem[BHS17b]{HHS_3}
Jason Behrstock, Mark~F Hagen, and Alessandro Sisto.
\newblock Asymptotic dimension and small-cancellation for hierarchically
  hyperbolic spaces and groups.
\newblock {\em Proceedings of the London Mathematical Society},
  114(5):890--926, 2017.

\bibitem[BHS17c]{HHS_n}
Jason Behrstock, Mark~F Hagen, and Alessandro Sisto.
\newblock Quasiflats in hierarchically hyperbolic spaces.
\newblock {\em preprint arXiv:1704.04271}, 2017.

\bibitem[BKMM12]{BKMM:consistency}
Jason Behrstock, Bruce Kleiner, Yair Minsky, and Lee Mosher.
\newblock Geometry and rigidity of mapping class groups.
\newblock {\em Geom. Topol}, 16(2):781--888, 2012.

\bibitem[BM08]{BehrstockMinsky:dimension_rank}
Jason~A Behrstock and Yair~N Minsky.
\newblock Dimension and rank for mapping class groups.
\newblock {\em Annals of Mathematics}, 167:1055--1077, 2008.

\bibitem[BM11]{BehrstockMinsky:RD}
Jason~A. Behrstock and Yair~N. Minsky.
\newblock Centroids and the rapid decay property in mapping class groups.
\newblock {\em J. Lond. Math. Soc. (2)}, 84(3):765--784, 2011.

\bibitem[Bow13]{Bowditch:coarse_median}
Brian~H. Bowditch.
\newblock Coarse median spaces and groups.
\newblock {\em Pacific J. Math}, 261(1):53--93, 2013.

\bibitem[Bow14a]{Bowditch:embeddings}
Brian~H Bowditch.
\newblock Embedding median algebras in products of trees.
\newblock {\em Geometriae Dedicata}, 170(1):157--176, 2014.

\bibitem[Bow14b]{Bowditch:uniform}
Brian~H. Bowditch.
\newblock Uniform hyperbolicity of the curve graphs.
\newblock {\em Pacific J. Math}, 269(2):269--280, 2014.

\bibitem[Bow18]{Bowditch:large_scale}
Brian~H. Bowditch.
\newblock Large-scale rigidity properties of the mapping class groups.
\newblock {\em Pacific J. Math.}, 293(1):1--73, 2018.

\bibitem[BR18]{BerlaiRobbio}
Federico Berlai and Bruno Robbio.
\newblock A refined combination theorem for hierarchically hyperbolic groups.
\newblock {\em In preparation}, 2018.

\bibitem[Bri00]{Brinkmann:free_aut}
Peter Brinkmann.
\newblock Hyperbolic automorphisms of free groups.
\newblock {\em Geometric \& Functional Analysis}, 10(5):1071--1089, 2000.

\bibitem[Bri10]{Bridson:semisimple}
Martin~R Bridson.
\newblock Semisimple actions of mapping class groups on {CAT}(0) spaces.
\newblock {\em Geometry of Riemann surfaces}, 368:1--14, 2010.

\bibitem[BS05]{BuyaloSvetlov:homological}
S~Buyalo and P~Svetlov.
\newblock Topological and geometric properties of graph-manifolds.
\newblock {\em St. Petersburg Mathematical Journal}, 16(2):297--340, 2005.

\bibitem[BV]{BridsonVogtmann:dehn_2}
M~R Bridson and K~Vogtmann.
\newblock The {D}ehn functions of ${O}ut({F}_n)$ and ${A}ut({F}_n)$.
\newblock {\em arXiv:1011.1506}.

\bibitem[BV95]{BridsonVogtmann:dehn_1}
M~R Bridson and K~Vogtmann.
\newblock On the geometry of the automorphism group of a free group.
\newblock {\em Bulletin of the London Mathematical Society}, 27(6):544--552,
  1995.

\bibitem[CDH10]{ChatterjiDrutuHaglund:measured_walls}
Indira Chatterji, Cornelia Dru{\c{t}}u, and Fr{\'e}d{\'e}ric Haglund.
\newblock Kazhdan and {H}aagerup properties from the median viewpoint.
\newblock {\em Advances in Mathematics}, 225(2):882--921, 2010.

\bibitem[Che00]{Chepoi:cube_median}
Victor Chepoi.
\newblock Graphs of some {${\rm CAT}(0)$} complexes.
\newblock {\em Adv. in Appl. Math.}, 24(2):125--179, 2000.

\bibitem[CLM12]{ClayLeiningerMangahas:RAAGs}
Matt~T. Clay, Christopher~J. Leininger, and Johanna Mangahas.
\newblock The geometry of right-angled {A}rtin subgroups of mapping class
  groups.
\newblock {\em Groups Geom. Dyn.}, 6(2):249--278, 2012.

\bibitem[CRS15]{ClayRafiSchleimer:uniform}
Matt Clay, Kasra Rafi, and Saul Schleimer.
\newblock Uniform hyperbolicity of the curve graph via surgery sequences.
\newblock {\em Algebraic and Geometric Topology}, 14(6):3325--3344, 2015.

\bibitem[DHS17]{DurhamHagenSisto:HHS_III}
Matthew Durham, Mark Hagen, and Alessandro Sisto.
\newblock Boundaries and automorphisms of hierarchically hyperbolic spaces.
\newblock {\em Geom. Topol.}, 21(6):3659--3758, 2017.

\bibitem[Dil50]{Dilworth:poset}
R.~P. Dilworth.
\newblock A decomposition theorem for partially ordered sets.
\newblock {\em Ann. of Math}, 51:161--166, 1950.

\bibitem[DS05]{DrutuSapir:TreeGraded}
C.~Dru\c{t}u and M.~Sapir.
\newblock Tree-graded spaces and asymptotic cones of groups.
\newblock {\em Topology}, 44(5):959--1058, 2005.
\newblock With an appendix by Denis Osin and Mark Sapir.

\bibitem[DT15]{DurhamTaylor:stable}
Matthew~Gentry Durham and Samuel~J. Taylor.
\newblock Convex cocompactness and stability in mapping class groups.
\newblock {\em Algebr. Geom. Topol.}, 15(5):2839--2859, 2015.

\bibitem[FLS15]{FrigerioLafontSisto:rigidity}
Roberto Frigerio, Jean-Fran\c{c}ois Lafont, and Alessandro Sisto.
\newblock Rigidity of high dimensional graph manifolds.
\newblock {\em Ast\'erisque}, (372):xxi+177, 2015.

\bibitem[Ger92]{Gersten:dehn}
Stephen~M Gersten.
\newblock {\em Dehn functions and $\ell_1$-norms of finite presentations}.
\newblock Springer, 1992.

\bibitem[GHR03]{GerstenRiley:nilpotent}
Steve~M Gersten, Derek~F Holt, and Tim~R Riley.
\newblock Isoperimetric inequalities for nilpotent groups.
\newblock {\em Geometric \& Functional Analysis}, 13(4):795--814, 2003.

\bibitem[Hae16]{Haettel:symmetric}
Thomas Haettel.
\newblock Higher rank lattices are not coarse median.
\newblock {\em Algebr. Geom. Topol.}, 16(5):2895--2910, 2016.

\bibitem[Hag07]{Haglund:semisimple}
Fr{\'e}d{\'e}ric Haglund.
\newblock Isometries of {CAT}(0) cube complexes are semi-simple.
\newblock {\em {arXiv}:0705.3386}, 2007.

\bibitem[Hag14]{Hagen:quasi_arb}
Mark~F. Hagen.
\newblock Weak hyperbolicity of cube complexes and quasi-arboreal groups.
\newblock {\em J. Topol.}, 7(2):385--418, 2014.

\bibitem[HM13a]{HandelMosher:hyperbolicity}
Michael Handel and Lee Mosher.
\newblock The free splitting complex of a free group, {I}: hyperbolicity.
\newblock {\em Geom. Topol.}, 17(3):1581--1672, 2013.

\bibitem[HM13b]{HandelMosher:dehn}
Michael Handel and Lee Mosher.
\newblock Lipschitz retraction and distortion for subgroups of ${O}ut({F}_n)$.
\newblock {\em Geometry \& Topology}, 17(3):1535--1579, 2013.

\bibitem[HP15]{HagenPrzytycki:graph}
Mark~F. Hagen and Piotr Przytycki.
\newblock Cocompactly cubulated graph manifolds.
\newblock {\em Israel J. Math.}, 207(1):377--394, 2015.

\bibitem[HPW15]{HenselPrzytyckiWebb:unicorn}
Sebastian Hensel, Piotr Przytycki, and Richard C.~H. Webb.
\newblock 1-slim triangles and uniform hyperbolicity for arc graphs and curve
  graphs.
\newblock {\em J. Eur. Math. Soc. (JEMS)}, 17(4):755--762, 2015.

\bibitem[HS16]{HagenSusse}
Mark~F Hagen and Tim Susse.
\newblock On hierarchical hyperbolicity of cubical groups.
\newblock {\em preprint arXiv:1609.01313v2}, 2016.

\bibitem[HW15]{HagenWise:freebyz}
Mark~F. Hagen and Daniel~T. Wise.
\newblock Cubulating hyperbolic free-by-cyclic groups: the general case.
\newblock {\em Geometric and Functional Analysis}, 25(1):134--179, 2015.

\bibitem[KK14]{KimKoberda:curve_graph}
Sang-Hyun Kim and Thomas Koberda.
\newblock The geometry of the curve graph of a right-angled {A}rtin group.
\newblock {\em International Journal of Algebra and Computation},
  24(02):121--169, 2014.

\bibitem[KL96]{KapovichLeeb:actions}
Michael Kapovich and Bernhard Leeb.
\newblock Actions of discrete groups on nonpositively curved spaces.
\newblock {\em Math. Ann.}, 306(2):341--352, 1996.

\bibitem[KL97]{KapovichLeeb:haken}
M.~Kapovich and B.~Leeb.
\newblock Quasi-isometries preserve the geometric decomposition of {H}aken
  manifolds.
\newblock {\em Invent. Math.}, 128(2):393--416, 1997.

\bibitem[KL98]{KapovichLeeb:3manifolds}
M.~Kapovich and B.~Leeb.
\newblock {$3$}-manifold groups and nonpositive curvature.
\newblock {\em Geom. Funct. Anal.}, 8(5):841--852, 1998.

\bibitem[Liu13]{Liu:graph_man}
Yi~Liu.
\newblock Virtual cubulation of nonpositively curved graph manifolds.
\newblock {\em J. Topol.}, 6(4):793--822, 2013.

\bibitem[Man10]{Mangahas:UU}
Johanna Mangahas.
\newblock Uniform uniform exponential growth of subgroups of the mapping class
  group.
\newblock {\em Geometric and Functional Analysis}, 19(5):1468--1480, 2010.

\bibitem[MM99]{MasurMinsky:I}
Howard~A Masur and Yair~N Minsky.
\newblock Geometry of the complex of curves {I}: Hyperbolicity.
\newblock {\em Inventiones mathematicae}, 138(1):103--149, 1999.

\bibitem[MM00]{MasurMinsky:II}
Howard~A Masur and Yair~N Minsky.
\newblock Geometry of the complex of curves {II}: Hierarchical structure.
\newblock {\em Geometric and Functional Analysis}, 10(4):902--974, 2000.

\bibitem[MO14]{MinasyanOsin:acylindrical}
Ashot Minasyan and Denis Osin.
\newblock Acylindrical hyperbolicity of groups acting on trees.
\newblock {\em Mathematische Annalen}, pages 1--51, 2014.

\bibitem[Mos95]{Mosher:automatic}
L.~Mosher.
\newblock Mapping class groups are automatic.
\newblock {\em Ann. of Math.}, 142:303--384, 1995.

\bibitem[Mou17]{Mousley2}
Sarah~C Mousley.
\newblock Exotic limit sets of {T}eichm\"{u}ller geodesics in the hhs boundary.
\newblock {\em preprint arXiv:1704.08645}, 2017.

\bibitem[Mou18]{Mousley1}
Sarah~C. Mousley.
\newblock Nonexistence of boundary maps for some hierarchically hyperbolic
  spaces.
\newblock {\em Algebr. Geom. Topol.}, 18(1):409--439, 2018.

\bibitem[MR18]{MousleyRussell}
Sarah~C Mousley and Jacob Russell.
\newblock Hierarchically hyperbolic groups are determined by their morse
  boundaries.
\newblock {\em {preprint arXiv:1801.04867}}, 2018.

\bibitem[PS17]{PrzytyckiSisto:universe}
Piotr Przytycki and Alessandro Sisto.
\newblock A note on acylindrical hyperbolicity of mapping class groups.
\newblock In {\em Hyperbolic geometry and geometric group theory}, volume~73 of
  {\em Adv. Stud. Pure Math.}, pages 255--264. Math. Soc. Japan, Tokyo, 2017.

\bibitem[PW14]{PrzytyckiWise:graph_man}
Piotr Przytycki and Daniel~T. Wise.
\newblock Graph manifolds with boundary are virtually special.
\newblock {\em J. Topol.}, 7(2):419--435, 2014.

\bibitem[Rie01]{Rieffel:H2crossR}
E.~Rieffel.
\newblock Groups quasi-isometric to $\mathbf {H}^{2}\times\mathbf {R}$.
\newblock {\em J. London Math. Soc. (2)}, 64(1):44--60, 2001.

\bibitem[Sis11]{Sisto:unique_cones}
Alessandro Sisto.
\newblock 3-manifold groups have unique asymptotic cones.
\newblock {\em arXiv:1109.4674}, 2011.

\bibitem[Sis12]{Sisto:metric_rel_hyp}
Alessandro Sisto.
\newblock On metric relative hyperbolicity.
\newblock {\em arXiv:1210.8081}, 2012.

\bibitem[Sis13]{Sisto-distformrelhyp}
Alessandro Sisto.
\newblock Projections and relative hyperbolicity.
\newblock {\em Enseign. Math. (2)}, 59(1-2):165--181, 2013.

\bibitem[Spr17]{Spriano1}
Davide Spriano.
\newblock Hyperbolic {HHS I}: Factor systems and quasi-convex subgroups.
\newblock {\em {preprint arXiv:1711.10931}}, 2017.

\bibitem[Spr18]{Spriano2}
Davide Spriano.
\newblock Hyperbolic {HHS II}: Graphs of hierarchically hyperbolic groups.
\newblock {\em {preprint arXiv:1801.01850}}, 2018.

\bibitem[SS]{SabalkaSavchuk:projections}
L.~Sabalka and D.~Savchuk.
\newblock Submanifold projection.
\newblock \textsc{{arXiv}:1211.3111}.

\bibitem[Vok17]{Vokes:separating}
Kate~M Vokes.
\newblock Hierarchical hyperbolicity of the separating curve graph.
\newblock {\em arXiv preprint arXiv:1711.03080}, 2017.

\bibitem[Web15]{Webb:BGI}
Richard C.~H. Webb.
\newblock Uniform bounds for bounded geodesic image theorems.
\newblock {\em J. Reine Angew. Math.}, 709:219--228, 2015.

\bibitem[Wis14]{Wise:tubular}
Daniel Wise.
\newblock Cubular tubular groups.
\newblock {\em Transactions of the American Mathematical Society},
  366(10):5503--5521, 2014.

\bibitem[Woo16]{Woodhouse:tubular}
Daniel~J. Woodhouse.
\newblock Classifying finite dimensional cubulations of tubular groups.
\newblock {\em Michigan Math. J.}, 65(3):511--532, 2016.

\end{thebibliography}
\end{document}